\documentclass[final]{amsart}

\usepackage[T1]{fontenc}
\usepackage{amssymb,amsthm,amsmath}
\usepackage{tikz}
\usepackage[numbers]{natbib}
\usepackage{enumitem}
\usepackage{hyperref}
\usepackage{csquotes}
\usepackage[final]{microtype}
\usepackage{thmtools,thm-restate}
\usepackage{url,doi}

\setlist[enumerate]{label=\textup{(\roman{*})}}
\usetikzlibrary{positioning}

\hypersetup{hypertexnames=false}

\theoremstyle{plain}
\newtheorem{theorem}{Theorem}[section]
\newtheorem{proposition}[theorem]{Proposition}
\newtheorem{corollary}[theorem]{Corollary}
\newtheorem{lemma}[theorem]{Lemma}

\theoremstyle{remark}

\newtheorem{example}[theorem]{Example}
\newtheorem{definition}[theorem]{Definition}
\newtheorem{remark}[theorem]{Remark}

\newcommand{\T}{\mathbb{T}}
\newcommand{\R}{\mathbb{R}}
\newcommand{\Z}{\mathbb{Z}}
\newcommand{\N}{\mathbb{N}}
\newcommand{\Q}{\mathbb{Q}}
\newcommand{\A}{\mathcal{A}}
\newcommand{\M}{\mathcal{M}}

\newcommand{\cL}{\mathcal{L}}
\newcommand{\DD}{\mathcal{D}}
\newcommand{\RR}{\mathcal{R}}

\newcommand*{\skewprod}{\rtimes}
\newcommand*{\from}{\colon}

\newcommand*{\ab}{\operatorname{Ab}}
\newcommand*{\card}{\operatorname{Card}}
\newcommand*{\aut}{\operatorname{Aut}}
\renewcommand*{\hom}{\operatorname{Hom}}
\newcommand*{\img}{\operatorname{im}}
\newcommand*{\GL}{\operatorname{GL}}
\newcommand*{\lext}{\operatorname{L}}
\newcommand*{\rext}{\operatorname{R}}
\newcommand*{\ext}{\operatorname{E}}

\newcommand*{\comp}{\mathrm{c}}
\newcommand*{\lex}{\mathrm{lex}}
\renewcommand*{\bar}{\overline}

\numberwithin{equation}{section}

\title{Density of group languages in shift spaces}

\author[V. Berthé]{Valérie Berthé}
    \address{IRIF, Université Paris Cité, Paris, France}
    \email{berthe@irif.fr}

\author[H. Goulet-Ouellet]{Herman Goulet-Ouellet}
    \address{Département de mathématiques et de statistique, Université de Moncton, Moncton, Canada}
    \email{herman.goulet-ouellet@umoncton.ca}

\author[C.-F. Nyberg Brodda]{Carl-Fredrik Nyberg Brodda} 
    \address{School of Mathematics, Korea Institute for Advanced Study, Republic of Korea} 
    \email{cfnb@kias.re.kr}

\author[D. Perrin]{Dominique Perrin}
    \address{Laboratoire d’Informatique Gaspard Monge, Université Paris Est, Paris, France}
    \email{dominique.perrin@esiee.fr}

\author[K. Petersen]{Karl Petersen}
    \address{Mathematics department, University of North Carolina, USA}
    \email{petersen@math.unc.edu}



\date{June 2026}

\begin{document}

\begin{abstract}
  The density of a rational language can be understood as the frequency of some pattern in the shift space, for example a pattern like ``words with an even number of a given letter.''   We study the density of group languages, i.e. rational languages recognized by morphisms onto finite groups, inside shift spaces. We show that the density with respect to any given ergodic measure on a shift space exists for every group language, because it can be computed by using any ergodic lift of the given measure to a skew product between the shift space and the recognizing group. We then further study densities in shifts of finite type (with a suitable notion of irreducibility), and then in minimal shifts. In the latter case, we obtain a closed formula for the density under the condition that the aforementioned skew product has minimal closed invariant subsets which are ergodic under the product of the original measure and the uniform probability measure on the group. The formula is derived in part from a characterization of minimal closed invariant subsets for skew products between shifts and finite groups relying on notions of cocycles and coboundaries. In the case where the whole skew product is ergodic under the product measure, then the density is just the cardinality of the subset of the group which defines the language divided by the cardinality of the group. Moreover, we provide sufficient conditions for the skew product to have minimal closed invariant subsets that are ergodic under the product measure. Finally, we investigate the link between minimal closed invariant subsets, return words and bifix codes.
\end{abstract}

\maketitle

\markboth{\MakeUppercase{V. Berthé et al.}}{\MakeUppercase{Density of group languages in shift spaces}}

\tableofcontents

\section{Introduction}

The study of language density can be traced back to the work of Schützenberger in the 60s~\cite{Schutzenberger1965}, Berstel in the 70s~\cite{Berstel1972}, and Hansel and Perrin in the 80s~\cite{Hansel1983}. The idea also appears in Eilenberg's monograph~\cite{Eilenberg1974} and in the monograph by Berstel, Perrin and Reutenauer~\cite{BerstelPerrinReutenauer2009}. These earlier works, motivated mainly by automata theory and the theory of codes, focused on density with respect to \emph{Bernoulli measures} (see Section~\ref{sec:Bernoulli}).

In the present paper, we draw our motivation from symbolic dynamics and turn to \emph{ergodic measures on shift spaces}. Within this setting, the density of a given language can be understood as the frequency (with respect to some given ergodic measure) of some pattern in the shift space -- for example a pattern like \enquote{words with an even number of a given letter.} More precisely,  given  a shift $X$  on the alphabet $A$ endowed with a shift-invariant probability measure $\mu$,  and $L \subseteq A^*$  a language, the \emph{density $\delta_\mu(L)$ of $L$ under the measure $\mu$} is defined as the Ces\`{a}ro limit of $\mu(\{ u\in L : |u|=n\})$ as $n\to\infty$, whenever it exists.
  
We restrict our attention to \emph{group languages}, i.e.~languages recognized by morphisms onto finite groups: by fixing a morphism $\varphi\from A^*\to G$ onto a finite group $G$,  we consider a  language of the form  $L=\varphi^{-1}(K)$, $K\subseteq G$.
The density of the languages of this form can be understood using a \emph{skew product} (sometimes called a \emph{group extension}), which, briefly put, is a dynamical system built from a group $G$, together with  a dynamical system acting on $G$.  In our case, we  consider a first  dynamical system, a shift $X$ over $A$  which supports a measure $\mu$, 
which is then enriched through an action on a finite  group $G$ defined via a morphism $\varphi$. This yields a further dynamical system, the skew product, which is governed both by the dynamics of the original system $X$, and by the algebraic structure of the group $G$.  
 A precise definition of the skew products used in this paper may be found at the beginning of Section~\ref{s:skew}.
The key idea in this paper is that the density can be expressed in terms of limits of \emph{ergodic sums in a skew product} (cf.~Theorem~\ref{t:first-main}). This allows us to show that the density always exists and, under suitable conditions, to calculate it. This approach is strongly related with foundational work by Furstenberg, Veech, Schmidt, and Zimmer (among others) on ergodic properties of skew products~\cite{Furstenberg1961,Veech1969,Veech1975,Schmidt1977,Zimmer1976}. Veech's work in particular is concerned with a special case of the very same notion of density studied here, and it played a key role in guiding our investigation.

Given a shift space $X$ with an ergodic measure $\mu$, one of our main results relates the density $\delta_\mu(L)$ with ergodic measures on the skew product between $X$ and $G$, denoted $G\skewprod _{\varphi}X$, where the skewing function is the \emph{cocycle} determined by $\varphi$ (a skewing function which depends only on the symbol at the zero coordinate of elements of the shift). We say that a measure $\bar\mu$ on the skew product \emph{projects to} $\mu$ if $\bar\mu(G\times B) = \mu(B)$ for every measurable $B\subseteq X$. 

\begin{restatable}{theorem}{firstmain}\label{t:first-main}
    Let $X$ be a shift space on a finite alphabet $A$ with  a shift-invariant probability measure $\mu$ and let $\varphi\from A^*\to G$ be a morphism onto a finite group $G$. For every group language $L=\varphi^{-1}(K)$, where $K\subseteq G$, the density $\delta_\mu(L)$ exists.
    Moreover, if $\mu$ is ergodic, the density is given by 
    \begin{equation*}
        \delta_\mu(L)= \sum_{g \in G} \bar\mu(\{g\}\times X) \, \bar\mu(gK\times X)
    \end{equation*} 
    where $\bar\mu$ is any ergodic measure on the skew product $G \skewprod_{\varphi} X$ that projects to $\mu$.
\end{restatable}

The result is proved in Section~\ref{s:skew}. The proof uses the known fact that the skew product always admits an ergodic measure $\bar\mu$ that projects to $\mu$ (see Lemma~\ref{l:existence}). A natural candidate for the measure $\bar\mu$ is the product measure $\nu\times\mu$, where $\nu$ is the uniform probability measure on $G$. When $\nu\times\mu$ happens to be ergodic, then the above formula takes on the following very simple form (see Corollary~\ref{c:equidistribution}):
\begin{equation*}
   \delta_\mu(L)= \sum_{g \in G} (\nu\times\mu)(\{g\}\times X) \, (\nu\times\mu)(gK\times X)=\sum_{g\in G}(1/|G|)(|gK|/|G|) = |K|/|G|.
\end{equation*}
However this is not always the case: see for instance Examples~\ref{eg:periodicskewprod-2}, \ref{eg:not-si}, \ref{eg:thue-morse-2} and \ref{eg:unimodular}.
We also introduce in Section \ref{subsec:pointwise} a pointwise version of density (Definition \ref{def:pointwisedensity}). This leads to a further proof of the existence of the density of group languages, expressed as an integral over pointwise densities.

In the rest of the paper we apply Theorem~\ref{t:first-main} to two cases: when $\mu$ is a Markov measure on a shift of finite type, and when $\mu$ is an ergodic measure on a minimal shift space. In the former case, we present a characterization of ergodicity of the product measure $\nu\times\mu$ when $\mu$ is an \emph{$r$-step Markov} measure, $r\geq 1$, meaning that 
\[
\mu(x_m = a \mid x_{[0,m)} = w) = \mu(x_m=a \mid x_{[m-r,m)} = w_{[m-r,m)})
\] 
for  all words $w$ of length $m\geq r$ such that $\mu(x_m = a \mid x_{[0,m)} = w)\neq 0$. 
We define for this purpose a suitable notion of irreducibility, called \emph{$\varphi$-irreducibility} (Definition~\ref{d:phi-irreducible}), which leads to the following theorem.

\begin{restatable}{theorem}{secondmain}\label{t:second-main}   
    Let $\mu$ be an $r$-step Markov measure on $A^\Z$, $r\geq 1$, and let $X$ be the support of $\mu$.
    Let $\varphi\from A^*\to G$ be a morphism onto a finite group with uniform probability measure $\nu$. 
    If $X$ is $\varphi$-irreducible, then the product measure $\nu\times\mu$ is ergodic on the skew product $G\skewprod _{\varphi}X$. 
    
    In particular, for every $K\subseteq G$, the language $L = \varphi^{-1}(K)$ has density $\delta_\mu(L) = |K|/|G|$.
\end{restatable}

When $\mu$ is a Bernoulli measure, the above result also follows from the work of Hansel and Perrin~\cite[Theorem~3]{Hansel1983}, which uses a completely different approach based on the theory of codes. We provide a survey of this approach in Section~\ref{sec:Bernoulli}. 

In the case of minimal shift spaces, we specialize Theorem~\ref{t:first-main} into a formula which holds under the condition that $\nu\times\mu$ is ergodic on the minimal closed invariant subsets of $G\skewprod _{\varphi}X$. This generalizes the simple formula found when the product measure $\nu\times\mu$ is ergodic.  

Let us briefly present the notion, inspired by~\cite[Proposition~2.1]{LemanczykMentezen2002}, which lies at the core of this more general formula. Given a subgroup $H\leq G$, a map $\alpha$ from $X$ to the right coset space $H\backslash G$ is called a \emph{cobounding map} mod $H$ if it satisfies the following coboundary type equation, where $S$ denotes the shift map: 
\begin{equation*}
    \alpha(Sx) = \alpha(x)\varphi(x_0).
\end{equation*}

The term \emph{cobounding} appears for instance in a paper by Baggett et al.~\cite{Baggett1998}; the terms \emph{transfer function} and \emph{intertwining} have also been used with similar meanings in~\cite{Baggett1998,Baggett1986,Ramsay1976}.
As the name suggests, this notion is related with \emph{cocycles} and \emph{coboundaries} and more broadly to the long history of cohomology in ergodic theory. Besides~\cite{LemanczykMentezen2002} and the previously mentioned work of Furstenberg, Veech, Schmidt and Zimmer, we drew inspiration also from~\cite{Anzai1951,Conze1976}. The following theorem is our third main result; its statement uses a natural partial order on cobounding maps which is introduced in Section~\ref{ss:charac}.

\begin{restatable}{theorem}{thirdmain}\label{t:third-main}
    Let $X$ be a minimal shift space on $A$ with an ergodic measure $\mu$ and $\varphi\colon A^*\to G$ a morphism onto a finite group $G$ with uniform probability measure $\nu$. Suppose that $\nu\times\mu$ is ergodic on each of the minimal closed invariant subsets of $G\skewprod_\varphi X$. Then for every group language $L = \varphi^{-1}(K)$, where $K\subseteq G$, the density $\delta_\mu(L)$ is given by the following formula, where $\alpha\from X\to H\backslash G$ is any minimal cobounding map:
    \begin{equation*}
        \delta_\mu(L) = \frac{1}{|H|}\sum_{k\in K}\sum_{Hg\in H\backslash G}\mu(\alpha^{-1}(Hg))\mu(\alpha^{-1}(Hgk)).
    \end{equation*}
\end{restatable}

Behind this last result is a bijection between the minimal closed invariant subsets of $G\skewprod_{\varphi} X$ and the cobounding maps which are minimal under the aforementioned partial order (Proposition~\ref{p:minimal-cobounding}). This entails, among other things, that whenever $X$ is minimal, $\nu\times\mu$ can be ergodic only when $G\skewprod _{\varphi} X$ itself is minimal (Corollary~\ref{cor:minerg}). 
The formula in Theorem~\ref{t:third-main} is derived from Theorem~\ref{t:first-main}, where the role of the ergodic measure $\bar\mu$ is played by an appropriate rescaling of the measure $\nu\times\mu$, with $\nu$ being the uniform probability measure on $G$. This rescaling explains the factor $1/|H|$ in the formula.

Theorem~\ref{t:third-main} motivates further study of the structure of the skew products between minimal shifts and finite groups, which includes a characterization of minimality of such skew products in terms of \emph{return words} (Theorem~\ref{t:minimalskew}). We first give a combinatorial proof of this characterization using the theory of bifix codes, inspired by ideas from earlier works~\cite{BerstelDeFelicePerrinReutenauerRindone2012,BertheDeFeliceDolceLeroyPerrinReutenauerRindone2015c}. Later on, we show that the same conclusion can be reached using cobounding maps (Proposition~\ref{p:cobounding-return}).

Having in mind Theorem~\ref{t:third-main}, we also provide sufficient conditions for ergodicity of $\nu\times\mu$ on minimal closed invariant subsets of $G\skewprod_{\varphi}  X$. A first condition is given in Corollary~\ref{c:mod1} combined with Proposition~\ref{prop:minimal-ergodic}, and a second one, restricted to the case of shift spaces generated by \emph{primitive substitutions}, is given in Proposition~\ref{p:invertible-ergodic}. We deduce as a corollary of the second condition that \emph{substitutive dendric shifts} (definitions are recalled in Section~\ref{s:morphic}) have ergodic skew products with all finite groups, and consequently the density of $L= \varphi^{-1}(K)$ is simply $|K|/|G|$, $K\subseteq G$ (see Corollary~\ref{c:equidistribution}).

Here is an example of what these results tell us about specific systems. Consider the Fibonacci shift on the two-letter alphabet $\{a,b\}$, whose definition is recalled in Section~\ref{ss:fibo}. For each word $w$ in the language of the shift space, denote by $|w|_a$ the number of occurrences of $a$ in $w$. Then, we can apply Theorem~\ref{t:first-main} to deduce that the average probability over sufficiently long words $w$ that $|w|_a$ is congruent to $r$ mod $m$ is approximately $1/m$, for each $r=0,\ldots, m-1$. However the convergence of these probabilities holds only in the sense of Ces\`{a}ro mean. This specific example with $m=2$ is developed in detail in Section~\ref{ss:fibo}.

There have been previous results of a similar nature concerning equidistribution modulo~$m$. For instance, Veech~\cite{Veech1969,Veech1975} studied the parity of the number of visits to an interval by the orbit of a point under an irrational rotation (see Example \ref{ex:counting1}), and  Jager and Liardet~\cite{JL:88} studied the congruence classes of the matrices in $\GL(2,\Z)$ associated with continued fraction expansions of real numbers {see Example \ref{ex:convergents}}. We also revisit these results in Section~\ref{s:examples} in the particular case of Sturmian shifts. In each case, it is ergodicity of a relevant skew product that implies equidistribution among cosets. Our approach provides still more examples. These examples include the Thue--Morse shift (explored in Examples~\ref{eg:thue-morse-1} and \ref{eg:thue-morse-2}) and the case of substitutive dendric shifts (presented in Section~\ref{ss:dendric}), which includes  substitutive  Sturmian shifts and substitutive codings of interval exchanges.

Let us give a brief overview of the paper's structure. Section~\ref{s:symbolic-dynamics} gives some preliminaries on symbolic dynamics.  Section~\ref{s:skew} recalls the definition of skew products and gives an elementary proof of our first main result, Theorem~\ref{t:first-main}, which shows in particular that the density always exists. 
We also provide an alternative proof of the existence using a pointwise version of density.
We then survey the original approach to density under Bernoulli measures in Section~\ref{sec:Bernoulli}, relying on notions   from  algebraic theory of formal languages. The study of the density for shifts of finite type is  then handled in Theorem~\ref{t:second-main} together with a discussion on various notions of irreducibility. Section~\ref{s:bifix} contains the material on bifix codes and a characterization of minimal skew products in terms of return words.  
Our third  main result, Theorem~\ref{t:third-main}, is presented in Section~\ref{s:cobounding}, where cobounding maps are studied. The section also contains a simple sufficient condition for ergodicity of the product measure $\nu\times\mu$ on minimal closed invariant subsets. In Section~\ref{s:morphic} we take a closer look at shifts generated by primitive substitutions, and we consider particular examples of skew products based  on Sturmian shifts  in Section~\ref{s:examples}.

\section{Symbolic dynamics}\label{s:symbolic-dynamics}

Let $A$ be a finite alphabet. Let $\varepsilon$ stand for the empty word of the free monoid ${A}^*$ and $A^+ = A^*\setminus\{\varepsilon\}$. We denote by $A^\Z$ the set of two-sided infinite words on $A$. For any word~$w$ in the free monoid $A^*$ (endowed with concatenation), $|w|$ denotes the length of~$w$ and $|w|_a$ stands for the number of occurrences of the letter~$a$ in the word~$w$. We start indexing finite words with 0, so that a word $w\in A^*$ has the form $w = w_0w_1\ldots w_{n-1}$, where $n=|w|$. Given $0\leq i\leq j< n$, we let
\begin{equation*}
    w_{[i,j)} = w_iw_{i+1}\ldots w_j,
\end{equation*}
and we extend a similar notation for infinite words. 

A~\emph{factor} of a (finite or infinite) word $w$ is defined as a finite concatenation of consecutive letters occurring in~$w$, i.e. a word $u$ is a factor of $w$ if there exist indices $i\leq j$ such that $u = w_{[i,j)}$. If $w$ is a finite word, then $u$ is a factor of $w$ precisely when there are words $p$ and $s$ such that $w = pus$. When $p = \varepsilon$ (resp., $s = \varepsilon$), we say that $u$ is a \emph{prefix} (resp., \emph{suffix}) of $w$. 

An infinite word $x = (x_n)_{n\in\Z}$ is \emph{uniformly recurrent} if every word occurring in $x$ occurs in an infinite number of positions with bounded gaps; in other words, for every factor $w$ of $x$, there exists a positive integer $m$ such that for every $n$, $w$ is a factor of $x_{[n,n+m)}$. 

We view closed shift-invariant sets of two-sided infinite words as dynamical systems under the map $S$, called the \emph{shift map}, defined by 
\begin{equation*}
    S\left( (x_n)_{n\in\Z} \right) = (x_{n+1})_{n\in\Z}. 
\end{equation*}

A \emph{shift space} (also shortened to \emph{shift}) is a pair $(X,S)$ where $X$ is a  topologically closed and shift-invariant subset of $A^{\Z}$ for some finite alphabet $A$. We usually shorten $(X,S)$ as $X$ when we refer to the system $(X,S)$. The \emph{language} of $X$ is defined as the set $\cL(X)$ of factors of elements of $X$
\begin{equation*}
    \cL(X) = \{ x_{[i,j)} \mid x\in X,\ i, j\in\Z,\ i\leq j \}.
\end{equation*}
When a shift $X$ is said to be defined on the alphabet $A$, we assume that $A\subseteq\cL(X)$. 

A shift space is said to be \emph{minimal} if it admits no proper non-empty closed and shift-invariant subset; equivalently the $S$-orbit of every element of $X$ is dense. Note that a shift space $X$ is minimal if and only if every infinite word $x \in X$ is uniformly recurrent. On the other hand, a shift space is called \emph{irreducible} if \emph{there exists} an element $x\in X$ with dense $S$-orbit. This is equivalent to the following property of $\cL(X)$: for every $u,v\in\cL(X)$, there exists $w\in A^*$ such that $uwv\in \cL(X)$. 

Let $X$ be a shift space on $A$ and fix $w\in\cL(X)$. We denote by $\RR_X(w)$ the set of (right) \emph{return words} to $w$. It is, by definition, the set of words $r$ such that $rw$ is in $\cL(X)$ and has exactly two factors equal to $w$, one as a prefix and the other one as a suffix; that is,
\begin{equation*}
    \RR_X(w) = \{ r\in A^* \mid rw\in\cL(X)\cap wA^*\setminus A^+wA^+\}. 
\end{equation*}

Let $X\subseteq A^\Z$ be a shift space equipped with a Borel probability measure $\mu$. For $w\in\cL(X)$, we denote
\begin{equation*}
    [w]_X=\{x\in X\mid x_{[0,n)}=w\}
\end{equation*} 
the right cylinder defined by $w$ (\emph{right} refers here to the fact this definition involves  only  non-negative indices). When $X$ is clear from context, we often drop the subscript $X$ and write simply $[w]$. 
Note that we have $\mu([\varepsilon])=1$ and
\begin{equation*}
  \sum_{a\in A}\mu([wa])=\mu([w]).
\end{equation*}
  
The measure $\mu$ is \emph{invariant} if $\mu(S^{-1}U)=\mu(U)$ for every Borel set $U\subseteq X$. Note that $\mu$ is invariant if and only if, for every $w\in\cL(X)$,
\begin{equation*}
    \sum_{a\in A}\mu([aw])=\mu([w]).
\end{equation*}

Recall that an invariant measure $\mu$ is \emph{ergodic} if every Borel set $U$ which is invariant (i.e. $S^{-1}U=U$) has measure $0$ or $1$. A well-known equivalent condition is that $\mu$ is ergodic if and only if
\begin{equation}\label{eq:ergodic}
    \lim_{n\to\infty}\frac{1}{n}\sum_{i=0}^{n-1}\mu(U\cap S^{-i}V)=\mu(U)\mu(V),
\end{equation}
for every pair $U,V$ of Borel sets~\cite[p.~56, exercise 4(a)]{Petersen1983}. If instead of converging on average (i.e. in Ces\`{a}ro's sense) the sequence $(\mu(U\cap S^{-n}V))_{n\in\N}$ converges directly to $\mu(U)\mu(V)$ for all pairs of Borel sets $U,V$, then the measure is said to be \emph{mixing}.
  
The shift $X$ is said to be \emph{uniquely ergodic} if there is a unique invariant probability measure on $X$, in which case this unique measure is necessarily ergodic. 
An important class of uniquely ergodic shifts is given by \emph{primitive substitution shifts}, whose definition may be recalled in \cite[Section~1.4]{DurandPerrin2021} for instance.
By a theorem of Michel, every primitive substitution shift is uniquely ergodic~\cite{Michel1974}. Other important sufficient conditions for unique ergodicity of shift spaces are due to Boshernitzan~\cite{Boshernitzan1984,Boshernitzan1992}. More details on such results may be found in~\cite{Queffelec2010,DurandPerrin2021}.

\begin{example}\label{eg:thue-morse-1}
  The Thue--Morse shift $X=X(\sigma)$ with $\sigma\colon a\mapsto ab,b\mapsto ba$ is uniquely ergodic by Michel's theorem. Its unique ergodic measure $\mu$ is depicted in Figure~\ref{f:proba-morse}.
    \begin{figure}[hbt]
        \centering
        \begin{tikzpicture}[xscale=2,yscale=.18]
            \node (1)    at (0,0)   {$\tfrac{1}{1}$};
            \node (a)    at (1,6)   {$\tfrac{1}{2}$};
            \node (b)    at (1,-6)  {$\tfrac{1}{2}$};
            \node (aa)   at (2,12)  {$\tfrac{1}{6}$};
            \node (ab)   at (2,4)   {$\tfrac{1}{3}$};
            \node (ba)   at (2,-4)  {$\tfrac{1}{3}$};
            \node (bb)   at (2,-12) {$\tfrac{1}{6}$};
            \node (aab)  at (3,15)  {$\tfrac{1}{6}$};
            \node (aba)  at (3,9)   {$\tfrac{1}{6}$};
            \node (abb)  at (3,3)   {$\tfrac{1}{6}$};
            \node (baa)  at (3,-3)  {$\tfrac{1}{6}$};
            \node (bab)  at (3,-9)  {$\tfrac{1}{6}$};
            \node (bba)  at (3,-15) {$\tfrac{1}{6}$};
            \node (aaba) at (4,18)  {$\tfrac{1}{12}$};
            \node (aabb) at (4,14)  {$\tfrac{1}{12}$};
            \node (abaa) at (4,10)  {$\tfrac{1}{12}$};
            \node (abab) at (4,6)   {$\tfrac{1}{12}$};
            \node (abba) at (4,2)   {$\tfrac{1}{6}$};
            \node (baab) at (4,-2)  {$\tfrac{1}{6}$};
            \node (baba) at (4,-6)  {$\tfrac{1}{12}$};
            \node (babb) at (4,-10) {$\tfrac{1}{12}$};
            \node (bbaa) at (4,-14) {$\tfrac{1}{12}$};
            \node (bbab) at (4,-18) {$\tfrac{1}{12}$};

            \draw[above,node font=\small] (1) edge node{$a$} (a);
            \draw[below,node font=\small] (1) edge node{$b$} (b);
            \draw[above,node font=\small] (a) edge node{$a$} (aa);
            \draw[below,node font=\small] (a) edge node{$b$} (ab);
            \draw[above,node font=\small] (b) edge node{$a$} (ba);
            \draw[below,node font=\small] (b) edge node{$b$} (bb);
            \draw[above,node font=\small] (aa) edge node{$b$} (aab);
            \draw[above,node font=\small] (ab) edge node{$a$} (aba);
            \draw[below,node font=\small] (ab) edge node{$b$} (abb);
            \draw[above,node font=\small] (ba) edge node{$a$} (baa);
            \draw[below,node font=\small] (ba) edge node{$b$} (bab);
            \draw[below,node font=\small] (bb) edge node{$a$} (bba);
            \draw[above,node font=\small] (aab) edge node{$a$} (aaba);
            \draw[below,node font=\small] (aab) edge node{$b$} (aabb);
            \draw[above,node font=\small] (aba) edge node{$a$} (abaa);
            \draw[below,node font=\small] (aba) edge node{$b$} (abab);
            \draw[below,node font=\small] (abb) edge node{$a$} (abba);
            \draw[above,node font=\small] (baa) edge node{$b$} (baab);
            \draw[above,node font=\small] (bab) edge node{$a$} (baba);
            \draw[below,node font=\small] (bab) edge node{$b$} (babb);
            \draw[above,node font=\small] (bba) edge node{$a$} (bbaa);
            \draw[below,node font=\small] (bba) edge node{$b$} (bbab);
        \end{tikzpicture}
        \caption{The invariant probability measure on the Thue--Morse shift}\label{f:proba-morse}
    \end{figure}
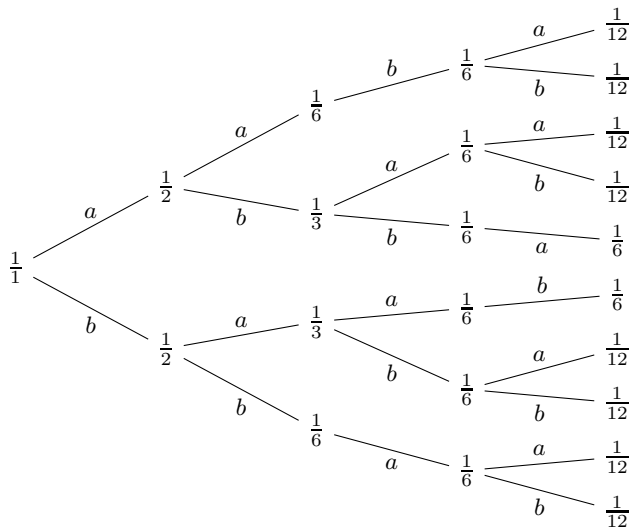
\end{example}

The aim of this paper is to study the density of languages under a given measure $\mu$, defined as follows.
\begin{definition}\label{def:density}
Let $L$ be a language on  an alphabet $A$ and $\mu$ a measure on $A^{\Z}$. The \emph{density} of $L$ under the measure $\mu$ is the following limit whenever it exists:
\begin{equation*}
    \delta_\mu(L)=\lim_{n\to\infty}\frac{1}{n}\sum_{i=0}^{n-1}\mu([L\cap A^i]).
\end{equation*}
\end{definition}

In other words, $\delta_\mu(L)$ is the Ces\`{a}ro limit of $\mu([L\cap A^n])$ as $n\to\infty$. Since $\mu(w)=0$ when $w\notin\cL(X)$, we have of course that $\delta_\mu(L)=\delta_\mu(L\cap\cL(X))$ and $\delta_\mu(\cL(X))=1$. Moreover, it follows from the definition of density that, whenever $L$ and $L'$ are disjoint,
\begin{equation*}
    \delta_\mu(L\cup L') = \delta_\mu(L)+\delta_\mu(L').
\end{equation*}
There is also a dual rule for the intersection: when $L\cup L' = A^*$, the density of the intersection is given by $\delta_\mu(L\cap L') = \delta_\mu(L)+\delta_\mu(L')-1$. 

Note however that the density of a general intersection $L \cap L'$ might not exist, even when the density of both $L$ and $L'$ exist.
For instance, if we consider the two languages
\[
    L = \{ w\in A^* : |w|\equiv 1\bmod 2\},\qquad L' = \{ w\in A^* : |w|\equiv\lfloor\log_2(|w|)\rfloor \bmod 2\},
\]
over some fixed finite alphabet $A$, then it is not hard to see that $\delta_\mu(L) = \delta_\mu(L') = \frac12$, no matter the probability measure $\mu$. 
However, the sequence of Ces\`aro sums $s_n = \frac1n\sum_{i=1}^n\mu([L\cap L'\cap A^{i}])$ has two subsequences $(s_{2^{2n}})_{n\in\N}$ and $(s_{2^{2n+1}})_{n\in\N}$ which converge to $1/3$ and $1/6$ respectively. Therefore the sequence of Ces\`aro sums does not converge, and $L\cap L'$ does not have a density.

An easy observation is that a finite language has zero density for every measure. This observation can be generalized to a larger class of languages called \emph{thin languages}, see Section~\ref{sec:Bernoulli}.

Next we introduce a language that will serve as a guiding example throughout the paper.
\begin{example}\label{ex:Lrational}
    Consider  integers $m,d,r$  with $m,d\geq 2$ and $0\leq r\leq m-1$. Define the language  $$L_r:= \{ w\in A^* : |w|_1 \equiv r \bmod m\} $$  on the alphabet  $A=\{1,2,\dots,d\}$, as the set of  finite words having a number of $1$'s congruent  to $r$ modulo $m$. This is a group language, as  defined in  the next section (see also Example \ref{ex:counting0}).  One  of our motivations for the present study is to establish the existence of the  density of $\delta_ \mu(L_r)$,  and  obtain conditions for \emph{equidistribution}, that  is, $\delta_ \mu(L_r)=1/m$.
\end{example}

\section{Group languages and skew products}
\label{s:skew}

We first recall basic definitions on group languages and skew products in Section~\ref{ss:def}. We then prove Theorem~\ref{t:first-main} in Section~\ref{ss:formula}. The proof uses any ergodic lift to the skew product of the given ergodic measure on the shift space. We also consider in Corollary~\ref{c:equidistribution} the special case where the product between the original measure and the uniform probability measure on the group is ergodic. If the shift space is uniquely ergodic, then this is even equivalent to unique ergodicity of the skew product, cf.~Proposition~\ref{prop:ue}. 
We then consider in Section~\ref{subsec:pointwise}  a  pointwise version of density.
Lastly, we focus on the case of the Fibonacci shift in Section~\ref{ss:fibo}, where we prove that the density considered as a Ces\`aro mean converges, whereas 
the sequence of measures  $(\mu([L\cap A^n]))_{n\in\N}$ does not converge in the classical sense.

\subsection{First definitions}\label{ss:def}

A \emph{rational language}  is a language recognized by a finite automaton.
Equivalently, a language is rational if and only if it is of the form $\varphi^{-1}(N)$ where $\varphi\from A^*\to M$ is a morphism onto a finite monoid $M$, and $N\subseteq M$ is a subset.

A \emph{group language} is a set of the form $L = \varphi^{-1}(K)$ where $\varphi\colon A^*\to G$ is a morphism onto a finite group  $G$ and $K\subseteq G$. Note that such languages are in particular rational, being recognized by finite groups.  For more on the topic, see~\cite{BerstelPerrinReutenauer2009}.

Let $X$ be a minimal shift space equipped with an invariant measure $\mu$. We consider the \emph{skew product} of $G$ and $X$ with respect to $\varphi$, denoted $G\skewprod_\varphi X$: it is the dynamical system $G\skewprod_\varphi X=(G\times X,T_\varphi)$, where $T_\varphi$ is the transformation defined by
\begin{equation} \label{eq:Tphi}
  T_\varphi(g,x)=(g\varphi(x_0),Sx).
\end{equation}

More generally, $T_\varphi$ satisfies, for every $n\in\Z$,
\begin{equation*}
    T_\varphi^n(g,x)=(g\varphi^{(n)}(x),S^nx),
\end{equation*}
where $\varphi^{(n)}$ is defined by
\begin{equation}\label{eq:fn}
    \varphi^{(n)}(x) = 
    \begin{cases}
        \varphi(x_{[0,n)}) & \text{if}\ n\geq 0;\\
        \varphi(x_{[n,0)})^{-1} & \text{if}\ n<0.
    \end{cases}
\end{equation}

The map $(n,x)\mapsto \varphi^{(n)}(x)$ is the \emph{cocycle} defined by $\varphi$. When the morphism $\varphi$ is clear from context, we may simply write $T$ and $G\skewprod X$. Skew products constitute one of the basic extensions of dynamical systems; see e.g.~\cite{CFS:82,Petersen1983}.

\begin{lemma}\label{l:isomskewprod}
    The skew product $G\skewprod_{\varphi}  X$ is topologically conjugate to a shift space on $G\times A$ via the map $\Psi\colon G\skewprod_{\varphi} X\to (G\times A)^\Z$ defined by 
    \begin{equation*}
        \Psi(g,x)_n=(g\varphi^{(n)}(x),x_n),\quad n\in\Z.
    \end{equation*}
\end{lemma}

Note that the lemma does not say that $\Psi$ is surjective on $(G\times A)^\Z$, but simply that $\Psi$ is a topological conjugacy from $G\skewprod_{\varphi}  X$ to the subshift $\Psi(G\skewprod_{\varphi}  X)$.

\begin{proof}
    The map $\Psi$ is continuous and injective, hence it is an homeomorphism onto its image. Moreover it follows from the definitions that $\Psi\circ T_{\varphi} = S\circ\Psi$.
\end{proof}

\begin{example}\label{ex:counting0}
 Let $L_r= \{ w\in A^* : |w|_1 \equiv r \bmod m\}$ be the language from Example \ref{ex:Lrational}.  
Let $X$ be a  shift on $\{1, \dots, d\}$.  
We   consider the  skew product  $\Z/m\Z\skewprod_{\varphi} X$, where
 $ \Z/m\Z$ is the $m$-element additive group,  and $\varphi$ is the morphism $\{1,\dots,d\}^*\to\Z/m\Z$ given by $1\mapsto 1$ and $i \mapsto 0$ for $2\leq i\leq d$.  The language $L_r$  is indeed  
 a group language, since  $L_r= \varphi^{-1}(\{r\})$.  
 In particular, one has $\varphi^{(n)}(x)= |x_0 \ldots x_{n-1}|_1$ modulo $m$ for $n \geq 0$. Understanding the equidistribution properties of  the sequence  $(\varphi^{(n)}(x))_n$ in $G = \Z/m\Z$ allows one to see how often factors have a given congruence modulo $m$ of occurrences of the letter $1$; see Section \ref{ss:fibo} as an illustration.
\end{example}

\begin{example}
    \label{eg:periodicskewprod-1}
    Let $X$ be the three-element shift  defined as the  finite orbit of the periodic word $x=(abc)^\infty$, so $X=\{x,y,z\}$ with $y=Sx$, $z=Sy$. Let $\varphi\from A^*\to \Z/2\Z$ be the morphism
    \begin{equation*}
        \varphi(a) = \varphi(b) = 1,\quad \varphi(c)=0.
    \end{equation*}
    The skew product $\Z/2\Z\skewprod_{\varphi} X$ has six elements. Viewed as a shift on the alphabet $\Z/2\Z\times A$, it is the disjoint union of the orbits of two periodic words,
    \begin{equation*}
        \Psi(0,x) = ((0,a)(1,b)(0,c))^\infty\quad\text{and}\quad \Psi(1,x) = ((1,a)(0,b)(1,c))^\infty.
    \end{equation*}
  Despite its apparent simplicity, this   example  proves to be more profound    than it seems, by   offering  a  clear   way to illustrate  ergodicity (see Example \ref{eg:periodicskewprod-2}
  and \ref{eg:periodicskewprod-3}), 
irreducibility   (see Example \ref{eg:periodicskewprod-2bis} and   \ref{eg:3element}), as well as minimality properties (see Example \ref{eg:periodicskewprod-2bis}). It can also be generalized as  detailed in Example \ref{ex:3pointn}.
   The shift $X$ and the skew product $\Z/2\Z\skewprod _{\varphi} X$ are depicted in Figure~\ref{f:periodicskewprod}.
    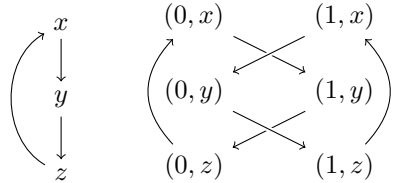
\begin{figure}[hbt]
        \centering
        \begin{tikzpicture}
            \node (abc) at (0,0)  {$x$};
            \node (bca) at (0,-1) {$y$};
            \node (cab) at (0,-2) {$z$};

            \draw[->] (abc) to (bca) ; 
            \draw[->] (bca) to (cab) ; 
            \draw[->] (cab) to [in=210,out=150] (abc) ; 
        \end{tikzpicture}
        \qquad
        \begin{tikzpicture}[xscale=2]
            \node(0x) at (0,0)   {$(0,x)$};
            \node(0y) at (0,-1)  {$(0,y)$};
            \node(0z) at (0,-2)  {$(0,z)$};
            \node(1x) at (1,0)   {$(1,x)$};
            \node(1y) at (1,-1)  {$(1,y)$};
            \node(1z) at (1,-2)  {$(1,z)$};

            \draw[->] (0x) to (1y) ; 
            \draw[->] (1y) to (0z) ; 
            \draw[->] (0z) to [out=120, in=240] (0x) ; 
            \draw[line width=2pt,white,->] (1x) to (0y) ; 
            \draw[->] (1x) to (0y) ; 
            \draw[line width=2pt,white,->] (0y) to (1z) ; 
            \draw[->] (0y) to (1z) ; 
            \draw[->] (1z) to [out=60,in=-60] (1x) ; 
        \end{tikzpicture} 
        \caption{The finite shift $X$ generated by $x=(abc)^\infty$ and its skew product with $\Z/2\Z$ from Example~\ref{eg:periodicskewprod-1}}
        \label{f:periodicskewprod}
    \end{figure}
\end{example}

\subsection{A formula for the density}\label{ss:formula}

The aim of this section is to prove our first main result, Theorem~\ref{t:first-main}.
The proof uses the existence of an ergodic measure $\bar\mu$ on the skew product which projects to $\mu$ (see Lemma \ref{l:existence} below). This known fact follows from classical results in ergodic theory (see for example~\cite{Kryloff1937,book/Denker1976,Oxtoby1952,book/Furstenberg1981}); we provide a proof for the sake of completeness.
Let us recall briefly that a \emph{factor map} between two dynamical systems $(Z,R)$ and $(Y,T)$ is a map $\pi\from Z\to Y$ which is onto and satisfies $\pi\circ R = T\circ\pi$.

\begin{lemma}\label{l:existence}
    Let $(Z,R)$ and $(Y,T)$ be two compact dynamical systems with a factor map $\pi\from Z\to Y$. Then, for each ergodic measure $\mu$ on $(Y,T)$, there is an ergodic measure $\bar\mu$ on $(Z,R)$ such that $\bar\mu\circ\pi^{-1}=\mu$. 
\end{lemma}

\begin{proof}
    Let $\M$ be the set of invariant measures on $(Z,R)$ and $\M_\mu$ be the subset of those $\zeta\in\M$ such that $\zeta\circ\pi^{-1}=\mu$. First we show that $\M_\mu\neq\emptyset$. Start by choosing a point $y\in Y$ satisfying
    \begin{equation*}
        \lim_{n\to\infty} \frac{1}{n}\sum_{i=0}^{n-1}f(T^iy) = \int_Y f\, d\mu
    \end{equation*}
    for every continuous function on $X$; the point $y$ is called a \emph{generic point} for $\mu$. It is well-known that ergodicity of $\mu$ implies that $\mu$-almost all points of $Y$ have this property (see for instance \cite[Proposition~3.7]{book/Furstenberg1981}), thus such a point $y$ exists. Take a preimage $z\in\pi^{-1}(y)$ and let $\zeta_i$ be the point mass measure concentrated on $R^iz$. Consider the sequence of Ces\`{a}ro averages $\frac{1}{n}\sum_{i=0}^{n-1}\zeta_i$. By the Banach--Alaoglu theorem, we may choose a subsequence which converges for the weak-$*$ topology in the space of all probability measures on $(Z,R)$, say:
    \begin{equation*}
        \zeta = \lim_{k\to\infty} \frac{1}{n_k}\sum_{i=0}^{n_k-1}\zeta_i.
    \end{equation*}
    Observe that $\zeta_i\circ R^{-1} = \zeta_{i+1}$, and thus $\zeta - \zeta\circ R^{-1} = \lim_{k\to\infty}\frac{1}{n_k}(\zeta_0 - \zeta_{n_k}+1)=0$, i.e. $\zeta\in\M$. Moreover for every measurable subset $B\subseteq Y$
    \begin{align*}
        \zeta\circ\pi^{-1}(B) = \lim_{k\to\infty}\frac{1}{n_k}\sum_{i=0}^{n_k-1}\zeta_i(\pi^{-1}(B)) = \lim_{k\to\infty}\frac{1}{n_k}\sum_{i=0}^{n_k-1} 1_B(T^iy) = \mu(B).
    \end{align*}
    Thus, $\zeta\in\mathcal{M}_\mu$.

    Observe that $\M_\mu$ is a closed convex subspace of $\M$, and since it is non-empty we can apply the Krein--Milman theorem to conclude that it contains an extreme point. Since the ergodic measures on $(Z,R)$ are precisely the extreme points of $\M$, it remains only to show that any extreme point $\bar\mu\in\M_\mu$ is also extreme in $\M$. 
    Assume that $\bar\mu = s\zeta' + (1-s)\zeta''$ where $0<s<1$ and $\zeta'$, $\zeta''\in\mathcal{M}$.  Let us prove that   $\bar\mu$ cannot be expressed as a non-trivial convex combination of \emph{distinct} elements of $\mathcal{M}$. Then we have 
    \[
    \mu = \bar\mu\circ\pi^{-1} = s\zeta'\circ\pi^{-1} + (1-s)\zeta''\circ\pi^{-1}.
    \] 
    Since $\mu$, being ergodic, is an extreme point in the space of invariant measures of $(Y,T)$, we conclude that $\zeta'\circ\pi^{-1} = \zeta''\circ\pi^{-1} = \mu$. Therefore $\zeta',\zeta''\in\M_\mu$, and  since $\bar\mu$ is an extreme point in $\M_\mu$ it follows that $\zeta'=\zeta''$.
    Thus, $\bar\mu$ is an extreme point in $\mathcal{M}$.
\end{proof}

Let us recall now our first main result, which expresses the density of a group language in a shift space in terms of an ergodic measure on the skew product $G\skewprod X$. This will later be specialized for shifts of finite type (Section~\ref{s:SFT}),  and then for minimal shifts (Section~\ref{s:cobounding}).
\firstmain*

\begin{proof}
    We first assume that  $\mu$ is ergodic.
    
    For $C\subseteq G$, let $U_{C} = C\times X$; in case $C = \{g\}$, $g\in G$, we write simply $U_g$. 
    Then we find
    \begin{equation*}
        \{g\}\times [L\cap A^i]_X=(\{g\}\times X)\cap T_{\varphi}^{-i}(\{gK\}\times X)=U_g\cap T_{\varphi}^{-i}U_{gK}.
    \end{equation*}
    Since the projection $G\times X\to X$ is a factor map, by Lemma~\ref{l:existence}, we may take an ergodic measure $\bar\mu$ on $G\skewprod _{\varphi} X$ which projects to $\mu$. 
    In terms of $\bar\mu$, we have
    \[
        \mu([L\cap A^i]) = \bar\mu(G\times [L\cap A^i])
        = \sum_{g\in G} \bar\mu(\{g\}\cap [L\cap A^i]) 
        = \sum_{g\in G} \bar\mu(U_g\cap T_{\varphi}^{-i}U_{gK}).
    \]
    Since $\bar\mu$ is ergodic, we can use \eqref{eq:ergodic} to obtain
    \begin{align*}
        \delta_\mu(L) &= \lim_{n\to\infty}\frac{1}{n}\sum_{i=0}^{n-1}\mu([L\cap A^i]) \\
        &= \sum_{g\in G}\lim_{n\to\infty}\frac{1}{n}\sum_{i=0}^{n-1} \bar\mu (U_g\cap T_{\varphi}^{-i}U_{gK})\\
        &= \sum_{g\in G} \bar\mu (U_g)\,\bar\mu(U_{gK}),
    \end{align*}

    We now only assume that $\mu$ is a shift-invariant probability measure on $X$.
    Let  $\mathcal{E} = \mathcal{E}(X,S)$ be the 
    set of ergodic invariant probability measures on $(X,S)$.  By e.g. \cite[Chapter 12]{book/Phelps2001}, there exists a Borel probability measure $\tau$ on $\mathcal E$ such that
    \[
        \mu=\int_{\mathcal E}\nu \,d\tau(\nu).
    \]
    Then by the dominated convergence theorem
      \[
      \delta_{\mu}(L)
      = \lim_{n\to\infty}\frac1n\sum_{i=0}^{n-1}\mu([L\cap A^i])
      =\lim_{n\to\infty}\frac1n\sum_{i=0}^{n-1}\int_{\mathcal E}\nu([L\cap A^i])\,d\tau(\nu)
      =\int_{\mathcal E} \delta_{\nu} (L)  \,d\tau(\nu),
      \]
    where the densities $\delta_\nu(L)$ exist by the first part of the proof.
    
\end{proof}

\begin{remark}\label{r:strong-convergence}
    Under the stronger assumption that $\bar\mu$ is mixing, then the density converges in a stronger sense, i.e.
    \begin{equation*}
        \delta_\mu(L) = \lim_{n\to\infty}\mu([L\cap A^n]).
    \end{equation*}
    See for instance Example~\ref{ex:full}.
\end{remark}

In case $\nu\times\mu$ is ergodic on the skew product, the above theorem yields the following simple formula for the density.
	
\begin{corollary}\label{c:equidistribution}
    Let $X$ be a shift space on a finite alphabet $A$ with an ergodic measure $\mu$ and let $\varphi\from A^*\to G$ be a morphism onto a finite group $G$ with uniform probability measure $\nu$.     If the product measure $\nu \times \mu$ is ergodic on $G \times X$, then for every group language $L=\varphi^{-1}(K)$ with $K\subseteq G$,  $\delta_\mu(L) = |K|/|G|$.
\end{corollary}

\begin{proof}
    Since $(\nu \times \mu)(U_C)=|C|/|G|$ for all $C\subseteq G$, the above theorem with $\bar\mu=\nu\times\mu$ yields $\delta_\mu(L) = \sum_{g\in G}|K|/|G|^2=|K|/|G|$.
\end{proof}

There are however many examples where $\nu\times\mu$ is not ergodic on the skew product, such as the simple one below. 

\begin{example}\label{eg:periodicskewprod-2}
    Let $X$   be the three-element shift,  $\varphi$ be as in Example~\ref{eg:periodicskewprod-1}, and $L=\varphi^{-1}(0)$. The shift $X$ is uniquely ergodic, with ergodic measure $\mu$ given by the uniform probability measure. Moreover,
    \begin{equation*}
        L\cap\cL(X)=(abc)^*\{\varepsilon,ab\}\cup(bca)^*\{\varepsilon\}\cup (cab)^*\{\varepsilon,c\},
    \end{equation*}
    and thus,
    \begin{equation*}
        \mu([L\cap A^i])=
        \begin{cases}
            1 & \text{if $i\equiv 0\bmod 3$}\\
            1/3 & \text{otherwise.}
        \end{cases}
    \end{equation*}

    It follows that $\delta_\mu(L)$ is given by $\delta_\mu(L)=\left(1+1/3+1/3\right)/3=5/9$. In contrast, if $\nu\times\mu$ would be ergodic, then we should find instead $1/2$. The fact that $\nu\times\mu$ is not ergodic can also be observed directly by noting that $\Z/2\Z\skewprod_{\varphi}  X$ has two invariant subsets, each of $\nu\times\mu$-measure~$1/2$ (which are in fact minimal invariant subsets). 

    We can also compute $\delta_\mu(L)=5/9$ using the formula from Theorem~\ref{t:first-main}. Indeed, there are two ergodic measures $\bar\mu_1$ and $\bar\mu_2$ which project to $\mu$, each supported on one of the two invariant subsets of $\Z/2\Z \skewprod_{\varphi}  X$. For each $i$, $\bar\mu_i(\{h\}\times X)$ is either $1/3$ or $2/3$, depending on whether $h=0$ or $1$. Therefore, for $i=1$ or $2$, we find ${\delta_\mu (L) = \bar\mu_i(\{0\}\times X)^2 + \bar\mu_i(\{1\}\times X)^2 = 5/9}$.
\end{example}

\begin{example}\label{ex:3pointn}

We can in fact generalize the above example as follows. Take $d\geq 3$ and consider an alphabet $A = \{a_0,\ldots,a_{d-1}\}$ of size $d$. Let $X$ be the $d$-element shift space generated by the periodic word $(a_0\ldots a_{d-1})^\infty$ and $\varphi\from A^*\to \Z/d\Z$ mapping $a_0$ to $0$ and $a_i$ to 1 for $0<i<d$. Then the density of $\varphi^{-1}(0)$ is $(2n-1)/d^2$. Thus the actual value of the density tends to be twice what ergodicity of $\nu\times\mu$ would yield (namely $1/d$), in the sense that $d\delta_\mu(L)\to 2$ as $d\to\infty$. 
\end{example}

Next we show that if $X$ is uniquely ergodic, then ergodicity of $\nu\times\mu$ implies unique ergodicity of $G\skewprod_{\varphi}  X$. This is used  e.g. in Example~\ref{ex:counting2}. It was proved by Veech~\cite{Veech1969} for the case where $X$ is a binary coding of an irrational rotation and $G=\Z/2\Z$. We show below how Veech's argument can be adapted in a straightforward way to the case where $X$ is any uniquely ergodic shift and $G$ is a finite group. The result is true more generally when $G$ is a compact group with normalized Haar measure $\nu$, cf.~\cite[Proposition 3.10]{book/Furstenberg1981}.

\begin{proposition}\label{prop:ue}
    Let $X$ be a uniquely ergodic shift on $A$ with ergodic measure $\mu$ and $\varphi\from A^*\to G$ be a morphism onto a finite group $G$ with uniform probability measure $\nu$. If $\nu\times\mu$ is an ergodic measure on $G\skewprod  _{\varphi}X$, then $G\skewprod _{\varphi}X$ is uniquely ergodic.
\end{proposition}

\begin{proof}
    Let $\zeta$ be an invariant measure on $G\skewprod _{\varphi}X$. Observe that the measure $\zeta\circ P_X^{-1}$, where $P_X\from G\skewprod _{\varphi}X\to X$ denotes the projection on the second component, is an invariant measure on $X$, thus it must be equal to $\mu$ as $X$ is uniquely ergodic. In other words,
    \begin{equation*}
        \zeta(G\times E) = \mu(E),\quad \text{for every measurable set $E\subseteq X$.}
    \end{equation*}

    Let $g\in G$ act on the left of $G\skewprod _{\varphi}X$ by $g(h,x) = (gh,x)$. This action is $T_{\varphi}$-commuting, as well as measure-preserving as is easily checked on rectangular sets (i.e.  sets of the form $F\times E$, where $F \subseteq G$, and  $E\subseteq X$ is  measurable). Thus the measure $g\zeta$ defined by $g\zeta(F) = \zeta(gF)$ is also an invariant measure. We claim that the average measure $\bar\zeta = (\sum_{g\in G}g\zeta)/|G|$ is equal to $\nu\times\mu$. Indeed, for every measurable set $E\subseteq X$ and $h\in G$, we have
    \begin{equation*}
        \bar\zeta(\{h\}\times E) = \frac{1}{|G|}\sum_{g\in G}\zeta(\{gh\}\times E) = \frac{1}{|G|}\zeta(G\times E) = \frac{1}{|G|}\mu(E).
    \end{equation*}

    Thus we conclude that $\nu\times\mu = \bar\zeta$. If $\nu\times\mu$ is ergodic, then it is an extremal point in the convex set of invariant measures of $G\skewprod _{\varphi} X$. Since $\nu\times\mu=\bar\zeta$ is a uniform convex combination of the measures $g\zeta$, it follows that $g\zeta= \nu\times\mu$ for all $g\in G$.
\end{proof}

\subsection{Pointwise densities}\label{subsec:pointwise}

We show  in this section that the concept of  density of a language can be adapted to Ces\`aro averages along orbits of individual points (see Theorem \ref{theo:pointwise}).  
This leads to another proof of the fact the density always exists for group languages (as stated in Theorem \ref{t:first-main}) and an expression for the density as an integral of
pointwise densities.
For the definition of the  pointwise density below, we recall the notation $\varphi^{(i)}(x) = \varphi(x_{[0,i)})$ from \eqref{eq:fn}.

\begin{definition} \label{def:pointwisedensity}
    Let $X$ be a shift space on a finite alphabet $A$ with a shift-invariant measure $\mu$ and let $\varphi\from A^*\to G$ be a morphism onto a finite group $G$.   
    For every group language $L=\varphi^{-1}(K)$, where $K\subseteq G$, we define the \emph{pointwise density} of $L$ along the orbit of a point $x \in X$ as the following limit whenever it exists:
    \[
        \dot\delta_L (x)=  \lim _{n \rightarrow \infty} \frac{1}{n}\sum_{i=0}^{n-1} {1}_{K} (\varphi ^{(i)}(x)) ,
    \]
    where $1_{K}$ stands for the indicator function of the set $K$.
\end{definition}

The main result of the section,  namely Theorem~\ref{theo:pointwise},  states the a.e. existence of pointwise densities for all group languages under all invariant measures. 

\begin{example} \label{ex:counting1}
    We continue Example \ref{ex:counting0}.
    Recall that $A = \{1,\dots,d\}$ and $L = \{ w\in A^* : |w|_1 \equiv r\bmod m\}$, $0\leq r <m$. 
    Alternatively $L= \varphi^{-1}(\{r\})$ where $\varphi\colon A^*\to \Z/m\Z$ is the morphism which sends $1\in A$ to $1\in\Z/m\Z$, and all other letters of $A$ to $0$.  
    The pointwise density $\dot\delta_L(x) $, when it exists, is the frequency of occurrences of prefixes in $x$ having a number of $1$'s congruent to $r$ mod $m$. 
    In particular, Theorem \ref{theo:pointwise} below states that this frequency exists for $\mu$-almost every $x$, for all invariant probability measures $\mu$ on $A^\Z$.
\end{example}

We now turn to Theorem~\ref{theo:pointwise}.
First, let us recall part of Birkhoff's Ergodic Theorem; for a complete statement see \cite[Theorem~2.3]{Petersen1983}. 
Let $(X,T)$ be a dynamical system with an invariant probability measure $\mu$. 
Given an $\mathrm{L}^1$ map $f\from X\to\R$, let us define the \emph{time average} of $f$ as
\[
    \bar f(x) = \lim_{n\to\infty}\frac1n\sum_{k=0}^{n-1}f(T^kx),
\]
whenever it exists.
The Ergodic Theorem states that $\bar f$ exists $\mu$-a.e and moreover
\[
    \int_X \bar f\, d\mu = \int_X f\, d\mu.
\]

\begin{theorem} \label{theo:pointwise}
    Let $X$ be a shift space on a finite alphabet $A$ with a shift-invariant probability measure $\mu$ and  a morphism $\varphi\colon A^*\to G$ onto a finite group $G$. 
    For every group language $L = \varphi^{-1}(K)$, the pointwise density $\dot\delta_L(x)$ exists for $\mu$-almost every $x\in X$. 
    Moreover, the density $\delta_\mu(L)$ exists and is given by
    \[
        \delta_\mu(L) = \int_X \dot\delta_L\,d\mu.
    \]
\end{theorem}

\begin{proof}
    Let $\bar\mu$ be any $T_{\varphi}$-invariant measure on the skew product $G\skewprod _{\varphi} X$ that projects to $\mu$ (for instance the product measure $\nu\times\mu$, where $\nu$ is the uniform probability measure on $G$).
    For $g\in G$, let $f_g$ be the indicator function of $\{g\}\times X$. 
    By the Ergodic Theorem applied to $G\skewprod_{\varphi}X$, the time average function $\bar f_g$ exists on some subset $E_g \subseteq G \times X$ with $\bar\mu(E_g)=1$.

    Next, notice that $\varphi^{(i)}(x) = g$ if and only if $T_\varphi^i(1_G,x) \in \{g\}\times X$, so
    \[
        \dot\delta_L(x) = \lim_{n\to\infty}\frac1n\sum_{i=0}^{n-1}{1}_K(\varphi^{(i)}(x))
        =\sum_{g\in K}\lim_{n\to\infty}\frac1n\sum_{i=0}^{n-1}f_g(T_\varphi^i(1_G,x))
        = \sum_{g\in K}\bar f_g(1_G,x).
    \]

    Let $E = \bigcap_{g \in G} E_g$ and $F = \{ x\in X : (1_G,x)\in E\}$. 
    Then $\dot\delta_L(x)$ certainly exists for all $x\in F$ by the above, so it suffices to show that $\mu(F)=1$. 
    But observe that $T_\varphi^i(h,x)\in \{g\}\times X$ is equivalent to $T_\varphi^i(1_G,x)\in\{h^{-1}g\}\times X$, so that
    \[
        \bar f_g(h,x) = \lim_{n\to\infty}\frac1n\sum_{i=0}^{n-1}f_g(T_\varphi^i(h,x))
        = \lim_{n\to\infty}\frac1n\sum_{i=0}^{n-1}f_{h^{-1}g}(T_\varphi^i(1_G,x))
        = \bar f_{h^{-1}g}(1_G,x).
    \]
    It follows that if $(h,x)\in E$, then $\bar f_{g}(1_G,x) = \bar f_{hg}(h,x)$ for all $g\in G$, and thus $x\in F$. 
    Therefore we have $E\subseteq G\times F$. 
    Since $\bar\mu$ projects to $\mu$, we conclude that $\mu(F) = \bar\mu(G\times F) \geq \bar\mu(E)$. But $E$ is a finite intersection of sets of $\bar\mu$-measure 1, thus $\bar\mu(E)=1$, and $\mu(F)=1$.  

    For the last part of the statement, we first note that for $w\in A^i$, we have $\varphi^{(i)}(x)\in K$ if and only if $x\in[w]$ for some $w\in L\cap A^i$, and thus
    \[
        \int_X {1}_K(\varphi^{(i)}(x))\,d\mu = \sum_{w\in L\cap A^i} \int_X{1}_{[w]}(x)\,d\mu = \mu([L\cap A^i]).
    \]
    Hence we have
    \[
\begin{aligned}
            \int_X \dot\delta_L(x)\,d\mu 
            &= \int_X \lim_{n\to\infty}\frac1n\sum_{i=0}^{n-1} 1_K(\varphi^{(i)}(x))\,d\mu\\
            &=  \lim_{n\to\infty}\frac1n\sum_{i=0}^{n-1}\int_X 1_K(\varphi^{(i)}(x))\,d\mu\\
            &= \lim_{n\to\infty}\frac1n\sum_{i=0}^{n-1}\mu([L\cap A^i]) = \delta_\mu(L),
\end{aligned}
    \]
    where the integral and limit can be swapped by the dominated convergence theorem. 
\end{proof}

\begin{remark}\label{rem:pointwise}
    We can also recover the formula for $\delta_\mu (L)$ in Theorem~\ref{t:first-main} using pointwise densities. 
    Indeed, suppose that  $\mu$ is an ergodic measure on $X$, and let $\bar\mu$ be an ergodic measure on $G\rtimes_\varphi X$ that projects to $\mu$ (which exists by Lemma~\ref{l:existence}).
    We can then use the trick from the proof of Theorem~\ref{t:first-main} involving the sets $U_C=C \times X$ as follows:
    \[
\begin{aligned}
\int_X \dot\delta_L  \, d\mu &= \int_X \lim_{n\to\infty} \frac{1}{n}\sum_{i=0}^{n-1} {1}_K(\varphi^{(i)}(x)) \, d\mu(x)
= \lim_{n\to\infty} \frac{1}{n}\sum_{i=0}^{n-1} \int_X {1}_K(\varphi^{(i)}(x)) \, d\mu(x) \\
&= \lim_{n\to\infty} \frac{1}{n} \sum_{i=0}^{n-1} \mu(\{x \in X: \varphi^{(i)}(x) \in K\})
=\lim_{n\to\infty}\frac{1}{n}\sum_{i=0}^{n-1} \bar\mu(G \times \{x \in X: \varphi^{(i)}(x)\in K\}) \\
&=\lim_{n\to\infty}\frac{1}{n}\sum_{i=0}^{n-1}  \sum_{g \in G} \bar\mu(U_g \cap T_{\varphi}^{-i}U_{gK}) = \sum_{g \in G} \bar\mu(U_g) \bar\mu(U_{gK}).
\end{aligned}
    \]
\end{remark}
  
\begin{example} \label{eg:periodicskewprod-3}
    Consider the three-point example introduced in Example \ref{eg:periodicskewprod-1}, where $X$ is the orbit closure of the two-sided sequence $x=(abc)^\infty$, the group is $\Z/2\Z$ (with its additive structure), and the morphism $\varphi\from \{a,b,c\}^*\to \Z/2\Z$ maps $a,b$ to $1$ and $c$ to 0. 
    The sequences $\varphi^{(i)}(y)$ for $y = x, Sx, S^2x$ are in turn,
\[
\begin{array}{lll}
	100 & 100 & 100 \dots\\
	110 & 110 & 110 \dots \\
	010 & 010 & 010 \dots 
\end{array}
\]
For $L=\varphi^{-1}(0)$, the pointwise densities $\dot\delta_L(y)$ are the average frequencies of 0 in each of these sequences, which are in turn $2/3, 1/3, 2/3$. 
As $X$ is a finite space with a uniform probability measure, the integral of $\dot\delta_L$ is simply the average $(2/3 +  1/3 + 2/3)/3 = 5/9$. Thus $\delta_\mu(L)=5/9$ by Theorem~\ref{theo:pointwise}, which agrees with the value found in Example~\ref{eg:periodicskewprod-2}.
Note that the two ergodic measures on the skew product $\Z/2\Z \rtimes_\varphi X$ are not fully supported and that the function $\dot\delta_L(x)$ is not constant.
\end{example}

We now give two examples which show how our equidistribution statement, namely Corollary~\ref{c:equidistribution}, together with  the unique ergodicity  result from Proposition \ref{prop:ue}, can be applied in this pointwise setting. Both examples will be pursued in Section \ref{s:examples}  for Sturmian shifts (see Examples~\ref{ex:Fibocounting} and \ref{ex:Fibocontfrac}).

\begin{example}\label{ex:counting2}
 This example  is inspired by the equidistribution results   from  Veech~\cite{Veech1969,Veech1975} for irrational rotations in  continuation of Example \ref{ex:counting0}.
 Let $L_r$ and $\varphi$ be as in Example~\ref{ex:counting0}.
 Assume that $X$ is a uniquely ergodic shift whose ergodic measure $\mu$ is such that the product measure $\nu\times\mu$ is ergodic on $G\rtimes_\varphi X$; or equivalently, $G\rtimes_\varphi X$ is uniquely ergodic (Proposition~\ref{prop:ue}).  
 Then Corollary~\ref{c:equidistribution} yields that $\delta_{\mu} (L_r)$ exists and  equals $ 1/m$.

 We now  revisit this result  in terms of   elements of $X$.  Recall that  $1_{\{r\}}$ stands for the indicator function of  the set $\{r\}$ in  $\Z/m\Z$.
Applying unique ergodicity to  the continuous function $1_{\{r\}} \times 1$ on $G\skewprod_{\varphi} X$ yields indeed  that for every $x \in X$
$$ \begin{aligned} 
\lim _{n \rightarrow + \infty}
 \frac{1}{n} \sum_{i=0}^{n-1}  (1_{\{r\}} \times 1) T_{\varphi}^i(0,x)) &=  \lim_{n \rightarrow + \infty} \frac{1}{n} \sum_{i=0}^{n-1} ( 1_{\{r\}} \times 1) (( \varphi^{(i)}(x),S^ix))\\
&=  \lim_{n \rightarrow + \infty} \frac{1}{n} \sum_{i=0}^{n-1} 1_{\{r\}}( \varphi^{(i)}(x)) = \int  1_{\{r\}} \times 1  
=1/m .
\end{aligned}
$$
 In other words, for every $x \in X$,  and for every $r$, one has:
    \begin{equation*}
        \frac{1}{N} \card \{ n : 0 \leq n \leq N-1,\ |x_0 \ldots x_{n-1}|_1  \equiv r  \bmod m\} \to \frac{1}{m}.
    \end{equation*}
  \end{example}
  
  In the next example, we consider a skew product with a non-Abelian skewing group, namely $G(2)= \GL(2, {\mathbb Z}/2{\mathbb Z})$, i.e. the group of $2 \times 2$  matrices with entries in $\Z/2\Z$ and determinant~$1$. The example could also be carried out with $G(m)$ for arbitrary $m\geq 2$, but we treat only the case $m=2$ for simplicity.
    It is inspired by  the work of Jager and Liardet~\cite{JL:88} (see also~\cite{Szusz,Moeckel,Borda2025} for related works), and is  motivated by distribution properties for convergents in continued fraction expansions. 
  
  \begin{example}\label{ex:convergents}
    Let $X$ be a shift on the alphabet $ A=\{1,2\}$ and consider the morphism
    \begin{equation*}
        \varphi \from A^*\to G(2),\quad  k \mapsto 
        \begin{pmatrix} 
            0 & {1} \\ 
            {1} & \overline{k}
        \end{pmatrix},
    \end{equation*}
    where $\overline{k}$ stands for the congruence  class of the integer $k$ modulo $2$. Note that this morphism is onto.
    Let $x=(x_n)_{n\in\Z} \in X$. Consider the  real number in the unit interval $[0,1]$ that admits $(x_n)_{n\in\N}$  as its sequence of partial quotients and let $(p_n(x)/q_n(x))_{n\in\N}$ stand for the  associated sequence of  rational approximations. 
    Recall that one has $q_{-1}(x)=0$, $p_{-1}(x)=1$, $q_0(x)=1$, $p_0(x)=0$, and for all positive $n$,
    \[
\begin{aligned}
            q_{n+1}(x) &=x_{n+1} q_n(x) +q_{n-1}(x), \\ 
            p_{n+1}(x)&=x_{n +1} p_n (x)+p_{n-1}(x),
\end{aligned}
    \] 
    \cite[Theorem~149]{book/Hardy2008} and moreover 
    \begin{equation*}
        \varphi^{(n)}(x)= 
        \begin{pmatrix} 
            \overline{p_{n-1} (x)} &\overline{ p_n (x)}\\
            \overline{q_{n-1} (x)}&   \overline{q_n(x)}
        \end{pmatrix}.
    \end{equation*}
In the case where the skew product $G(2) \skewprod_{\varphi}  X$ turns out to be  uniquely  ergodic, the following distribution results for the sequence  $( \varphi^{(n)}(x))_{n\in\N}$ in the group $G(2)$  can be deduced:    for every $k=1,2$ and for every $x \in X$
    \begin{gather*}
        \lim_{N\to\infty}\frac{1}{N}  \card\{ n:  1\leq  n  \leq N,\  q_n(x) \equiv 0 \bmod 2\}= \frac{1}{3},\\
        \lim_{N\to\infty}\frac{1}{N}  \card\{ n: 1\leq  n  \leq N,\  q_n(x) \equiv 1 \bmod 2\}= \frac{2}{3}.
    \end{gather*}
    The numbers $1/3$ and $2/3$ come from the counting measure on $G(2)$. Consider indeed  the entries of index $(2,2)$
     in the matrices of $G(2)$. There   are twice as many  matrices  in $G(2)$ for which this entry is odd than  matrices for which it is even.
     We will show in Example~\ref{ex:Fibocontfrac} that this analysis applies in particular to the Fibonacci shift, because there unique ergodicity of the skew product follows from Theorem~\ref{t:dendric-ergodic}.
\end{example}

\subsection{An example in the substitutive Fibonacci shift}
\label{ss:fibo}
  
We now proceed to illustrate Theorem~\ref{t:first-main} (and more specifically Corollary~\ref{c:equidistribution}, see also Example \ref{ex:counting1}) in the substitutive Fibonacci shift, whose precise definition is recalled below. We consider the language $L = \varphi^{-1}(0)$ where $\varphi\from \{a,b\}^*\to\Z/2\Z$ is the morphism defined by $\varphi(a) = 1$ and $\varphi(b)=0$; in other words,
\begin{equation*}
    L = \{ w\in \{a,b\}^* : |w|_a \equiv 0 \bmod 2\}.  
\end{equation*}

We show that $\delta_\mu(L)=1/2$ for $\mu$ the unique ergodic measure on the substitutive Fibonacci shift (Proposition~\ref{p:fibo}), but also that the sequence $(\mu([L\cap A^n]))_{n\in\N}$ does not converge in the classical sense (Proposition~\ref{p:lim-fibo}). In particular, the measure $\nu\times\mu$ on the skew product is ergodic but not mixing, cf.~Remark~\ref{r:strong-convergence}.

Consider the substitution $\sigma\colon a\mapsto ab, b\mapsto a$, called the \emph{Fibonacci substitution}. The substitution $\sigma$ is primitive; thus the shift space $X=X(\sigma)$ generated by $\sigma$, called the \emph{substitutive Fibonacci shift}, or \emph{Fibonacci shift} for short, is uniquely ergodic by Michel's theorem (this alternatively follows from Boshernitzan's criterion~\cite{Boshernitzan1984}). Its unique ergodic measure $\mu$ viewed as a map on $\cL(X)$ is depicted in Figure~\ref{f:proba-fibo}.

\begin{figure}[hbt]
    \centering
    \begin{tikzpicture}[xscale=1.45,yscale=.05,every node/.style={circle,minimum size=4ex,inner sep=1pt},font=\small]
        \node[circle,draw] (1)    at (0,0) {$1$};

        \node[circle] (a)    at (1,10) {$\lambda^{-1}$};   
        \node[circle,draw] (b)    at (1,-10) {$\lambda^{-2}$};

        \node[circle,draw] (aa)   at (2,20) {$\lambda^{-3}$};   
        \node[circle] (ab)   at (2,0) {$\lambda^{-2}$};
        \node[circle] (ba)   at (2,-20) {$\lambda^{-2}$};

        \node[circle,draw] (aab)  at (3,30) {$\lambda^{-3}$};   
        \node[circle,draw] (aba)  at (3,10) {$\lambda^{-2}$};
        \node[circle,draw] (baa)  at (3,-10) {$\lambda^{-3}$};   
        \node[circle] (bab)  at (3,-30) {$\lambda^{-4}$};

        \node[circle] (aaba) at (4,40) {$\lambda^{-3}$};   
        \node[circle] (abaa) at (4,20) {$\lambda^{-3}$};
        \node[circle,draw] (abab) at (4,0) {$\lambda^{-4}$};
        \node[circle,draw] (baab) at (4,-20) {$\lambda^{-3}$};   
        \node[circle,draw] (baba) at (4,-40) {$\lambda^{-4}$};

        \node[circle,draw] (aabaa) at (5,50) {$\lambda^{-5}$};   
        \node[circle] (aabab) at (5,30) {$\lambda^{-4}$};   
        \node[circle] (abaab) at (5,10) {$\lambda^{-3}$};
        \node[circle] (ababa) at (5,-10) {$\lambda^{-4}$};
        \node[circle] (baaba) at (5,-30) {$\lambda^{-3}$};   
        \node[circle] (babaa) at (5,-50) {$\lambda^{-4}$};

        \node[circle,draw] (aabaab) at (6,60) {$\lambda^{-5}$};   
        \node[circle,draw] (aababa) at (6,40) {$\lambda^{-4}$};   
        \node[circle,draw] (abaaba) at (6,20) {$\lambda^{-3}$};
        \node[circle,draw] (ababaa) at (6,0) {$\lambda^{-4}$};
        \node[circle,draw] (baabaa) at (6,-20) {$\lambda^{-5}$};   
        \node[circle] (baabab) at (6,-40) {$\lambda^{-4}$};   
        \node[circle] (babaab) at (6,-60) {$\lambda^{-4}$};

        \node[circle] (aabaaba) at (7,70) {$\lambda^{-5}$};   
        \node[circle] (aababaa) at (7,50) {$\lambda^{-4}$};   
        \node[circle] (abaabaa) at (7,30) {$\lambda^{-5}$};
        \node[circle,draw] (abaabab) at (7,10) {$\lambda^{-4}$};
        \node[circle,draw] (ababaab) at (7,-10) {$\lambda^{-4}$};
        \node[circle,draw] (baabaab) at (7,-30) {$\lambda^{-5}$};   
        \node[circle,draw] (baababa) at (7,-50) {$\lambda^{-4}$};   
        \node[circle,draw] (babaaba) at (7,-70) {$\lambda^{-4}$};

        \node[circle] (aabaabab) at (8,80) {$\lambda^{-5}$};   
        \node[circle] (aababaab) at (8,60) {$\lambda^{-4}$};   
        \node[circle] (abaabaab) at (8,40) {$\lambda^{-5}$};
        \node[circle] (abaababa) at (8,20) {$\lambda^{-4}$};
        \node[circle] (ababaaba) at (8,0) {$\lambda^{-4}$};
        \node[circle] (baabaaba) at (8,-20) {$\lambda^{-5}$};   
        \node[circle] (baababaa) at (8,-40) {$\lambda^{-4}$};   
        \node[circle] (babaabaa) at (8,-60) {$\lambda^{-5}$};
        \node[circle,draw] (babaabab) at (8,-80) {$\lambda^{-6}$};

        \draw[above,node font=\small] (1)   edge node{$a$} (a);
        \draw[below,node font=\small] (1)   edge node{$b$} (b);
        \draw[above,node font=\small] (a)   edge node{$a$} (aa);
        \draw[below,node font=\small] (a)   edge node{$b$} (ab);
        \draw[below,node font=\small] (b)   edge node{$a$} (ba);
        \draw[above,node font=\small] (aa)  edge node{$b$} (aab);
        \draw[above,node font=\small] (ab)  edge node{$a$} (aba);
        \draw[above,node font=\small] (ba)  edge node{$a$} (baa);
        \draw[below,node font=\small] (ba)  edge node{$b$} (bab);
        \draw[above,node font=\small] (aab) edge node{$a$} (aaba);
        \draw[above,node font=\small] (aba) edge node{$a$} (abaa);
        \draw[below,node font=\small] (aba) edge node{$b$} (abab);
        \draw[below,node font=\small] (baa) edge node{$b$} (baab);
        \draw[below,node font=\small] (bab) edge node{$a$} (baba);

        \draw[above,node font=\small] (aaba) edge node{$a$} (aabaa);   
        \draw[below,node font=\small] (aaba) edge node{$b$} (aabab);   
        \draw[below,node font=\small] (abaa) edge node{$b$} (abaab);
        \draw[below,node font=\small] (abab) edge node{$a$} (ababa);
        \draw[below,node font=\small] (baab) edge node{$a$} (baaba);   
        \draw[below,node font=\small] (baba) edge node{$a$} (babaa);

        \draw[above,node font=\small] (aabaa) edge node{$b$} (aabaab);   
        \draw[above,node font=\small] (aabab) edge node{$a$} (aababa);   
        \draw[above,node font=\small] (abaab) edge node{$a$} (abaaba);
        \draw[above,node font=\small] (ababa) edge node{$a$} (ababaa);
        \draw[above,node font=\small] (baaba) edge node{$a$} (baabaa);   
        \draw[below,node font=\small] (baaba) edge node{$b$} (baabab);   
        \draw[below,node font=\small] (babaa) edge node{$b$} (babaab);

        \draw[above,node font=\small] (aabaab) edge node{$a$} (aabaaba) ;   
        \draw[above,node font=\small] (aababa) edge node{$a$} (aababaa) ;   
        \draw[above,node font=\small] (abaaba) edge node{$a$} (abaabaa) ;
        \draw[below,node font=\small] (abaaba) edge node{$b$} (abaabab) ;
        \draw[below,node font=\small] (ababaa) edge node{$b$} (ababaab) ;
        \draw[below,node font=\small] (baabaa) edge node{$b$} (baabaab) ;   
        \draw[below,node font=\small] (baabab) edge node{$a$} (baababa) ;   
        \draw[below,node font=\small] (babaab) edge node{$a$} (babaaba) ;
                                                          $ $
        \draw[above,node font=\small] (aabaaba) edge node{$b$} (aabaabab);   
        \draw[above,node font=\small] (aababaa) edge node{$b$} (aababaab);   
        \draw[above,node font=\small] (abaabaa) edge node{$b$} (abaabaab);
        \draw[above,node font=\small] (abaabab) edge node{$a$} (abaababa);
        \draw[above,node font=\small] (ababaab) edge node{$a$} (ababaaba);
        \draw[above,node font=\small] (baabaab) edge node{$a$} (baabaaba);   
        \draw[above,node font=\small] (baababa) edge node{$a$} (baababaa);   
        \draw[above,node font=\small] (babaaba) edge node{$a$} (babaabaa);
        \draw[below,node font=\small] (babaaba) edge node{$b$} (babaabab);
    \end{tikzpicture}
    \caption{The invariant probability measure on the Fibonacci shift ($\lambda =$ the golden ratio). Circled nodes represent elements from the language $L = \{w\in \{a,b\}^* : |w|_a \equiv 0 \bmod 2\}$}\label{f:proba-fibo}
\end{figure}

We next prove that the skew product of the Fibonacci shift with $\Z/2\Z$ for the skewing function determined by $\varphi$ is also uniquely ergodic, as an application of Michel's theorem. The argument is a special case of the general method described in Section~\ref{s:morphic}. 

\begin{proposition}\label{p:fibo}
    The skew product $\Z/2\Z\skewprod _{\varphi} X$ is uniquely ergodic. 
\end{proposition}

\begin{proof}
    Take the substitution $\bar\sigma$ defined as follows on the alphabet $\Z/2\Z\times A$, whose letters are denoted $a_0, a_1, b_0, b_1$ for conciseness,
    
    \begin{equation*}
        \begin{array}{lll}
            \bar\sigma\from & a_0\mapsto a_0b_1a_0a_1b_1,\quad & a_1\mapsto a_1b_0a_0a_1b_0, \\
            & b_0 \mapsto a_0b_1a_1,\quad & b_1 \mapsto a_1b_0a_0.
        \end{array}
    \end{equation*}

    This substitution is primitive and satisfies $\pi\circ\bar\sigma=\sigma^3$ where $\pi\from(\Z/2\Z\times A)^*\to A^*$ is the natural projection (mapping $a_i$ to $a$ and $b_i$ to $b$). Moreover, $\varphi\circ\sigma^3 = \varphi$ and $\bar\sigma(\cL(Y))\subseteq\cL(Y)$, where $Y = \Psi(G\skewprod_{\varphi}  X)$ is the skew product viewed as a shift space on $\Z/2\Z\times A$ via the map $\Psi$ of Lemma~\ref{l:isomskewprod}. It follows that $Y$ is the shift space generated by $\bar\sigma$. Therefore, by Michel's theorem, $Y$ is uniquely ergodic, hence so is $G\skewprod_{\varphi}  X$.
\end{proof}

As an immediate consequence of Corollary~\ref{c:equidistribution} (see also Example~\ref{ex:counting1}), we conclude that the density $\delta_\mu(L)$, 
    where $L = \{ w\in \{a,b\}^*: |w|_a \equiv 0 \bmod 2\}$,
exists and is $1/2$. In other words, the sequence $(\mu([L\cap A^n]))_{n\in\N}$ converges to $1/2$ in Ces\`{a}ro's sense, even though, as we next show, the sequence itself does not converge (this is in contrast with Example~\ref{ex:full} where $\nu\times \mu$ is  mixing). The rest of the section is devoted to the proof of the following result where $F = (F(n))_{n\in\N}$ is the Fibonacci number sequence (starting with $F(0)=0$, $F(1)=1$). 
\begin{proposition}
    \label{p:lim-fibo}
    Let $L = \{ w\in \{a,b\}^*: |w|_a \equiv 0 \bmod 2\}$. 
    Then the sequence of measures ${(\mu([L\cap A^n]))_{n\in\N}}$ does not have a limit, as 
    \begin{equation*}
        \lim_{n\to\infty} \mu([L\cap A^{F(4n)}]) = 1,\quad \lim_{n\to\infty} \mu([L\cap A^{F(4n+2)}]) = 0.
    \end{equation*}
\end{proposition}

The proof relies on a number of key facts about the structure of the language of the Fibonacci shift. An important tool is the notion of \emph{special word}. 
Recall that, for a general shift space $X\subseteq A^\Z$, a word $w\in \cL(X)$ is called \emph{left special} if $aw, bw\in \cL(X)$ for some $a,b\in A$, $a\neq b$; it is called \emph{right special} if instead $wa, wb\in\cL(X)$ for some $a,b\in A$, $a\neq b$; and it is called \emph{bispecial} if it is both left and right special. 
Since it is Sturmian, the Fibonacci shift has the property that for every $n\in\N$, $\cL(X)\cap A^n$ contains \emph{exactly} one left and one right special factor (which might or might not coincide). We next establish two lemmas which are central in the proof of Proposition~\ref{p:lim-fibo}. 
\begin{lemma}\label{l:fibo-special}
    Let $u$ be a left special factor and $v$ be a right special factor in the Fibonacci shift $X$.
    \begin{enumerate}
        \item The word $u' = \sigma^2(ub)$ is also left special.
            \label{i:fibo-left-special}
        \item The word $v' = \sigma^2(va)$ is also right special.
            \label{i:fibo-right-special}
    \end{enumerate}
\end{lemma}

\begin{proof}
    \ref{i:fibo-left-special}. As $u$ is left special, then $au$ and $bu$ are in $\cL(X)$. Let $c$ and $d$ be right extensions of respectively $au$ and $bu$; let $c'$ and $d'$ be such that $\sigma^2(c) = abc'$ and $\sigma^2(d) = abd'$ ($c'$~and $d'$ are either $a$ or $\varepsilon$). Then the following words also belong to $\cL(X)$:
    
    \begin{align*}
        \sigma^2(auc) &= aba\sigma^2(u)abc' = abau'c',\\
        \sigma^2(bud) &= ab\sigma^2(u)abd' = abu'd'.
    \end{align*}
    
    In particular, $au'$ and $bu'$ belong to $\cL(X)$. 
    
    \ref{i:fibo-right-special}. The proof of the second part follows similar lines. Let $c$ and $d$ be right extensions of respectively $va$, $vb$ (in fact $d=a$ since $bb$ does not occur in $X$); let $c'$ be such that $\sigma^2(c) = ac'$. Then we find that the following words belong to $\cL(X)$:
    \begin{equation*}
        \sigma^2(vac) = v'ac',\quad \sigma^2(vba) = \sigma^2(v)ababa = v'ba. 
    \end{equation*}
    Hence both $v'a$, $v'b\in \cL(X)$.
\end{proof}

\begin{lemma}\label{l:fibo-coincidence}
        Let $u_n$ and $v_n$ denote respectively the left and right special factors of length $n$ in the Fibonacci shift $X$. Then whenever $n = F(2k+2)-1$  for some $k\geq 0$, the equality $bu_n = v_nb$ holds, and moreover
    \begin{equation*}
        \varphi(u_n) = \varphi(v_n) \equiv k \mod 2,
    \end{equation*}
    where $\varphi\from \{a,b\}^*\to\Z/2\Z$ is the morphism defined by $\varphi(a)=1$ and $\varphi(b)=0$.
\end{lemma}

\begin{proof}
    Let us first prove that $bu_n = v_nb$ when $n=F(2k+2)-1$, $k\geq 0$. 
    Consider the following recursively defined sequence of words:
    \begin{equation*}
        w_0 = \varepsilon,\quad w_{k+1} = \sigma^2(w_kb).
    \end{equation*}
    In other words, this is the sequence of words starting with:
    \begin{equation*}
       \varepsilon,\quad \sigma^2(b),\quad \sigma^4(b)\sigma^2(b),\quad \sigma^6(b)\sigma^4(b)\sigma^2(b),\quad\ldots
    \end{equation*}

    It is clear from Lemma~\ref{l:fibo-special} that this is a sequence of left special factors of $X$. We claim that $|w_k| = F(2k+2) - 1$ for every $k\geq 0$. Indeed, observe that, for every $i\in\N$, $|\sigma^{2i}(b)| = F(2i+1)$ (a fact easily established by induction) and thus,
    \begin{equation*}
        1+|w_k| = 1+\sum_{i=1}^k|\sigma^{2i}(b)| = \sum_{i=0}^kF(2i+1) = F(2k+2).
    \end{equation*}

    Thus we are reduced to show that $bw_k = v_nb$ where $n = F(2k+2)-1$; or in other words that removing the last letter from $bw_k$ yields a right special factor. We do so by induction on $k$. The basis $k=0$ is trivial since $w_0 = \varepsilon$. Assume that the equality $bw_k = v_nb$ holds for some $k\geq 0$. Then we have:
    \begin{equation*}
        abw_{k+1}a = ab\sigma^2(w_kb)a = \sigma^2(bw_kb)a = \sigma^2(v_nbb)a = \sigma^2(v_na)ba.
    \end{equation*}
    By Lemma~\ref{l:fibo-special}, the word $\sigma^2(v_na)$ is right special, and thus so is the word $a^{-1}\sigma^2(v_na)$ obtained by removing its leading letter $a$. Finally we observe that $a^{-1}\sigma^2(v_na)b = bw_{k+1}$, and thus $bw_{k+1}=v_{m}b$, where $m=F(2k+4)-1$.

    Next, let us fix $k\geq 0$ and let $n=F(2k+2)-1$. Then by the first part of the proof, we know that $w_k = u_n$, $bu_n=v_nb$, and then 
    \[
        \varphi(u_n) = \varphi(bu_n) = \varphi(v_nb) = \varphi(v_n),
    \] 
    We need to show that $\varphi(w_k) \equiv k \bmod 2$. We do so by induction on $k$. Since $w_0 = \varepsilon$ and $w_1 = \sigma^2(b)=ab$, the result is clear for $k=0,1$. Assume then that the result holds for some $k\geq 0$. By definition, $w_{k+2} = \sigma^4(w_{k}b)\sigma^2(b) =\sigma^4(w_kb)ab$. But observe that 
    \[
        \sigma^4(a) = abaababa \quad\text{and}\quad \sigma^4(b)=abaab.
    \]
    Thus we have the relation $\varphi(\sigma^4(x)) \equiv |x| \bmod 2$ for all $x\in A^*$. By induction, we then have 
    \[
        \varphi(w_{k+2}) = \varphi(\sigma^4(w_kb)ab) = \varphi(\sigma^4(w_kb))+\varphi(ab) \equiv |w_kb| +1 \equiv |w_k|+2 \equiv k+2\bmod 2,
    \]
    which concludes the proof.
    
\end{proof}

\begin{proof}[Proof of Proposition~\ref{p:lim-fibo}]
    Recall that $L = \{w\in\{a,b\}^* : |w|_a\equiv 0 \bmod 2\} = \varphi^{-1}(0)$ where $\varphi\from A^*\to\Z/2\Z$ is the morphism defined by $\varphi(a)=1$, $­\varphi(b)=0$.
    Let $\leq_{\lex}$ denote the lexicographic order on words. It follows immediately from the description of $\leq_\lex$ on $\cL(X)\cap A^n$ by Perrin and Restivo~\cite[Theorem~2]{Perrin2012}, that for every $w\in \cL(X)$ with $|w|=n+1$,
    \begin{equation*}
        \varphi(w) = 
        \begin{cases}
            \varphi(v_{n}a) & \text{if}\ w\leq_\lex v_{n}a,\\
            \varphi(v_{n}b) & \text{if}\ w\geq_\lex v_{n}b.
        \end{cases}
    \end{equation*}

    Moreover \cite[Proposition~2]{Perrin2012} states that the lexicographically maximal element of $\cL(X)\cap A^{n+1}$ is $bu_{n}$. In particular, if $bu_n = v_nb$, then $v_nb$ is maximal in $\cL(X)\cap A^{n+1}$, and for such values of $n$,
    \[
        [L\cap A^{n+1}] = 
        \begin{cases}
            [v_nb] & \text{if }\varphi(v_nb)=0, \\
            X\setminus[v_nb] & \text{if }\varphi(v_nb)=1.
        \end{cases}
    \]
    But by Lemma~\ref{l:fibo-coincidence}, we know that $\varphi(v_nb) = \varphi(v_n)\equiv k \bmod 2$, where $n = F(2k+2)-1$, and 
    \begin{equation*}
        \mu([L\cap A^{n+1}]) = 
        \begin{cases}
            \mu([v_nb]) & \text{if $k$ is even,}\\
            1-\mu([v_nb]) & \text{if $k$ is odd.}
        \end{cases}
    \end{equation*}
    Finally, $k$ is even precisely when $n+1 = F(4\ell+2)$ for some $\ell$, and $k$ is odd precisely when $n+1 = F(4\ell)$ for some $\ell$. 
    We conclude the proof by observing that $\lim_{n\to\infty}\mu([v_n]) = 0$, which is a straightforward consequence of \cite[Proposition~13]{Durand2000}.
\end{proof}

\section{The case of Bernoulli measures}
\label{sec:Bernoulli}

This section gives a brief account of the original approach to densities under Bernoulli measures due to Schützenberger~\cite{Schutzenberger1965}, Berstel~\cite{Berstel1972}, and Hansel and Perrin~\cite{Hansel1983} with an approach based on the  algebraic theory of formal languages. Details may be found in the monograph~\cite{BerstelPerrinReutenauer2009}.

A probability measure $\mu$ on the full shift  $A^\Z$ is a \emph{Bernoulli measure} when the map $A^*\to [0,1]$ induced by evaluating $\mu$ on cylinders is a morphism. More explicitly,
\begin{displaymath}
    \mu([uv]) = \mu([u])\mu([v]) \quad \text{and}\quad\sum_{a\in A}\mu(a)=1.
\end{displaymath}
A Bernoulli measure is called \emph{positive} if $\mu([a])>0$ for all $a\in A$. The simplest example of a positive Bernoulli measure is that of the uniform distribution given by $\mu([a])=1/|A|$.

A first result on densities under Bernoulli measures is the following, which generalizes the simple observation that finite languages have zero density. We say that a language $L\subseteq A^*$ is \emph{thin} if there is a word $w\in A^*$ which does not appear in any word of $L$ as a factor. 
\begin{proposition}[{\cite[Proposition~13.2.3]{BerstelPerrinReutenauer2009}}]\label{propositionThin}
    Let $\mu$ be a Bernoulli measure on $A^\Z$ and $L\subseteq A^*$. If $L$ is thin, then one has $\delta_\mu(L)=0$. 
\end{proposition}

Recall that a \emph{rational language} is a language recognized by a finite automaton, which may without loss of generality be assumed to be complete.
Formally this means that there is a triple $\A = (Q,i,F)$ where $Q$ is a finite set of states on which $A^*$ acts on the right, $i\in Q$ is the initial state, and $F\subseteq Q$ is the set of final states, and such that $L = \{w\in A^*\mid i\cdot w\in F\}$. 
\begin{example}\label{eg:rat-1}
    The language $L=\{aa,ab,b\}^*$ is rational. An automaton recognizing $L$ is shown in Figure~\ref{figureAutomaton}; it is in fact the \emph{minimal} automaton of $L$.

    \begin{figure}[hbt]
        \centering
        \tikzset{node/.style={circle,draw,minimum size=.4cm,inner sep=0pt}}
        \tikzstyle{loop above}=[out=115,in=65,loop,looseness=8]
        \begin{tikzpicture}[scale=1.5]
            \node[node] (1) at (0,0) {$1$};
            \node[node] (2) at (1,0) {$2$};

            \draw[<->] (1) to ++(-.3,0);
            \draw[above,->,loop above] (1) edge node {$b$} (1);
            \draw[above,->,bend left] (1) to node {$a$} (2) ;
            \draw[below,->,bend left] (2) to node {$a,b$} (1);
        \end{tikzpicture}
        \caption{An automaton recognizing the language $L=\{aa,ab,b\}^*$ with  the state $1$ being  the initial state and the only final  state}\label{figureAutomaton}
    \end{figure}
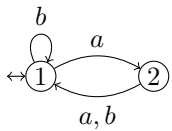
\end{example}

For rational languages, we have the following converse of Proposition~\ref{propositionThin}. We provide a sketch of proof using notions from the algebraic theory of formal languages which may be recalled in~\cite{BerstelPerrinReutenauer2009}.
\begin{proposition}
    Let $\mu$ be a positive Bernoulli measure on $A^\Z$ and $L\subseteq A^*$ be a rational language. If $\delta_\mu(L)=0$, then $L$ is thin.
\end{proposition}

\begin{proof}
    Let $\varphi\from A^*\to M$ be a morphism onto a finite monoid $M$ such that $L = \varphi^{-1}(\varphi(L))$, and let $K$ be the minimal ideal of $M$. For $q\in M$, we denote by $L_q$ the language $\varphi^{-1}(q)$. Observe that every element $m\in K$ satisfies $MmM = K$, and in particular, $L_q$ is thin whenever $q\in M\setminus K$. As a result we deduce that $\delta_\mu(L)=\delta_\mu(L \cap L_K)$ and $\delta_\mu(L_K)=1$, where $L_K = \varphi^{-1}(K)$. 

    We claim that $\delta_\mu(L_m)>0$ for all $m\in K$. First, observe that there must be at least one $n\in K$ with $\delta_\mu(L_n)>0$. Next, fix $m\in K$ and let $r,s\in M$ such that $rns=m$; let $u,v\in A^*$ such that $\varphi(u)=r$ and $\varphi(v)=s$. It follows that $uL_nv \subseteq L_m$. Then the fact that $\mu$ is a Bernoulli measure implies that
    \begin{equation*}
        \delta_\mu(uL_nv) = \mu([u])\mu([v])\delta_\mu(L_n).
    \end{equation*}
    Since $\mu$ is positive we get $\mu([L_m]) \geq \mu([u])\mu([v])\delta_\mu(L_n) > 0$. It now follows that if $\delta_\mu(L) = 0$, then $L$ must be a union of $L_q$ with $q\in M\setminus K$. This concludes the proof since finite unions of thin codes are thin, cf.~\cite[Proposition~2.5.8]{BerstelPerrinReutenauer2009}. 
\end{proof}

Note however that the previous proposition fails for non-rational languages: take for instance 
\begin{displaymath}
    L=\{a^nw\mid n\ge 0,w\in\{a,b\}^n\}.
\end{displaymath}
Then one has $\mu([L\cap A^{2n}])=1/\mu([a])^n$, which implies $\delta_\mu(L)=0$. However $L$ is clearly not thin, since every word appears as a suffix.

Next we give a proof of the fact that the density of a rational language under a Bernoulli measure always exists.

\begin{proposition} \label{prop:rational}
    Let $\mu$ be a Bernoulli measure on $A^\Z$ and $L\subseteq A^*$ be a rational language. Then the density $\delta_\mu(L)$ exists and if $\mu([A])\subseteq \Q$, it is a rational number.
\end{proposition}

\begin{proof}
    Let $\A=(Q,i,F)$ be a complete deterministic automaton recognizing $L$. For every terminal state $t\in F$, let $L_t$ be the language recognized by $(Q,i,t)$. Since $L=\cup_{t\in F}L_t$, it is enough to prove that every $L_t$ has a density. Thus we may assume that $F=\{t\}$. 
    Let $M$ be the $Q\times Q$-matrix defined by
    \begin{displaymath}
        M_{p,q}=\mu(\{x\in A^\Z\mid p\cdot x_0=q\}).
    \end{displaymath}
    Observe that $M$ is a stochastic matrix, and since the sets $\{x\in A^\Z\mid p\cdot x_0=q\}$ are disjoint unions of cylinders $[a]$ for $a\in A$, it has rational entries under our hypothesis.
    Moreover this matrix has the property that
    \begin{displaymath}
    \mu([L\cap A^n])=M_{i,t}^n.
    \end{displaymath}

    Without loss of generality we may assume that every $q\in Q$ is on a path from $i$ to $t$ (otherwise $q$ can be removed). Let $U$ be the set of labels of paths $i\to t$ and which pass by $t$ exactly once; let $V$ be the set of labels of simple loops $t\to t$. Then $L = UV^*$, $0\leq\mu([U])<\infty$, and $\delta_\mu(UV^*) = \mu([U])\delta_\mu(V^*)$ exists whenever $\delta_\mu(V^*)$ exists~\cite[Proposition~13.2.5]{BerstelPerrinReutenauer2009}. Hence we may assume moving forward that there is a path from $t$ to $i$. In particular, it means that the matrix $M$ is \emph{irreducible}, i.e.  for all $p,q$ in $Q$, there exists an  integer $k$ such that  $M_{p,q}^k >0$, or else
 $(M+I)^{|Q|-1}>0$ where $I$ is the $Q\times Q$ identity matrix. 
     By the Perron--Frobenius Theorem, the numbers $M_{i,t}^n$ converge in average and thus the density of $L$ exists. 
     Moreover, the matrices $M^n$ converge as $n\to\infty$ to a matrix with all rows equal to a stochastic eigenvector $v$ of $M$ of eigenvalue 1. 
     But since $M$ has rational entries, the eigenspace of eigenvalue 1 (which has dimension one by the Perron--Frobenius Theorem) has a spanning vector $u$ with only rational entries. 
     Therefore the aforementioned vector $v$ is equal to $\pm u/\Vert u\Vert_1$, where $\Vert u\Vert_1$ denotes the 1-norm of $u$, and thus also has rational entries.
\end{proof}

\begin{example}
    Let $L$ and $\A$ be the language and automaton from Example~\ref{eg:rat-1}. Set $\mu([a])=p$ and $\mu([b])=1-p$. The matrix $M$, which is irreducible, is
    \begin{equation*}
        M=\begin{bmatrix}1-p&p\\1&0\end{bmatrix}.
    \end{equation*}
    The normalized left eigenvector is 
    \begin{equation*}
        v=\begin{bmatrix}1/(1+p)&p/(1+p)\end{bmatrix}.
    \end{equation*}
    Thus $\delta_\mu(L)=1/(1+p)$  (by considering here $M_{1,1}$ since $i=t=1$).
\end{example}

\begin{remark}
    The proof of the above theorem uses a Markov chain associated to an automaton. It can also be understood as a skew product $Q\skewprod X$ of the set $Q$ with the shift $X=A^\Z$, in the same way as we did with groups. Indeed, if $\A=(Q,i,F)$ is a deterministic automaton, then the set $Q\times X$ may be turned into   a dynamical system using the transformation
    \begin{displaymath}
        T(q,x)=(q\cdot x_0,Sx).
    \end{displaymath}

    Let $\rho$ be the unique probability measure on $Q$ such that
    \begin{displaymath}
        \rho(q)=\sum_{p\cdot a=q}\rho(p)\mu([a]).
    \end{displaymath}
    Then $\rho\times\mu$ is an invariant probability measure on the skew product $Q\skewprod X$, which is ergodic whenever the matrix of the automaton is irreducible. Moreover, it is mixing as soon as $\A$ is aperiodic, that is when the $\gcd$ of the lengths of cycles in $\A$ is $1$; or equivalently the matrix $M$ is primitive. In this case we recover the conclusion of Remark~\ref{r:strong-convergence}. Also,  by ergodicity of Bernoulli measures,  the existence of the density in Proposition~\ref{prop:rational} is a special case of Theorem~\ref{t:first-main}.
\end{remark}

We end with a discussion on prefix codes, a notion which also appears later in Section~\ref{s:bifix}. A \emph{prefix code} is a subset $U\subseteq A^*$ where no word is a strict prefix of another. Any prefix code must satisfy $\mu([U])\leq 1$ for any Bernoulli measure $\mu$, cf.~\cite[Proposition~3.7.1]{BerstelPerrinReutenauer2009}. 

If $U$ is a prefix code such that $\mu([U])=1$, then its \emph{average length} (relatively to $\mu$)  is defined as $$\ell(U)=\sum_{p\in U}|p|\,\mu([p]).$$ It is well-known that the  average length satisfies
\begin{equation}\label{eqAverage}
    \ell(U)=\mu([P]),
\end{equation}
where $P$ is the set of proper prefixes of the words in $U$~\cite[Proposition~3.7.11]{BerstelPerrinReutenauer2009}. The following is closely related to our second main result, Theorem~\ref{t:second-main}. Its proof is based on relations between associated generating functions.
 
\begin{proposition}[{\cite[Theorem 13.2.11]{BerstelPerrinReutenauer2009}}]
     Let $\mu$ be a positive Bernoulli measure on $A^{\Z}$.    
     Let $U$ be a prefix code such that $\mu([U])=1$ and $\ell(U)<\infty$. Then $\delta_\mu(U^*) = 1/\ell(U)$.
\end{proposition}

Note that a \emph{rational} prefix code satisfying $\mu([U])=1$ must satisfy the assumption $\ell(U)<\infty$. Indeed, any rational code is thin by \cite[Proposition~2.5.20]{BerstelPerrinReutenauer2009}, hence the set $P$ of proper prefixes of $U$ is also thin, which implies $\mu([P])<\infty$ by \cite[Proposition~2.5.12]{BerstelPerrinReutenauer2009}. 

Let $\varphi\colon A^*\to G$ be a morphism from $A^*$ onto a finite group $G$. If $H$ is a subgroup of $G$, then the submonoid $M=\varphi^{-1}(H)$ is generated by the prefix code $U$ consisting of the nonempty words in $M$ with no non-trivial prefix in $M$; we say that $U$ is a \emph{group code}. Observe in fact that $U$ is also equal to the nonempty words in $M$ with no non-trivial suffix in $M$, and as a result has the dual property of being a \emph{suffix code} (no element of $U$ is suffix of another). Note that sets which are both prefix and suffix codes are called \emph{bifix codes}; bifix codes will be at the heart of Section~\ref{ss:codes}. 
\begin{proposition}\label{prop:bernoulli}
    Let $\mu$ be a positive Bernoulli measure on $A^{\Z}$. Let $\varphi\colon A^*\to G$ be a morphism from $A^*$ onto a finite group $G$,  and let  $H$ be a subgroup of $G$. Let $U$ be the group code defined by $\varphi\colon A^*\to G$ and $U^*=\varphi^{-1}(H)$. Then $\ell(U)=d$ and $\delta_\mu(U^*)=1/d$ with $d=[G:H]=|H|/|G|$.
\end{proposition}
\begin{proof}
The proof that $\ell(U)=d$ uses the fact that the set $P$ of proper prefixes of $U$ is a disjoint union of $d$ suffix codes $U_i$ such that $\mu([U_i])=1$. Then, using \eqref{eqAverage}, we obtain $\ell(U)=\sum_{i=1}^{d}\mu([U_i])=d$.
For further details, see~\cite[Corollary~6.13.16, Theorem~13.2.9]{BerstelPerrinReutenauer2009}, noting that the code $U$ is a maximal bifix code~\cite[p.~64, 65]{BerstelPerrinReutenauer2009} which is moreover rational, and thus also thin~\cite[Proposition~2.5.20]{BerstelPerrinReutenauer2009}.
\end{proof}
\begin{example}
    Let $A=\{a,b\}$. Set $p=\mu([a])$, $q=\mu([b])$. Let $\varphi\colon A^*\to \Z/2\Z$ be defined by $\varphi(a)=0$, $\varphi(b)=1$. Then $\varphi^{-1}(1)=U^*$ with $U=\{a\}\cup ba^*b$.  The set $P$ of proper prefixes of $U$ is $P=\{\varepsilon\}\cup ba^*$, and we have $\ell(U)=\mu([\{\varepsilon\}\cup ba^*])=1+q/(1-p)=2$.
\end{example}

\section{Densities in shifts of finite type}\label{s:SFT}

The aim of this section is to apply our main density formula (Theorem~\ref{t:first-main}) within the setting of \emph{shifts of finite type}. We provide a simple condition which guarantees topological transitivity of the skew product with a finite group; this in turn implies ergodicity for the product of the uniform probability measure on the group and a Markov measure on the shift. This also includes the case of Bernoulli measures  which we already discussed in the previous Section~\ref{sec:Bernoulli}.

The question of ergodicity for skew products over Bernoulli measures was studied by Kakutani~\cite{Kakutani1951}, and the more general Markov case was studied by Bufetov~\cite{Bufetov2003}. The latter introduced  a condition that he called ``strongly connected''   to characterize ergodicity of certain skew products involving Markov measures~\cite[Theorem~4]{Bufetov2003}.  Restating Bufetov's criterion in our setting, for consistency of terminology, we rename this condition \emph{strong irreducibility}; his result was also recently extended in a paper by Lummerzheim et al.~\cite[Theorem~4.3]{Lummerzheim2025}. 

In Section~\ref{ss:sft} we define $\varphi$-irreducibility of a subshift with respect to a morphism onto a finite group, which characterizes topological transitivity in skew products. Section~\ref{ss:Markov} treats skew products of shift spaces over Markov measures, using the fact that irreducibility implies ergodicity. In Section~\ref{ss:strong} we discuss strong irreducibility, which provides topological transitivity simultaneously for all skew products.

\subsection{\texorpdfstring{Shifts of finite type and $\varphi$-irreducibility}{Shifts of finite type and phi-irreducibility}}
\label{ss:sft}

Recall that a shift $X$ is an $r$-step shift of finite type (SFT, $r\geq 1$) if there is a list $F\subseteq A^{r+1}$ of \emph{forbidden factors} of length $r+1$, with the property that an infinite word $x \in A^{\Z}$ belongs to $X$ precisely when none of its factors of length $r+1$ are in $F$.
Recall also that  a shift  $X$ is  \emph{topologically transitive} if for every pair $(U,V)$ of nonempty open sets in $X$, there is $n>0$ for which 
$S^nU \cap V \neq \varnothing$. This is equivalent with  the \emph{irreducibility} of $X$,  i.e. for  every 
 $u,v\in\cL(X)$, there is $w \in A^*$  such that 
$uwv \in \cL(X)$.   
For more on shifts of finite type,  see e.g.~\cite{LM:95,Petersen1983}.  
Topological transitivity will be used to prove ergodicity  for $r$-step Markov measures fully supported on  $r$-step  shifts of finite type in Section~\ref{ss:Markov}. 

\begin{definition}\label{d:phi-irreducible}
    Let $X$ be a shift on $A$, $G$ a finite group, and $\varphi\from A^*\to G$ be a morphism onto $G$. We say that $X$ is \emph{$\varphi$-irreducible} if, for all $u,v\in\cL(X)$, there exists $w\in A^*$ such that $uwv\in \cL(X)$ and $\varphi(uw) = 1_G$. 
\end{definition}

Clearly $\varphi$-irreducibility always implies irreducibility. This is a special case of the following remark.
\begin{remark}
    For a morphism $\varphi\from A^*\to G$ as above, let
    \begin{equation*}
        \ker(\varphi) = \{(u,v)\in A^*\times A^* \mid \varphi(u) = \varphi(v)\}.
    \end{equation*}
    Observe that if $\psi\from A^*\to H$ is another morphism onto a finite group $H$ such that $\ker(\varphi)\subseteq\ker(\psi)$, then there exists a unique morphism $\beta\from G\to H$ such that $\beta\circ\varphi = \psi$, and as a result, $\varphi$-irreducibility implies $\psi$-irreducibility.
\end{remark}

We shall prove the following result, which shows that $\varphi$-irreducibility is precisely the notion needed for topological transitivity of the original shift to propagate to skew products with finite groups.
\begin{theorem}\label{t:transitive}
    Let $X$ be an $r$-step shift of finite type on $A$ and let $\varphi\from A^*\to G$ be a morphism onto a finite group.  Then $X$ is $\varphi$-irreducible if and only if the skew product $G \skewprod_{\varphi}  X$ is topologically transitive. 
\end{theorem}

Before doing so, we establish an intermediate result involving the following notion.

\begin{definition}
    We say that $X$ is \emph{fiber ergodic} with respect to $\varphi$ if, for every $g,h\in G$ there exists $w\in\cL(X)$ such that $g\varphi(w) = h$; or equivalently the restriction of $\varphi$ to $\cL(X)$ is onto.
\end{definition}

Fiber ergodicity follows from $\varphi$-irreducibility for shifts of finite type, as shown next.

\begin{lemma}\label{l:fiber-ergodicity}
    Let $X$ be a shift of finite type on $A$ and let $\varphi\from A^*\to G$ be a morphism onto a finite group. If $X$ is $\varphi$-irreducible, then it is fiber ergodic with respect to $\varphi$. 
\end{lemma}

\begin{proof}
    Let $r\geq 1$ such that $X$ is an $r$-step SFT. Take $g\in G$. Since $\varphi$ is onto, we may find letters $a_1,\ldots, a_k\in A$ such that $\varphi(a_1\ldots a_k) = g$. For $i=1,\ldots,k$ let $u_i$ be a word of length $r+1$ in $\cL(X)$ starting with $a_i$. Let $t_i$ and $v_i$ be the suffix and prefix of length $r$ of $u_i$. By assumption, there exists for each $i = 1,\ldots k-1$ a word $w_i$ such that $t_iw_iv_{i+1}\in\cL(X)$ and $\varphi(t_iw_i)=1_G$. Since $X$ is an $r$-step shift, it follows that the word $z = u_1w_1\ldots w_{k-1}u_kw_k$ belongs to $\cL(X)$, and
    \begin{equation*}
        \varphi(z) = \varphi(a_1)\varphi(t_1w_1)\ldots\varphi(a_k)\varphi(t_kw_k) = \varphi(a_1\ldots a_k) = g,
    \end{equation*}
    which proves that $X$ is fiber ergodic.
\end{proof}

\begin{proof}[Proof of Theorem~\ref{t:transitive}]
    For $g\in G$, define a relation $\prec_g$ on $\cL(X)$ by $u\prec_gv$ if there exists $w\in \cL(X)$ such that $uwv\in\cL(X)$ and $\varphi(uw)=g$. Observe that $u\prec_gv$ precisely when $T^m(\{1_G\}\times[u]_X)$ intersects $\{g\}\times[v]_X$ for some $m\geq |u|$. Therefore $G\skewprod_{\varphi}  X$ is topologically transitive precisely when all relations $\prec_g$, $g\in G$, are total. In particular, whenever this is the case, $\prec_{1_G}$ must contain all pairs $u,v\in \cL(X)$, which is precisely the definition of $\varphi$-irreducibility. Thus topological transitivity of $G\skewprod_{\varphi}  X$ implies $\varphi$-irreducibility of $X$. It remains to prove the converse.

    Assume that $X$ is $\varphi$-irreducible, so that ${\prec} = {\prec_{1_G}}$ contains all pairs of words in $\cL(X)$. Take $u,v\in\cL(X)$ and $g\in G$; we need to show that $u\prec_g v$. 

    By fiber ergodicity (which holds thanks to Lemma~\ref{l:fiber-ergodicity}), there is $u' \in \cL(X)$ such that $\varphi(u')=g$. Since $u\prec u'$, there is $z\in\cL(X)$ such that $uzu'\in\cL(X)$ and $\varphi(uz)=1_G$. Then $w_0=zu'$ satisfies $uw_0 \in \cL(X)$ and $\varphi(uw_0)=g$.

    Extend $uw_0$ to a word $wu_0v' \in \cL(X)$ with $|v'| \geq r+1$. Since $v'\prec v$, there is a word $v''$ such that $v'v''v \in \cL(X)$ and $\varphi(v'v'')=1_G$. Then all subwords of length $r+1$ of $uw_0v'v''v$ are in $\cL(X)$, and hence $uw_0v'v''v \in \cL(X)$. Finally, letting $w=w_0v'v''$, we have $uwv \in \cL(X)$ and $\varphi(uw)=g$. This shows that $u\prec_g v$, concluding the proof.
\end{proof}
It is not hard to see directly  that topological transitivity of $G\skewprod_{\varphi}  X$ implies fiber ergodicity (as also follows from Theorem~\ref{t:transitive}); but the converse is false, as shown by the following example.

\begin{example}  \label{eg:periodicskewprod-2bis} 
Take again the skew product from Example~\ref{eg:periodicskewprod-1}, based on the three-element shift generated by the periodic infinite word $(abc)^{\infty}$ taken with respect to the morphism $\varphi\from\{a,b,c\}^*\to \Z/2\Z$, $\varphi(a)=\varphi(b)=1$, $\varphi(c)=0$. Then $X$ is fiber ergodic but $G\skewprod_{\varphi}  X$ is not topologically transitive.
\end{example}

Finally we prove the following useful property of $\varphi$-irreducibility in shifts of finite type. It shows that for shifts of finite type, the task of verifying the condition in Definition~\ref{d:phi-irreducible} can be reduced to a finite set of words.

\begin{proposition}\label{p:reduction}
    Let $X$ be an $r$-step shift of finite type on $A$ and $\varphi\from A^*\to G$ be a morphism onto a finite group. Then for $X$ to be $\varphi$-irreducible, it suffices that, for all $u,v\in\cL(X)$ with $|u|=|v|=r$, there exists $w\in A^*$ such that $uwv\in \cL(X)$ and $\varphi(uw) = 1_G$. 
\end{proposition}

\begin{proof}
    Let $u,v\in \cL(X)$. If $|v|<r$ then we may simply replace $v$ by one of its extensions in $\cL(X)$ of length $r$. If $|v|>r$, then we may likewise replace it by its prefix of length $r$. Thus we may assume moving forward that $|v|=r$. 

    First we assume that $|u|>r$. Let $u = pq = p'q'$ where $|q| = |p'| = r$. By assumption, there are words $w,w'$ such that $qw'p', qwv\in \cL(X)$ and $\varphi(qw) = \varphi(qw') = 1_G$. Observe that, for every $n\in\N$, $u(w'u)^nwv$ has all of its factors of length $r$ in $\cL(X)$. Thus the word $z_n = (w'u)^nw$ is such that $uz_nv\in\cL(X)$, while $\varphi(uz_n) = \varphi(p)^n$. Taking $n = |G|$, we get $\varphi(uz_n) = 1_G$, as needed.

    It remains to handle the case where $|u|<r$. Take a word $p$ such that $|p|=r-|u|$ and $pu\in\cL(X)$. Then we may find words $w, w'$ such that $puw'uwv\in\cL(X)$ with $\varphi(puw') = 1_G$ and $\varphi(puw'uw) = 1_G$, thus $\varphi(uw) = 1_G$.
\end{proof}

\begin{example} \label{eg:3element}
    Consider once again the three-element shift $X$ from Example~\ref{eg:periodicskewprod-1} generated by the periodic word $(abc)^\infty$, which is an irreducible 1-step shift of finite type. Let $\varphi\from \{a,b,c\}^*\to\Z/2\Z$, $\varphi(a) = \varphi(b) = 1$ and $\varphi(c) =0$.

    Observe that every word $w$ such that $awb\in\cL(X)$ is of the form $(bca)^n$ for some $n\geq 0$. In particular, it follows that for every such word $w$, $\varphi(aw) = 1$, thus $X$ is not $\varphi$-irreducible (though notice that it is fiber ergodic). It is however $\varphi$-irreducible if $\varphi$ is similarly defined but takes values instead in $\Z/3\Z$.
\end{example}

\begin{example}\label{eg:golden-mean-1}
    Let $X$ be the \emph{golden mean} shift, i.e. the 1-step SFT formed by sequences in $\{a,b\}^\Z$ avoiding the factor $bb$. Take the morphism $\varphi\from \{a,b\}^*\to\Z/2\Z$, $\varphi(a) = 1$, $\varphi(b)=0$.  As seen before, this choice of  a morphism allows one to count the parity of occurrences of $a$'s.  Then $X$ is $\varphi$-irreducible, as evidenced by the fact that $aab, aaa, ba, baab\in\cL(X)$, according to Proposition~\ref{p:reduction}. It is also not hard to verify directly that the skew product viewed as an SFT under the topological conjugacy $\Psi$ from Lemma~\ref{l:isomskewprod} is indeed irreducible. The shift $X$ and the skew product $\Z/2\Z\skewprod_{\varphi}  X$ are depicted in Figure~\ref{f:golden-mean}. Note that the same example works with $G=\Z/m\Z$ by considering longer words.
    \begin{figure}
        \centering
        \begin{tabular}{cc}
            \begin{tabular}{c}
                \begin{tikzpicture}[xscale=1.5]
                    \node[anchor=base] (a) at (0,0) {\strut$a$} ;             
                    \node[anchor=base] (b) at (1,0) {\strut$b$} ;             
                    \draw[->] ([yshift=-2pt]a.east) to ([yshift=-2pt]b.west) ;             
                    \draw[->] ([yshift=2pt]b.west) to ([yshift=2pt]a.east) ;             
                    \draw[->] (a) to [loop above] (a) ;
                \end{tikzpicture}
            \end{tabular}
            &
            \begin{tabular}{c}
                \begin{tikzpicture}[xscale=1.5,yscale=1.2]
                    \node[anchor=base] (1a) at (0,1) {\strut$(0,a)$} ;             
                    \node[anchor=base] (0a) at (0,0) {\strut$(1,a)$} ;             
                    \node[anchor=base] (1b) at (1,1) {\strut$(0,b)$} ;             
                    \node[anchor=base] (0b) at (1,0) {\strut$(1,b)$} ;             
                    \draw[->] ([xshift=-2pt]0a.north) to ([xshift=-2pt]1a.south) ;
                    \draw[->] (0a) to (1b) ;             
                    \draw[->] ([xshift=2pt]1a.south) to ([xshift=2pt]0a.north) ;
                    \draw[line width=2pt,white,->] (1a) to (0b) ; 
                    \draw[->] (1a) to (0b) ;
                    \draw[->] (0b) to (0a) ;             
                    \draw[->] (1b) to (1a) ;             
                \end{tikzpicture}
            \end{tabular}
        \end{tabular}
        \caption{The SFT and skew product from Example~\ref{eg:golden-mean-1}. The arrows  represent the two  respective dynamics, namely the shift map and the skew product  map}
        \label{f:golden-mean}
    \end{figure}
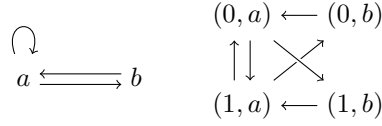
\end{example}

\subsection{Markov measures}\label{ss:Markov}

Using the topological transitivity condition from Theorem~\ref{t:transitive}, we establish ergodicity for skew products involving \emph{invariant Markov measures}. Recall that a measure $\mu$ on $A^\Z$ is an $r$-step Markov measure, $r\geq 1$, when for every word $w\in A^*$ of length $m\geq r$ such that $\mu(x_{[0,m)} = w)\neq 0$ and every letter $a\in A$, 
\begin{equation*}
    \mu(x_m = a \mid x_{[0,m)} = w) = \mu(x_{m}=a \mid x_{[m-r,m)} = w_{[m-r,m)}). 
\end{equation*}

Invariant Markov measures are also called \emph{stationary}. The support of an $r$-step invariant Markov measure $\mu$ is an $r$-step shift of finite type. We say that an $r$-step Markov measure $\mu$ is \emph{irreducible} when for all $u,v\in A^r$, there exists $m>0$ such that
\begin{equation*}
    \mu(x_{[m,m+r)}=v\mid x_{[0,r)}=u)>0.
\end{equation*}

Notice that this is equivalent to irreducibility of the shift of finite type supporting $\mu$. It is well known that a Markov measure is ergodic if and only if it is irreducible. The 1-step case may be found in~\cite[pp.~51--53]{Petersen1983}, while the general $r$-step case can be reduced to $r=1$ by passing to the higher block shift (whose definition is recalled after Example~\ref{ex:full}). We shall prove our second main result, which we now recall. 

\secondmain*

Recall that if $X$ is $\varphi$-irreducible, then it is irreducible, and so under the assumptions of the above theorem, $\mu$ itself must be ergodic. Recall also, from Lemma~\ref{l:isomskewprod}, that for every morphism $\varphi\from A^*\to G$ onto a finite group $G$ and every shift $X$, the skew product $X\skewprod _{\varphi} G$ is topologically conjugate to a subshift of $(G\times A)^\Z$ via the map 
\begin{equation*}
    \Psi(g,x)_n=(g\varphi^{(n)}(x),x_n),\quad n\in\Z.
\end{equation*}
    
\begin{lemma}\label{l:markov}
    Let $X$ be an $r$-step shift of finite type, $r\geq 1$, and $\varphi\from A^*\to G$ be a morphism onto a finite group $G$ with uniform probability measure $\nu$. 
    \begin{enumerate}
        \item The set $\Psi(G\skewprod _{\varphi} X)$ is an $r$-step shift of finite type.\label{i:sft}
        \item For every $r$-step Markov measure $\mu$ fully supported on $X$, the measure $(\nu\times\mu)\circ\Psi^{-1}$ is a Markov measure fully supported on $\Psi(G\skewprod _{\varphi} X)$.\label{i:markov}
    \end{enumerate}
\end{lemma}

\begin{proof}
    \ref{i:sft}. Take a word $\omega\in (G\times A)^*$ with $\omega_i = (g_i,w_i)$ and $m=|\omega|\geq r+1$. Observe that $\omega$ belongs to the language of $\Psi(G\skewprod_{\varphi}  X)$ as long as $w = w_0\ldots w_{m-1}$ belongs to $\cL(X)$ and $g_i\varphi(w_i) = g_{i+1}$. Both of those conditions only need to be verified on factors of length $r+1$ of $\omega$, hence $\Psi(G\skewprod _{\varphi} X)$ is indeed an $r$-step shift of finite type.

    \ref{i:markov}. Fix an $r$-step Markov measure on $X$ and let $\pi = (\nu\times\mu)\circ\Psi^{-1}$. Fix a word $\omega\in (G\times A)^*$, with $\omega_i = (g_i,w_i)$, $m=|\omega|\geq r$, and $w = w_0\ldots w_{m-1}$. Assume that $\omega$ belongs to the language of $\Psi(G\skewprod_{\varphi}  X)$, which means that $w\in\cL(X)$ and $g_{i+1} = g_i\varphi(w_i)$, $i=0,\ldots,m-2$. Take a letter $\alpha = (g,a)\in G\times A$. In case $g\neq g_{m-1}\varphi(a)$ then it is clear that 
    \begin{equation*}
        \pi(\xi_m=\alpha \mid \xi_{[0,m)} = \omega) = \pi(\xi_m=\alpha \mid \xi_{[m-r,m)} = \omega_{[m-r,m)}) = 0. 
    \end{equation*}

    Thus we may suppose from now on that $g=g_{m-1}\varphi(a)$. Then we have
    \begin{align*}
        \pi( \xi_m = \alpha \mid \xi_{[0, m)} = \omega ) 
        &= \frac{\pi( \xi_m = \alpha,\ \xi_{[0, m)} =  \omega )}{\pi( \xi_{[0, m)} = \omega )} \\
        &= \frac{(\nu\times\mu)(\{g_0\}\times [wa])}{(\nu\times\mu)(\{g_0\}\times [w])} \\
        &= \frac{\mu(x_m = a,\ x_{[0, m)} = w_{[0, m)})}{\mu(x_{[0,m)} = w)} \\
        &= \mu( x_m = a \mid x_{[0,m)}=w ).
    \end{align*}
    Using a similar argument together with the invariance of $\pi$, we find 
    \begin{equation*}
        \pi( \xi_m = \alpha \mid \xi_{[m-r,m)} = \omega_{[m-r,m)} ) = \mu( x_m = a \mid x_{[m-r,m)}=w_{[m-r,m)} ).
    \end{equation*}
    The fact that $\pi$ is $r$-step Markov now follows directly from the fact that $\mu$ is.
\end{proof}

\begin{proof}[Proof of Theorem~\ref{t:second-main}]
    In light of Lemma~\ref{l:markov}, the measure $(\nu\times\mu)\circ\Psi^{-1}$ is an $r$-step Markov measure fully supported on a shift of finite type, which is topologically transitive by Theorem~\ref{t:transitive}. Therefore, $(\nu\times\mu)\circ\Psi^{-1}$ must be ergodic, and since $\Psi$ is a topological conjugacy, we deduce that the measure $\nu\times\mu$ is ergodic as well.

    The last part of the statement is then an immediate consequence of Corollary~\ref{c:equidistribution}.
    
\end{proof}

\begin{example}\label{ex:full}
    Consider the full shift on $A=\{a,b\}$, equipped with a Bernoulli measure $\mu\from a\mapsto p$, $b\mapsto 1-p$, where $0<p<1$. Let $\varphi\to\Z/2\Z$ be the morphism defined by $a\mapsto 1$, $b\mapsto 0$, and $\nu$ be the uniform measure probability measure on $\Z/2\Z$. Since, clearly, the full shift is $\varphi$-irreducible, the product measure $\nu\times\mu$ is ergodic on the skew product $\Z/2\Z\skewprod_{\varphi}  A^\Z$. But looking directly at the transition diagram of the skew product, we can see that it is not only irreducible but also aperiodic. Hence we deduce that the skew product is strongly mixing, and by Remark~\ref{r:strong-convergence},
    \begin{equation*}
        \delta_\mu(L) = \lim_{n\to\infty}\mu([L\cap A^n]) = 1/2. 
    \end{equation*}
\end{example}

Next let us  briefly describe how ergodicity for $r$-step Markov measures can also be reduced to the 1-step case by passing to the higher block shift. Let $\mu$ be an invariant $r$-step Markov measure, $r\geq 1$, with support a shift of finite type $X$. Let $A_X^{[r]} = \cL(X)\cap A^r$, which we view as an alphabet. We define the map $\beta_r\from A^\Z\to (A_X^{[r]})^\Z$ by
\begin{equation*}
    \beta_r(x)_i = x_{[i,i+r)}. 
\end{equation*}
The image $\beta_r(X)$ forms a shift space on $A_X^{[r]}$ denoted $X^{[r]}$, which is called the \emph{higher block shift}. The image measure of $\mu$ under $\beta_r$, denoted $\mu^{[r]}$, is an invariant 1-step Markov measure on this shift space. Moreover, $\mu$ is ergodic exactly when $\mu^{[r]}$ is. 

The density of a group language also carries over to the higher block shift, as follows. Given a morphism $\varphi\from A^*\to G$ onto a finite group, let $\varphi^{[r]}\from A_X^{[r]}\to G$ be the morphism defined by $\varphi(w) = \varphi(w_0)$ for $w = w_0\ldots w_{r-1}\in A_X^{[r]}$. To the group language $L = \varphi^{-1}(K)$, $K\subseteq G$, corresponds the group language $L^{[r]} = (\varphi^{[r]})^{-1}(K)$. It is straightforward to check that $\mu^{[r]}(L^{[r]}\cap (A_X^{[r]})^i) = \mu([L\cap A^i])$, and as a result:
\begin{equation*}
    \delta_\mu(L) = \delta_{\mu^{[r]}}(L^{[r]}).
\end{equation*}

\subsection{Strong irreducibility}\label{ss:strong}

We now relate $\varphi$-irreducibility to a condition introduced by Bufetov~\cite{Bufetov2003}, originally under the name \emph{strongly connected}. The term \emph{strictly irreducible} was used in Lummerzheim et al.~\cite{Lummerzheim2025}. We make a compromise between the two and use the term \emph{strongly irreducible}. We shall see below that if it holds, then \emph{all} skew products by morphisms onto finite groups are topologically transitive; on the other hand, when it fails, such skew products can behave in a variety of manners.

\begin{definition}\label{def:SI}
    Let $X$ be a shift on an alphabet $A$. For $r\geq 1$, we define the relation $\sim_r$ on $\cL(X)$ by $u \sim_r v$ if there exists $w \in A^r$ such that $wu, wv\in \cL(X)$. Since $\sim_r$ is symmetric, its transitive closure, which we denote by $\simeq_r$, is an equivalence relation. 

    An $r$-step shift of finite type is called \emph{strongly irreducible} if it is irreducible and the relation $\simeq_r$ is total, meaning that it has a single equivalence class.
\end{definition}

\begin{remark}
    It may happen that $\simeq_r$ has a single equivalence class even in the absence of irreducibility. This is the case in the 1-step shift of finite type $X$ on the alphabet $A=\{a,b\}$ consisting of sequences avoiding the factor $ba$.
\end{remark}

Recall the notation ${\prec} = {\prec_{1_G}}$, used in the proof of Theorem~\ref{t:transitive} for the relation defined by $u\prec v$ if and only if there exists $w\in\cL(X)$ such that $uwv\in\cL(X)$ and $\varphi(uw) = 1_G$.

\begin{proposition}\label{p:simeq-R}
    In an irreducible $r$-step SFT, $r\geq 1$, the following implication holds:
    \[
        u \simeq_r v \implies u\prec v.
    \] 
\end{proposition}

\begin{proof}
    Let $X$ be an irreducible $r$-step SFT. 
    We start by showing that given $u,v \in \cL(X)$ with $u \sim_r v$, there is a word $w$ such that $|w|\geq r$, $uwv \in \cL(X)$, and $\varphi(uw)=1_G$. 
    By definition of $\sim_r$, there is a word $z$ such that $|z|=r$ and $zu,zv \in \cL(X)$.
    Let $t\in A^*$ be a word such that $zut\in\cL(X)$ and $|ut|\geq r$. Since $X$ is irreducible, there exists a word $w_0\in A^*$ such that $utw_0zut \in \cL(X)$. Then $(utw_0z)^nu, z(utw_0z)^nv \in \cL(X)$ for all $n \geq 1$, since every factor of length $r$ in either of those words is a factor of either $zv$, $zut$, or $utw_0z$, all of which are in $\cL(X)$.
    Taking $n=|G|$, we find $\varphi((utw_0z)^n) = 1_G$. The word $w=(tw_0zu)^{n-1}tw_0z$ then satisfies $uw = (tw_0zu)^{n}$, and thus we get $uwv\in\cL(X)$ and $\varphi(uw)=1_G$, as required.

    Suppose next $u = u_0 \sim_r u_1 \sim_r \dots \sim_r u_{n-1} \sim_r u_n = v$. We want to find a word $w$ such that $uwv \in \cL(X)$ and $\varphi(uw)=1_G$. Using the above claim, there are words $w_i$ such that $u_iw_iu_{i+1} \in \cL(X)$, $\varphi(u_iw_i)=1_G$, and $|w_i| \geq r$.
    Then defining $w=w_0u_1w_1\ldots u_{n-1}w_{n-1}$ produces the requisite word such that $uwv\in\cL(X)$ and $\varphi(uw)=1_G$.
\end{proof}

We then deduce the following. It is our version of Bufetov's theorem~\cite[Theorem~5]{Bufetov2003} and its generalization by Lummerzheim et al.~\cite[Theorem~4.3]{Lummerzheim2025}, which we here specialize to the case of skew products with finite groups and morphisms, but generalize to the case of higher step shifts. 

\begin{theorem}\label{t:si}
    Let $X$ be an irreducible $r$-step SFT on $A$, $r \geq 1$. If $X$ is strongly irreducible, then it is $\varphi$-irreducible for every morphism $\varphi\from A^*\to G$ onto a finite group $G$. When $r=1$, the converse holds.
\end{theorem}

\begin{proof}
    For the first part of the statement, notice that when $X$ is strongly irreducible then by Proposition~\ref{p:simeq-R}, the relation ${\prec} = {\prec_{1_G}}$ must be total for every morphism $\varphi\from A^*\to G$ onto a finite group $G$, which is precisely the definition of $X$ being $\varphi$-irreducible.

    We now suppose that $r=1$. We will prove the contrapositive: assuming that $X$ is not strongly irreducible, then it is not $\varphi$-irreducible for some $\varphi$. The construction presented is essentially the same as the one used in the proof~\cite[Theorem~4.3]{Lummerzheim2025}.

    Fix an equivalence class $C$ of $\simeq_1$ restricted to $A\times A$. For each $a\in A$, observe that the set of right extensions of $a$ in $X$,
    \begin{equation*}
        \rext(a) = \{b\in A \mid ab\in\cL(X)\},
    \end{equation*}
    is contained in a class of $\simeq_1$; thus either $\rext(a)\subseteq C$ or  $\rext(a)\subseteq  A\setminus C$. Let
    \begin{equation*}
        B = \{ a\in C \mid \rext(a) \subseteq A\setminus C \} \cup \{ a\in A\setminus C \mid \rext(a) \subseteq C \}.
    \end{equation*}
    
    Define a morphism $\varphi\from A^*\to\Z/2\Z$ by $\varphi(a) = 1$ if $a\in B$ and $\varphi(a) = 0$ otherwise. Let $a\in C$ and $b\in A$; take a word $w$ such that $awb\in \cL(X)$. We claim that $\varphi(aw) = 0$ if and only if $b\in C$. Observe that this claim finishes the proof, as it implies that $a\not\prec b$ whenever $b\in A\setminus C$. 

    To establish the claim, we argue by induction on $|w|$. We first consider the case   $|w|=0$. We have to check that
    $\varphi(a)=0$ if and only if $b\in C$. But $\varphi(a)=0$ if and only if $a \not \in B$, which is in turn equivalent to 
    $\rext(a) \subseteq C$, i.e. $b \in C$. For the induction step, suppose that $w = w'c$, $c\in A$, with $aw'b \in \cL(X)$, and that the claim holds for $w'$. First assume that $\varphi(aw)=0$. If $c\in C$, then by induction $\varphi(aw')=0$ and $\varphi(aw) = \varphi(c) = 0$, and as $b\in \rext(c)$, we get $b\in C$. If on the other hand $c\in A\setminus C$, then $\varphi(aw') = 1$, since  the claim holds for $w'$ and $ c \in  A \setminus C$, hence $\varphi(aw) = 1+\varphi(c) = 0$, and so $\varphi(c) = 1$, and $c \not \in B$. As $b\in \rext(c)$, it again follows that $b\in C$. Conversely, assume that $b\in C$. If $c\in C$, then as $b\in \rext(c)$, we get $c \not \in B$ and $\varphi(c) = 0$, while by induction $\varphi(aw') = 0$; hence $\varphi(w) = 0+0 = 0$. Likewise, if $c\in A\setminus C$, then $c\in B$  and $\varphi(aw') = 1 = \varphi(c)$ and $\varphi(aw) = 1+1 = 0$.
\end{proof}

At this time we are unsure whether the converse holds when $r>1$. Nonetheless, the first part of the above combined with Theorem~\ref{t:transitive} yields the following immediate corollary.
\begin{corollary}\label{c:si}
    Let $X$ be a strongly irreducible (hence irreducible)   $r$-step SFT on $A$, $r\geq 1$. For every morphism $\varphi\from A^*\to G$, assumed to be  onto the finite group $G$, the skew product $G\skewprod _{\varphi} X$ is topologically transitive.
\end{corollary}

To end this section we give a few examples of shifts of finite type which are or are not strongly irreducible. But first we make the following simple observation, similar to Proposition~\ref{p:reduction}.
\begin{proposition} \label{prop:si}
    Let $X$ be an irreducible $r$-step SFT on $A$, $r \geq 1$. Then for $X$ to be strongly irreducible, it suffices that $u\simeq_rv$  for all pairs $u,v\in \cL(X)$ with $|u|=|v|=r$.
\end{proposition}

\begin{proof}
    Let $u'$ and $v'$ be words in $\cL(X)$ of arbitrary lengths. Choose some words $u,v$ such that $|u|=|v|=r$, $u$ is a prefix of $u'$ or vice versa, and $v$ is a prefix of $v'$ or vice versa. By assumption, we may find  $n$ words  $w_0,\ldots,w_{n-1}$ in $A^*$,  and $n+1$ words $t_0,\ldots,t_{n}$, also all of length $r$, such that $t_0 = u$, $t_n=v$, {and} 
    \begin{equation*}
        w_it_i, w_it_{i+1} \in \cL(X),\quad 0\leq i\leq n-1. 
    \end{equation*}
    Then notice that $w_0u'$ is also in $\cL(X)$; indeed this is obvious when $u'$ is a prefix of $u$, and otherwise it follows from the fact that $X$ is an $r$-step SFT. Likewise $w_{n-1}v'\in\cL(X)$, which shows that $u'\simeq_r v'$.
\end{proof}

\begin{example}\label{eg:golden-mean-2}
    Continuing from Example~\ref{eg:golden-mean-1}, we claim that the golden mean shift is in fact strongly irreducible.  Indeed, it suffices to notice that in this irreducible $1$-step SFT, for any letters $u$ and $v$ in $A$, one has  $au$ and $av$ in $\cL(X)$, hence $u \sim_1 v$, and by Proposition \ref{prop:si}, $X$ is indeed strongly irreducible.  
    As a result, by Corollary \ref{c:si}, all skew products of the golden mean shift with skewing functions given by morphisms onto finite groups are topologically transitive, including the one in Example~\ref{eg:golden-mean-1}.
\end{example}

\begin{example}\label{eg:not-si}
    Let $X$ be the 1-step SFT on $A = \{a,b,c\}$ formed by the sequences avoiding the words $ca,ab,bb,cc$. Then $X$ is irreducible but not strongly irreducible. Indeed, the relation $\simeq_1$ has two classes in $A\times A$, namely $\{b\}$ and $\{a,c\}$. Taking for instance $C = \{b\}$, the set $B$ in the second part of the proof of Theorem~\ref{t:si} equals~$\{b,c\}$. Consider  now the morphism $\varphi\from A^*\to\Z/2\Z$, $\varphi(a) =0$, $\varphi(b)=\varphi(c)=1$.  One has that $X$ is irreducible, whereas  $\Z/2\Z \skewprod_{\varphi}  X$ is not.  If $w\in A^*$ and $bwa \in  \cL(X)$, we must have  $\varphi(bw)=1$ and  therefore $X$ is not $\varphi$-irreducible. The shift $X$ and the skew product $\Z/2\Z\skewprod_{\varphi}  X$ are depicted in Figure~\ref{f:not-si}.
    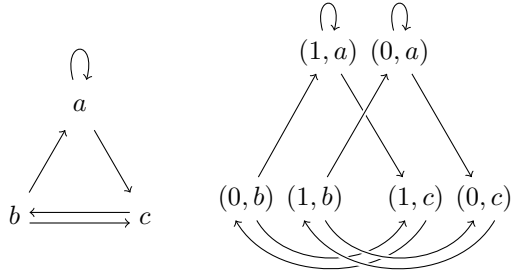
\begin{figure}
        \centering
        \begin{tabular}{cc}
            \begin{tabular}{c}
                \begin{tikzpicture}
                    \node (a) at (90:1)  {\strut$a$} ;
                    \node (b) at (210:1) {\strut$b$} ;
                    \node (c) at (-30:1) {\strut$c$} ;
                    \draw[->] (a) to (c) ;
                    \draw[->] (a) to [loop above] (a) ;
                    \draw[->] ([yshift=2pt]c.west) to ([yshift=2pt]b.east) ;
                    \draw[->] (b) to (a) ;
                    \draw[->] ([yshift=-2pt]b.east) to ([yshift=-2pt]c.west) ;
                \end{tikzpicture}
            \end{tabular}
            &
            \begin{tabular}{c}
                \begin{tikzpicture}[every node/.style = {inner sep =2pt}]
                    \node (0a) at (90: 1.3)    {\strut$(0,a)$} ;
                    \node (1b) at (210:1.3)    {\strut$(1,b)$} ;
                    \node (0c) at (-30:1.3)   {\strut$(0,c)$} ;
                    \node[left = 0 of 0a] (1a) {\strut$(1,a)$} ;
                    \node[left = 0 of 1b] (0b) {\strut$(0,b)$} ;
                    \node[left = 0 of 0c] (1c) {\strut$(1,c)$} ;
                    \draw[->] (0a) to (0c) ;
                    \draw[->] (1a) to (1c) ;
                    \draw[->] (0a) to [loop above] (0a) ;
                    \draw[->] (1a) to [loop above] (1a) ;
                    \draw[->,line width=2pt,white,bend left =60] ([xshift=4pt]1c.south) to ([xshift=-4pt]0b.south) ;
                    \draw[->,bend left =60] ([xshift=4pt]1c.south) to ([xshift=-4pt]0b.south) ;
                    \draw[->] (0b) to (1a) ;
                    \draw[line width=2pt,white,->] (1b) to (0a) ;
                    \draw[->] (1b) to (0a) ;
                    \draw[->,line width=2pt,white,bend right=60] ([xshift=4pt]0b.south) to ([xshift=-4pt]1c.south) ;
                    \draw[->,bend right=60] ([xshift=4pt]0b.south) to ([xshift=-4pt]1c.south) ;
                    \draw[->,line width=2pt,white,bend right=60] ([xshift=4pt]1b.south) to ([xshift=-4pt]0c.south) ;
                    \draw[->,bend right=60] ([xshift=4pt]1b.south) to ([xshift=-4pt]0c.south) ;
                    \draw[->,line width=2pt,white,bend left =60] ([xshift=4pt]0c.south) to ([xshift=-4pt]1b.south) ;
                    \draw[->,bend left =60] ([xshift=4pt]0c.south) to ([xshift=-4pt]1b.south) ;
                \end{tikzpicture}
            \end{tabular}
        \end{tabular}
        \caption{The SFT and skew product from Example~\ref{eg:not-si}}
        \label{f:not-si}
    \end{figure}

    In particular, it follows from Theorem~\ref{t:transitive} that no Markov measure fully supported on $\Z/2\Z\skewprod _{\varphi} X$ (viewed as an SFT as per Lemma~\ref{l:markov}) can be ergodic. This includes every measure of the form $\nu\times\mu$ where $\nu$ is the uniform probability measure on $\Z/2\Z$ and $\mu$ is a Markov measure fully supported on $X$.
\end{example}

\section{A characterization of minimality   via return words}
\label{s:bifix}

We now shift our focus away from shifts of finite type and towards \emph{minimal shift spaces}. We start by giving a first characterization of minimality for skew products (Theorem~\ref{t:minimalskew}), stated in terms of return words (Section~\ref{s:symbolic-dynamics}). Its proof, given in Section~\ref{ss:proof}, strongly relies on the  deep links between the skew products under consideration  and bifix codes, recalled in Section~\ref{ss:codes}. 
 \begin{theorem}\label{t:minimalskew}
    Let $X$ be a minimal shift space on $A$ and $\varphi\colon A^*\to G$ be a morphism onto a finite group $G$. The following conditions
    are equivalent.
    \begin{enumerate}
        \item For every $n>0$ and $x\in X$, $\{\varphi(x_{[0,m)}) \mid m\in\N, x_{[m, m+n)} = x_{[0,n)}\} = G$. \label{i:minimalskew-welldoc}
        \item The skew product $G\skewprod_{\varphi}  X$ is minimal. \label{i:minimalskew-minimal}
        \item For every $u\in\cL(X)$, the restriction of $\varphi$ to $\RR_X(u)^*$ is surjective. \label{i:minimalskew-return}
    \end{enumerate}
\end{theorem}
We will later provide a  further characterization of minimality  in Theorem~\ref{theo:minimal}, stated in terms of cobounding maps, in the flavor of Anzai's theorem on ergodicity of skew products~\cite{Anzai1951}. We will also provide an alternate proof of the equivalence between \ref{i:minimalskew-minimal} and \ref{i:minimalskew-return} (Remark~\ref{r:alternate}).

\begin{remark} 
    The condition \ref{i:minimalskew-welldoc} is reminiscent of the welldoc property (which stands for \emph{well distributed occurrences}) studied by Balkov\'{a} et al.~\cite{Balkova2016} in the context of pseudorandom number generators. For convenience, we recall that an infinite word $X$ is said to have the welldoc property when it satisfies the following condition where $\ab$ denotes the Parikh Abelianization map $A^*\to \Z^{|A|}$:
    \begin{equation*}
        \forall k\geq 2, \forall n\geq 0,\ \{ \ab(x_{[0,m)}) \bmod k \mid m\geq 0, x_{[m,m+n)}=x_{[0,n)} \} = (\Z/k\Z)^{|A|}.
    \end{equation*}
 
    It follows from the above theorem that the welldoc property is equivalent to all skew products of the form $(\Z/k\Z)^d\skewprod_{\varphi}  X$ being ergodic.
\end{remark}

\subsection{More on the theory of codes}\label{ss:codes}

We survey some basic notions from the theory of codes, some of which already appeared in Section~\ref{sec:Bernoulli}. Again we refer to \cite{BerstelPerrinReutenauer2009} for additional details.

Recall that a \emph{prefix code} is a subset $U\subseteq A^*$ where no word is a strict prefix of another, and likewise a \emph{suffix code} is a subset of $A^*$ where no word is a strict suffix of another. Let $X$ be a shift space. A prefix code $U\subseteq\cL(X)$ is called \emph{$X$-complete} if every word $w\in\cL(X)$ is either a prefix of a word in $U$ or has a prefix which is an element of $U$; replacing \emph{prefix} by \emph{suffix} in this definition, we obtain the corresponding notion of \emph{$X$-complete suffix code}.

\begin{example}
  Let $X$ be the Fibonacci shift considered in Section~\ref{ss:fibo}. Then the set $U=\{a,ba\}$ is an $X$-complete prefix code, though clearly not a suffix code.
\end{example}

We recall that a set $U$ is a \emph{bifix code} if it is both a prefix and a suffix code. It is $X$-complete if it is both an $X$-complete prefix code and an $X$-complete suffix code. When $X=A^\Z$, we simply say \emph{complete} instead of \emph{$A^\Z$-complete}, be it for prefix, suffix, or bifix codes.

Let $\varphi\colon A^*\to G$ be a morphism onto a finite group $G$. 
If $H\leq G$ is a subgroup, then the submonoid $M=\varphi^{-1}(H)$ of $A^*$ is  uniquely generated by a bifix code $Z$ (an argument is given in~\cite[p.~64]{BerstelPerrinReutenauer2009}), which is also a group code (as introduced in Section~\ref{sec:Bernoulli}). It is in fact a complete bifix code.

Let $Z$ be a group code. Let $P$ (resp. $S$) be the set of proper prefixes (resp. suffixes) of the words in $Z$. Note that $P$ (resp. $S$) is also the set of words which have no prefix (resp. suffix) in $Z$. The \emph{$Z$-degree} $d(u)$ of a word $u\in A^*$ is  then defined as any of the following numbers, which all coincide~\cite[Proposition 6.1.6]{BerstelPerrinReutenauer2009}:
\begin{enumerate}
    \item the number of suffixes of $u$ which are in $P$;
    \item the number of prefixes of $u$ which are in $S$;
    \item the number of \emph{$Z$-parses} of $u$, that is, the number of triples $(s,z,p)$ such that $u=szp$ with $s\in S$, $z\in Z^*$ and $p\in P$.
\end{enumerate}
It follows from the third definition that for every $u, v, w\in A^*$, the  $Z$-degrees of $u$ and $uvw$ satisfy
\begin{equation}
    d(v)\le d(uvw).\label{eq:degree}
\end{equation}

Indeed, if $(s,z,p)$ is a $Z$-parse of $v$, let $us=s'z'$ with $s'\in S$ and $z'\in Z^*$, and $pw=z''p'$ with $z''\in Z^*$ and $p'\in P$. Then $(s'z'zz'',p')$ is a $Z$-parse of $uvw$ which extends $(s,z,p)$. This shows that every parse of $v$ extends to a parse of $uvw$.

\begin{proposition}\label{p:degree}
    Let $\varphi\from A^*\to G$ be a morphism onto a finite group $G$, and $H$ be a subgroup of index $d$ in $G$. Let $Z$ be the group code such that $\varphi^{-1}(H) = Z^*$, and let $S$ be the set of proper suffixes of elements of $Z$. For every $p,q\in S$ such that $q$ is a proper prefix of $p$, $\varphi(p)H\neq\varphi(q)H$. In particular, every word has $Z$-degree at most $d$.
\end{proposition}

\begin{proof}
    Let $p = qx$ and assume by contradiction that $\varphi(p)H=\varphi(q)H$. Then it follows that $x$ is in $\varphi^{-1}(H)$, hence it is a nonempty word of $Z^*$, and thus it has a suffix in $Z$. Set $x=sr$ with $r\in Z$. Let $z\in Z$ be such that $z=tp$. Then,
    \begin{equation*}
        z=tp=tqx=tqsr,
    \end{equation*}
    so $r\in Z$ is a proper suffix of $z\in Z$, which contradicts the fact that $Z$ is bifix.
\end{proof}

Let $X$ be a minimal shift space and let $U=Z\cap \cL(X)$. The \emph{$X$-degree} of $U$, denoted $d_X(U)$, is the maximal value of the $Z$-degrees of all words in $\cL(X)$. The following is essentially a reformulation of \cite[Theorem~4.2.11]{BerstelDeFelicePerrinReutenauerRindone2012}; we include a proof for the convenience of the reader.
\begin{theorem}\label{theoremGroupCodes}
    Let $X$ be a minimal shift space on $A$ and $\varphi\from A^*\to G$ be a morphism onto a finite group $G$. Let $H$ be a subgroup of index $d$ in $G$ and $Z$ be the group code such that $\varphi^{-1}(H) = Z^*$. The set $U=Z\cap \cL(X)$ is a finite $X$-complete bifix code of $X$-degree $d_X(U)\leq d$.
\end{theorem}

The proof makes use of the following simple observation.
\begin{lemma}\label{lemmaMaximalDegree}
    Let $v$ be a word with maximal $Z$-degree. Then no word of $U$ can be of the form $uvw$ with $u$ or $w$ nonempty.
\end{lemma}

\begin{proof}
    Assume the contrary. Then we would have $d(uvw)>d(v)$ because every $Z$-parse of $v$ extends to a $Z$-parse of $uvw$, while the latter has the additional $Z$-parse $(\varepsilon,uvw,\varepsilon)$.
\end{proof}

\begin{proof}[Proof of Theorem~\ref{theoremGroupCodes}]
    Since $Z$ is a bifix code, the same is true for $U$. Consider a word $v\in\cL(X)$ of maximal $Z$-degree. By Lemma~\ref{lemmaMaximalDegree}, there is no word in $U$ containing $v$ as a strict factor, thus $U$ must be finite. Indeed, since $X$ is minimal, every long enough element of $\cL(X)$ is of the form $uvw$ with $u,w$ nonempty and thus cannot be a factor of a word in $U$; hence the length of the words in $U$ is bounded.

    Next, we show that $U$ is an $X$-complete prefix code. Consider a nonempty word $u\in \cL(X)$. Since $X$ is minimal, there is a word $w$ such that $uwv\in\cL(X)$. Set $u=au'$ where $a$ is a letter. Since $d(au'wv)\ge d(u'wv)\ge d(v)$ and since $d(v)$ is maximal; we have $d(au'wv)=d(u'wv)$. This forces the word $au'wv$ to have a prefix in $Z$, hence in $U$. Thus either $u$ has a prefix which is an element of $U$ or it is a prefix of an element of $U$. Hence $U$ is an $X$-complete prefix code. The proof that $U$ is an $X$-complete suffix code is similar.
\end{proof}

\begin{example}
    \label{eg:bifix-code}
    Let $X$ be the Fibonacci shift on $A=\{a,b\}$ (as in Section~\ref{ss:fibo}) and let $\varphi\colon A^*\to S_3$ be the morphism onto the symmetric group $S_3$ defined by $\varphi(a)=(1\,2),\varphi(b)=(1\,3)$, where permutations are written in usual cycle notation.  Note that $S_3$ is isomorphic to the group of matrices  $G(2)$ considered in Example \ref{ex:convergents}. Let $H$ be the subgroup of $G$ formed of the permutations fixing $1$. Let $Z$ be the group code such that $\varphi^{-1}(H)=Z^*$; it is given by $Z = ab^*a\cup ba^*b$. The elements of the infinite set $Z$ are represented in the tree found in Figure~\ref{f:bifix} as the labels of paths from the root to the leaves. The $X$-complete bifix code $U=Z\cap\cL(X)$, equal to $\{aa,aba,baab,bab\}$, is depicted in Figure~\ref{f:bifix}.
    \begin{figure}[hbt]
        \centering
        \tikzset{node/.style={circle,draw,minimum size=0.4cm,inner sep=0pt}}
        \begin{tikzpicture}
            \node[node] (1) at (0,0) {$1$};
            \node[node, above right = 1em and 2em of 1] (a) {$2$};
            \node[node, above right = 1em and 2em of a,double] (aa) {$1$};
            \node[node, right       = 3em of a] (ab) {$2$};
            \node[node, above right = 1em and 2em of ab,double] (aba) {$1$};
            \node[node, right       = 3em of ab] (abb) {$2$};
            \node[node, above right = 1em and 2em of abb] (abba) {$1$};
            \node[node, right       = 3em of abb,label=right:{$\cdots$}] (abbb) {$2$};
            \node[node, below right = 1em and 2em of 1] (b) {$3$};
            \node[node, right       = 3em of b] (ba) {$3$};
            \node[node, below right = 1em and 2em of b] (bb) {$1$};
            \node[node, right       = 3em of ba] (baa) {$3$};
            \node[node, below right = 1em and 2em of ba,double] (bab) {$1$};
            \node[node, right       = 3em of baa,label=right:{$\cdots$}] (baaa) {$3$};
            \node[node, below right = 1em and 2em of baa,double] (baab) {$1$};
      
            \draw[->,auto] (1) edge node {$a$} (a);
            \draw[->,auto] (1) edge [swap] node {$b$} (b);
            \draw[->,auto] (a) edge node {$a$} (aa);
            \draw[->,auto] (a) edge [swap] node {$b$} (ab);
            \draw[->,auto] (b) edge node {$a$} (ba);
            \draw[->,auto] (ab) edge node {$a$} (aba);
            \draw[->,auto] (ba) edge node {$a$} (baa);
            \draw[->,auto] (ba) edge [swap] node {$b$} (bab);
            \draw[->,auto] (baa) edge [swap] node {$b$} (baab);
            \draw[->,auto,densely dotted] (b) edge [swap] node {$b$} (bb);
            \draw[->,auto,densely dotted] (ab) edge [swap] node {$b$} (abb);
            \draw[->,auto,densely dotted] (abb) edge node {$a$} (abba);
            \draw[->,auto,densely dotted] (abb) edge [swap] node {$b$} (abbb);
            \draw[->,auto,densely dotted] (baa) edge node {$a$} (baaa);
    \end{tikzpicture}
    \caption{Representation of the $X$-complete bifix code $U$ of Example~\ref{eg:bifix-code}. Nodes are labeled by the image of 1 under the permutation given by the label of the path. Elements of $U$ correspond to paths ending in double-circled nodes}\label{f:bifix}
    \end{figure}
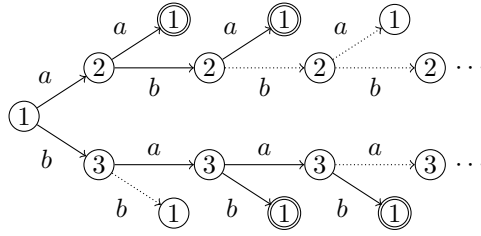
\end{example}

The next result uses ideas found in the proof of the main result of \cite{BertheDeFeliceDolceLeroyPerrinReutenauerRindone2015c}.

\begin{proposition}\label{propositionIndex}
    Let $X$ be a minimal shift space on $A$ and $\varphi\colon A^*\to G$ be a morphism onto a finite group $G$. Let $H\leq G$ be a subgroup of $G$, $Z$ be the group code such that $\varphi^{-1}(H)=Z^*$, and $U = Z\cap \cL(X)$. Assume that there exists some word $u\in\cL(X)$ with maximal $Z$-degree such that $\varphi(\RR_X(u)^*)H=G$. Then the $X$-degree of $U$ is $[G:H]$.
\end{proposition}

\begin{proof}
    Let $F_A$ be the free group $A$ and let $K$ be the subgroup of $F_A$ generated by~$U$. Let $u\in\cL(X)$ be of maximal $X$-degree. Let $Q$ be the set of prefixes of $u$ which are suffixes of some element of $U$ (so $|Q|$ is the $X$-degree of $U$). Observe that, for distinct elements $p$ and $q\in Q$, the cosets $pK$ and $qK$ are distinct by Proposition~\ref{p:degree}; indeed, since $\varphi(K)\leq H$, Proposition~\ref{p:degree} implies that $\varphi$ maps $pK$ and $qK$ inside disjoint right cosets of $H$ in $G$.

    Let us define
    \begin{equation*}
        V=\{v\in F_A \mid vQ\subseteq QK\}.
    \end{equation*}
    For every $v\in V$, the map $\pi(v)\colon p\mapsto q$ defined by $vp\in qK$ is a permutation. Indeed, let $vp,vp'\in qK$ for some $q\in Q$. Then
    $v^{-1}q$ is in $pK\cap p'K$, and thus $p=p'$.

    We claim that $V$ is a subgroup of $F_A$. First, let $v \in V$. Then for any $q\in Q$, since $\pi(v)$ is a permutation of $Q$, there is a $p\in Q$ such that $vp\in qK$. Then $v^{-1}q\in pK$. This shows that $v^{-1}\in V$. Next, if $v$ and $w\in V$, then $vwQ \subseteq vQK \subseteq QK$ and thus $vw\in V$. Since it is clear that $V$ contains the identity element, this proves the claim.

    Next, $V$ contains $\RR_X(u)$. Indeed, let $q \in Q$ and $y\in\RR_X(u)$. Since $q$ is a prefix of $u$, $yq$ is a prefix of $yu$, and since $yu$ is in $\cL(X)$ (by definition of $\RR_X(u)$), $yq$ is also in $\cL(X)$. Since, by Theorem~\ref{theoremGroupCodes}, $U$ is an $X$-complete bifix code, it is an $X$-complete suffix code. This implies that $yq$ is a suffix of a word in $U^*$, and thus there is a suffix $r$ of $U$ such that $yq \in rU^*$. We verify that the word $r$ is a suffix of $u$. Since $y \in \RR_X(u)$, there is a word $y'$ such that $yu = uy'$. Consequently, $r$ is a prefix of $uy'$, and in fact the word $r$ is a prefix of $u$. Indeed, one has $|r| \le |u|$, since otherwise $u$ would be in the set of internal factors of $U$, and this cannot be the case by Lemma~\ref{lemmaMaximalDegree} (recall that $u$ has maximal $Z$-degree). Thus we have $r \in Q$. Since $U^*\subseteq K$ and $r \in Q$, we have $yq \in QK$, hence $y \in V$.

    We finish the proof by showing that $|Q| = [V:K\cap V] \geq [G:H]$. The fact that $|Q| = [V:K\cap V]$ follows by noting that $\pi(v)$ is the identity on $Q$ if and only if $v$ is in $K$. For the remaining inequality, observe first that, as $K\leq\varphi^{-1}(H)$, the map $V/(K\cap V)\to G/H$, $v(K\cap V)\mapsto \varphi(v)H$ is well-defined. Moreover, the fact that $V$ contains $\RR_X(u)$ implies that $\varphi(V)H=G$, hence this map is surjective, thereby showing that $[V:K\cap V] \geq [G:H]$. Since the $X$-degree of $U$ is at most $[G:H]$ (by Theorem~\ref{theoremGroupCodes}) and equal to $|Q|$, this concludes the proof.
\end{proof}

The next example shows that there can be elements of maximal degree which fail the condition $\varphi(\RR_X(u)^*)H=G$. There are in fact even infinitely many of them.
\begin{example}\label{eg:thue-morse-1bis}
    Take $A = \{a,b\}$, $\varphi\from A\to\Z/2\Z$ defined by $\varphi(a) = 1, \varphi(b) =0$, and $X$ equal to the Thue--Morse shift, considered in Example~\ref{eg:thue-morse-1}. The bifix code $U = \varphi^{-1}(0)\cap\cL(X)$ is equal to $\{b, aa, aba, abba\}$. It has $X$-degree $2$, but $\RR_X(u)$ is contained in $\varphi^{-1}(0)$ for every $u$ sufficiently long. This last claim can be checked by finding one word $u$ with this property and then using the return preservation property~\cite{Berthe2023}. Nonetheless, the word $a$ has maximal degree and $\RR_X(a) = \{a,ab,abb\}$ satisfies $\varphi(\RR_X(a)^*) = \Z/2\Z$. 
\end{example}

Here is another example with the Thue--Morse shift where $d_X(U)<[G:H]$. Let $\varphi\colon \{a,b\}^*\to \Z/n\Z$ be a morphism with $\varphi(a)=1=-\varphi(b)$. Set $Z^*=\varphi^{-1}(0)$ and $U=Z\cap\cL(X)$. For every $n\geq 3$, we have $U=\{ab,ba,aabb,bbaa,aababb,bbabaa\}$, while $d_X(U)=3<n$ as soon as $n>3$.

\subsection{Proof of the characterization of minimality}\label{ss:proof}
We are now ready to give the proof of the main theorem of this section.
\begin{proof}[Proof of Theorem~\ref{t:minimalskew}]
    \ref{i:minimalskew-welldoc} implies \ref{i:minimalskew-minimal}. Let $Y\subseteq G\skewprod_{\varphi}  X$ be a minimal closed invariant subset. Fix $g\in G$ and $x\in X$; let us show that $(g,x)\in Y$.

    First, note that the projection $P_X(Y)=X$, since $X$ is minimal. Thus we may find $h$ such that $(h,x) \in Y$. By \ref{i:minimalskew-welldoc}, we can find a sequence of positive integers $(m_n)_{n\in\N}$ such that $\varphi(x_{[0,m_n)})=h^{-1}g$ and $x_{[0,n)} = x_{[m_n,m_n+n)}$. It follows that
	\begin{equation*}
        T_{\varphi}^{m_n}(h,x) \to (g,x),\quad\text{as } n\to\infty,
	\end{equation*}
	showing that $(g,x)\in Y$.

    \ref{i:minimalskew-minimal} implies \ref{i:minimalskew-return}. Fix a word $u\in\cL(X)$, an element $g\in G$, and consider the two clopen subsets $\{1_G\}\times[u]$ and $\{g\}\times[u]$. By minimality of $G\skewprod_{\varphi}  X$, there exists $n\in\Z$ such that:
    \begin{equation*}
        (\{1_G\}\times[u])\cap T_{\varphi}^{-n}(\{g\}\times[u])\neq\varnothing.
    \end{equation*}
    Choose a point $(1_G,x)$ in that intersection. Let $w = x_{[0,n)}$ if $n\geq 0$, and $w = x_{[n,0)}$ otherwise. Then we have $T_{\varphi}^n(1_G,x) = (\varphi(w),S^nx)\in\{g\}\times[u]$,  if $n \geq 0$, and $T_{\varphi}^n(1_G,x) = (\varphi(w)^{-1},S^nx)\in\{g\}\times[u]$, otherwise.  Thus ${w\in\RR_X(u)^*}$ and $g^{\pm1}\in\varphi(\RR_X(u)^*)$.

    \ref{i:minimalskew-return} implies \ref{i:minimalskew-welldoc}. Fix $x\in X$, $n\geq 0$ and $g\in G$. We want to prove that we can find $m\geq 0$ with $x_{[0,m)}\in\varphi^{-1}(g)$ and $x_{[0,n)} = x_{[m,m+n)}$. 

    Let $u = x_{[0,n)}$ and $Y = \DD_u(X)$ be the derivative shift of $X$ with respect to $u$ (whose definition may be found in~\cite[p.291]{DurandPerrin2021}). Thus $Y$ is a shift space on an alphabet $B = B_u$ with a bijection $\theta_u\from B\to\RR_X(u)$ such that $\theta_u(Y) = X$. Let $y\in Y$ be such that $\theta_u(y) = x$. Consider moreover the morphism $\psi = \varphi\circ\theta_u$, guaranteed to be onto by condition \ref{i:minimalskew-return}. Let $Z$ be the group code on $B$ such that $Z^* = \psi^{-1}(1_G)$, and $U = Z\cap\cL(Y)$. By Theorem~\ref{theoremGroupCodes}, the set $U$ is a $Y$-complete bifix code and thus it is, in particular, nonempty.

    Next, observe that the morphism $\psi$ also satisfies the condition of Proposition~\ref{propositionIndex}: the restriction of $\psi$ to every $\RR_{Y}(v)^*$ is onto. Indeed, for every $v\in \cL(Y)$, the morphism $\theta_v$ coding the return words to $v$ satisfies $\theta_u\circ\theta_v=\theta_w$ with $w=\phi_u(v)u$. Thus, we may apply Proposition~\ref{propositionIndex} to conclude that the $Y$-degree of $U$ is equal to~$|G|$. Since $U$ is a $Y$-complete prefix code, $y$ has arbitrary long prefixes in $U^*$. If such a prefix $v$ is long enough, it has $Z$-degree equal to the $Y$-degree of $U$, that is $|G|$.

    We claim that $v$ has a prefix $p$ such that $\psi(p)=g$. Indeed, since the $Z$-degree of $v$ is $|G|$, it has $|G|$ prefixes which belong to the set of proper suffixes of the elements of $Z$. By Lemma~\ref{lemmaMaximalDegree}, all these prefixes have distinct images by $\psi$, and thus one such prefix $p$ is mapped to $g$ by $\psi$. Letting $m = |\theta_u(p)|$, we find that $x_{[0,n)} = u = x_{[m,m+n)}$ and moreover $\varphi(x_{[0,m)}) = \psi(p) = g$.
\end{proof}

\begin{remark}
    We remark that there is also a direct proof that \ref{i:minimalskew-minimal} implies \ref{i:minimalskew-welldoc}, which sheds some light on these equivalences. Assume minimality of the skew product $G \skewprod_{\varphi}  X$ and fix $x \in X$, $n>0$. Let $w_n=x_{[0,n)}$ and $g \in G$ be given. We need to find a factor $v_n$ such that $x\in [w_nv_nw_n]_X$ and $\varphi(w_nv_n)=g$. By minimality of $G\skewprod_{\varphi}  X$, there is a sequence $(n_j)_{j\in\N}$ such that $T_{\varphi}^{n_j}(1_G,x)= (\varphi^{(n_j)}(x),S ^{n_j}x) \to (g,x)$. Thus, for large enough $j$, the sequences $S^{n_j} x$ and $x$ agree on  arbitrarily long prefixes $p_j$, and $x\in [p_j q_j p_j]_X $ for some factor $q_j$ such that  $|p_j q_j| = n_j.$ If $j$ is large enough, then the given prefix $w_n$ of $x$ is a prefix of $p_j$,  and (since $G$ is finite) $\varphi(p_jq_j)= \varphi^{(n_j)}(x) = g$.
\end{remark}

\begin{example}\label{eg:fibo-3}
    Consider once more the Fibonacci substitution $\sigma\colon a\mapsto ab,b\mapsto a$ and the Fibonacci shift $X$. Let $\varphi\colon A^*\to \Z/2\Z$ be the morphism $\varphi\colon a\mapsto 1, b\mapsto 0$. We saw in Section~\ref{ss:fibo} that the skew product $G\skewprod _{\varphi} X$ is minimal and uniquely ergodic. We provide here another argument for this which uses bifix codes. 

    First let $Z$ be the bifix code such that $Z^* = \varphi^{-1}(0)$ and $U = Z\cap\cL(X)$. We find $U = \{aa, aba, b\}$. Consider an alphabet $B = \{u,v,w\}$ and define a morphism $\phi\colon B^*\to A^*$ by 
    \begin{equation*}
        \phi\from u\mapsto aa, v\mapsto aba, w\mapsto b.
    \end{equation*}

    By construction, the image of $\phi$ has the same intersection with $\cL(X)$ as the submonoid $\varphi^{-1}(0) = Z^*$. Moreover, we have $\sigma^3\circ\phi=\phi\circ\tau$ where $\tau$ is the morphism 
    \begin{equation*}
        \tau\colon u\mapsto vvwuw, v\mapsto vvwuwuw, w\mapsto v.
    \end{equation*}

    Letting $Y$ be the shift space generated by $\tau$, the skew product $G\skewprod_{\varphi}  X$ can be identified with the tower $\widehat{Y}$ relative to the function $f(x)=|\phi(x_0)|$ (see \cite[Section~1.1.3]{DurandPerrin2021} for details). Since $\tau$ is primitive, $Y$ is minimal and uniquely ergodic by Michel's theorem, and so is $\widehat{Y}$. 
\end{example}

\subsection{Relation with average length}\label{ss:averag}

We finish the section by coming back to the notion of average length discussed in Section~\ref{sec:Bernoulli}. Let $X$ be a minimal shift space on $A$ and $\varphi\colon A^*\to G$ be a morphism onto a finite group $G$. Let $Z$ be the group code such that $\varphi^{-1}(1_G)=Z^*$, and $U = Z\cap \cL(X)$, the prefix code considered in Proposition~\ref{propositionIndex}. Recall the formula for the average length of $U$ relative to $\mu$:
\begin{equation*}
    \ell(U)=\sum_{u\in U}|u|\,\mu([u]).
\end{equation*}

\begin{example}
    The bifix code $U$ in  Example~\ref{eg:fibo-3} has average length \[\ell(U)=\mu([\varepsilon])+\mu([a])+\mu([ab])=1+\frac{1}{\lambda}+\frac{1}{\lambda^2}=2\] where $\lambda$ is the golden ratio (cf. Figure~\ref{figureFibonacci}).
\end{example} 

By \eqref{eqAverage}, one has also
\begin{equation*}
    \ell(U)=\sum_{w\in P}\mu([w]),
\end{equation*}
where $P$ is the set of proper prefixes of some element of $U$. Moreover, under the hypotheses of Theorem~\ref{t:first-main}, i.e. ergodicity of the appropriate product measure on the skew product $G\skewprod_{\varphi}  X$, where $L=Z^*=\varphi^{-1}(1_G)$, the following equalities hold, as an  extension of  Proposition~\ref{prop:bernoulli}:
\begin{equation*}
    \delta_\mu(L)= \frac{1}{|G|} = \frac{1}{\ell(U)}.
\end{equation*}

Indeed we will see that ergodicity of the skew product entails minimality (Corollary~\ref{cor:minerg}), which  implies that  the restriction of $\varphi$ on the sets $\RR_X(u)^*$ is onto by Theorem~\ref{t:minimalskew}. Then, by Proposition~\ref{propositionIndex},
 the degree  $d_X(U)$ of $U$ is $|G|$. We thus are in the scope of \cite[Corollary~4.3.8]{BerstelDeFelicePerrinReutenauerRindone2012}, which states that $d_X(U)=\ell(U)$, and so $\ell(U)=|G|$. 

\section{Minimal subsets and modular cobounding maps}
\label{s:cobounding}

We now generalize some results from the previous sections, in particular Theorem~\ref{t:first-main} which  allows a simple expression of the density under the assumption of ergodicity,  and Theorem~\ref{t:minimalskew} which characterizes minimality (referring here to the action of the  skew product $T_{\varphi}$ as defined in \eqref{eq:Tphi}, since  the action of the shift map $S$ on $X$ is already assumed to be minimal). 
We rely on the key notion of \emph{modular cobounding map} (Definition~\ref{d:cobounding}), closely related with the notion of coboundary. It allows in particular  to obtain a further  characterization  of minimality (Theorem~\ref{theo:minimal}) inspired by   Anzai's theorem on ergodicity of skew products~\cite{Anzai1951}.

We prove first  that the minimal skew products under consideration (that is, skew products $G\skewprod _{\varphi}X$  with $\varphi$ being a morphism onto a finite group $G$)  are  finite disjoint unions of their minimal closed invariant subsets, all of which have the same measure (Proposition~\ref{prop:finitecompo}). We conclude 
  that ergodicity implies minimality (Corollary~\ref{cor:minerg}). Moreover, Proposition~\ref{prop:minimal-ergodic}  together with Corollary~\ref{c:mod1}  provide sufficient conditions  for ergodicity  on every minimal closed invariant subset of $G \skewprod _{\varphi} X$. We prove our third main result, Theorem~\ref{t:third-main}, in  Section~\ref{ss:gencase} and conclude with examples in Section~\ref{ss:exa}.

\subsection{Modular coboundaries}\label{ss:cobounding}

We start by examining the general structure of skew products in terms of minimal closed invariant subsets. 

\begin{proposition}\label{prop:finitecompo}
     Let $X$ be a minimal shift space on $A$ with an invariant measure $\mu$ and $\varphi\colon A^*\to G$ be a morphism onto a finite group $G$ with uniform probability measure $\nu$. The skew product $G\skewprod _{\varphi} X$ is a disjoint union of its minimal closed invariant subsets, which are finite in number and equal in $\nu\times\mu$-measure.
\end{proposition}

\begin{proof}
    Let $G$ act on the left of $G\skewprod _{\varphi} X$ by $g(h,x) = (gh,x)$. This action commutes with the transformation $T_{\varphi}$ of the skew product (defined in \eqref{eq:Tphi}) and  is continuous, so $G$ acts by automorphisms. In particular, it permutes the set of minimal closed invariant subsets.

    Let us fix a pair $(g,x)\in G\times X$ and show that it is contained in some minimal closed invariant subset of $G\skewprod_{\varphi}  X$. By Zorn's lemma, there exists at least one minimal closed invariant subset, say $Y\subseteq G\skewprod_{\varphi}  X$. Its projection on $X$ is also a minimal closed invariant subset, thus by minimality of $X$, it must be $X$ itself. It follows that $(h,x)\in Y$ for some $h\in G$, and then $gh^{-1}Y$ is a minimal closed invariant subset containing $(g,x)$. 

    This also shows that $G$ acts transitively on the set of all minimal closed invariant subsets, thus it must be finite and with cardinality dividing $|G|$. Moreover, this action by $G$ is measure-preserving, as is easily checked on rectangular sets: indeed, for every measurable sets $E\subseteq G$ and $F\subseteq X$, we have 
    \begin{equation*}
        (\nu\times\mu)(g(E\times F)) = (\nu\times\mu)(gE\times F) = (\nu\times\mu)(E\times F).
    \end{equation*}
    Thus all minimal closed invariant subsets must have the same measure.
\end{proof}

The above result has the following straightforward consequences.
\begin{corollary}\label{cor:minerg}
   Let $X$ be a minimal shift space on $A$  with an invariant measure $\mu$ and $\varphi\colon A^*\to G$ be a morphism onto a finite group $G$ with uniform probability measure $\nu$.
   \begin{enumerate}
       \item If $\nu\times\mu$ restricts to an ergodic measure on a closed invariant subset $Y$, then $Y$ must be minimal.
       \item If $\nu\times\mu$ restricts to an ergodic measure on \emph{any} minimal closed invariant subset, then it must restrict to an ergodic measure on \emph{each} of them.
       \item If $\nu\times\mu$ is ergodic, then $G\skewprod _{\varphi} X$ must be minimal.
   \end{enumerate}
\end{corollary}

We now proceed to describe the minimal closed invariant subsets of $G\skewprod_{\varphi}  X$, using the following key notion inspired by \cite{LemanczykMentezen2002}.

\begin{definition}\label{d:cobounding}
    Let $(X,S)$ be a shift space on $A$ and $\varphi\colon A^*\to G$ be a morphism onto a finite group $G$. A \emph{cobounding map} for $\varphi$ on $X$ is a continuous map $\alpha\from X\to H\backslash G$ to the set of right cosets of a subgroup $H\leq G$ such that, for all $x\in X$,  the following relates 
    the images of $\alpha $ on $x$ and on its shifted image $Sx$:
    \begin{equation*}
        \alpha(Sx) = \alpha(x)\varphi(x_0).
    \end{equation*}
\end{definition}

We also say that $\alpha$ is a cobounding map \emph{mod $H$}. Observe that if $\alpha$ is a cobounding map, then $$\alpha(S^nx) = \alpha(x)\varphi^{(n)}(x)$$ for all $n\in\Z$, with $\varphi^{(n)}$ as in~\eqref{eq:fn}. 

As the name suggests, this definition is related to  the cohomological  equations developed   in ergodic theory. These equations have a rich history, as evidenced for instance by \cite{Anzai1951,Veech1969,Veech1975,Zimmer1976,Conze1976,Schmidt1977,LemanczykMentezen2002}. In particular, we may view a cobounding map as a \enquote{certificate} that (the cocycle defined by) $\varphi$ is a \emph{coboundary} mod $H$. The next proposition clarifies the link between cobounding maps and closed invariant subsets. It is a special case of a result of Lema\'{n}czyk and Mentzen~\cite[Proposition~2.1]{LemanczykMentezen2002}.

\begin{proposition}
    \label{p:minimal-cobounding}
    Let $X$ be a minimal shift space on $A$ with an invariant measure $\mu$ and $\varphi\colon A^*\to G$ be a morphism onto a finite group  $G$ with uniform probability measure $\nu$. Let $\alpha$ be a cobounding map mod $H$ for $\varphi$, and $Y_\alpha = \{ (g,x) \mid \alpha(x) = Hg \}$.
    \begin{enumerate}
        \item\label{i:cobound-to-inv}  The set $Y_\alpha$ is a closed invariant subset of $G\skewprod_{\varphi}  X$ of $\nu\times\mu$-measure $1/[G:H]$.
        \item\label{i:inv-to-cobound} If $Y$ is a minimal closed invariant subset of $G\skewprod_{\varphi}  X$, then $Y = Y_\alpha$ for some cobounding map $\alpha$.
    \end{enumerate}
\end{proposition}

\begin{proof}
    \ref{i:cobound-to-inv}. Note that $Y_\alpha = \bigcup_{Hg\in H\backslash G} Hg\times\alpha^{-1}(Hg)$. Since $\alpha$ is continuous, each $\alpha^{-1}(Hg)$ is closed, and so is $Y_\alpha$. Fix a pair $(g,x)\in Y_\alpha$, which means that $g\in Hg=\alpha(x)$. Hence $g\varphi^{(n)}(x)\in \alpha(x)\varphi^{(n)}(x) = \alpha(S^nx)$, and $T_{\varphi}^n(g,x) = (g\varphi^{(n)}(x), S^nx) \in Y_\alpha$. Thus $Y_\alpha$ is invariant. It has measure
    \begin{align*}
        (\nu\times\mu)(Y_\alpha) &= \sum_{Hg\in H\backslash G}(\nu\times\mu)(Hg\times\alpha^{-1}(Hg)) \\
                                 &= \frac{|H|}{|G|}\sum_{Hg\in H\backslash G}\mu(\alpha^{-1}(Hg)) = \frac{1}{[G:H]}.
    \end{align*}

    \ref{i:inv-to-cobound}. Let $Y$ be a minimal closed invariant subset of $G\skewprod_{\varphi}  X$. Consider the subgroup $H = \{ h\in G \mid hY = Y \}$, and for each $x\in X$ let 
    \begin{equation*}
        \alpha(x) = \{ h\in H \mid (h,x)\in Y\}.
    \end{equation*}
Since $X$ is minimal, the  restriction  to $Y$ of the  projection  $P_X \colon  G\times X  \to X$ is surjective, and therefore $\alpha$ is defined everywhere on $X$.
    Fix an element $g\in\alpha(x)$. We claim that $\alpha(x) = Hg$. On the one hand, it is clear that for $h\in H$, $(hg,x)\in hY = Y$, thus $hg\in\alpha(x)$. This shows that $Hg\subseteq\alpha(x)$. On the other hand, for $k\in\alpha(x)$ we have that $(g,x) = (gk^{-1}k,x)\in Y\cap gk^{-1}Y$. But note that $gk^{-1}Y$ is also a closed invariant subset of $G\skewprod_{\varphi}  X$, since $Y$ is and $T_{\varphi}$ commutes with the left action of $G$ on $G\skewprod_{\varphi}  X$. This shows that $Y\cap gk^{-1}Y$ is a non-empty closed invariant subset of $Y$, and thus $Y\cap gk^{-1}Y = Y$ by minimality of $Y$. 

    For every $(h,y)\in Y$, we conclude that $h = kg^{-1}h'$ for some $h'\in G$ such that $(h',y)\in Y$, which implies that $kg^{-1}(h,y) = (h',y)\in Y$. This means that $kg^{-1}\in H$ and thus $k = kg^{-1}g\in Hg$, as claimed. In particular, this shows that $\alpha$ is a map $X\to H\backslash G$ and that $Y = Y_\alpha$. It remains to show that $\alpha$ is a cobounding map. 

    To establish continuity, it suffices to show that $\alpha^{-1}(Hg)$ is closed for each $g\in G$. But observe that $\alpha^{-1}(Hg)$ can be written in terms of the two component projections $P_G\from G\skewprod_{\varphi}  X\to G$ and $P_X\from G\skewprod_{\varphi}  X\to X$:
    \begin{equation*}
        \alpha^{-1}(Hg) = \{ x\in X \mid g\in\alpha(x) \} = \{ x\in X \mid (g,x)\in Y \} = P_X(P_G^{-1}(g)\cap Y).
    \end{equation*}
    Since $Y$ is a closed subspace, $P_G$ is continuous, and $P_X$ is a closed map, we conclude that $\alpha^{-1}(Hg)$ is indeed closed.

    We end the proof by showing that $\alpha(x)\varphi^{(n)}(x) = \alpha(S^nx)$ for every $x\in X$ and $n\in\Z$. Take first $h\in\alpha(x)$, so $(h,x)\in Y$. Since $Y$ is invariant,
    \begin{equation*}
        (h\varphi^{(n)}(x), S^nx) = T_{\varphi}^n(h,x) \in Y,
    \end{equation*}
    and therefore $h\varphi^{(n)}(x)\in\alpha(S^nx)$. This shows that $\alpha(x)\varphi^{(n)}(x) \subseteq \alpha(S^nx)$. As this inclusion holds for all $x\in X$ and $n\in\Z$, it also holds for $x$ replaced by $S^nx$ and $n$ replaced by  $-n$, which yields:
    \begin{equation*}
        \alpha(S^nx)\varphi^{(-n)}(S^nx) \subseteq \alpha(S^{-n}S^nx) = \alpha(x).
    \end{equation*}
    Since $\varphi^{(-n)}(S^nx) = \varphi( (S^nx)_{[-n,0)} )^{-1} = (\varphi^{(n)}(x))^{-1}$ by \eqref{eq:fn}, this also proves the other inclusion. \end{proof}

\subsection{A characterization of minimality via cobounding maps}\label{ss:charac}

There is always at least one cobounding map, namely the constant map $X\to G\backslash G$, which we call the \emph{trivial cobounding map}. The corresponding closed invariant subset is then the whole skew product. It is immediately apparent that the existence of a non-trivial cobounding map thus forbids minimality of the skew product. In fact, we have the following consequence of Proposition~\ref{p:minimal-cobounding}, which is reminiscent of Anzai's theorem on ergodicity of skew products~\cite{Anzai1951} (see also \cite[Chapter 2, Theorem 4.8]{Petersen1983}). 

\begin{theorem}\label{theo:minimal}
    Let $X$ be a minimal shift space on $A$ and $\varphi\colon A^*\to G$ be a morphism onto a finite group $G$. The skew product $X\skewprod_{\varphi}  G$ is minimal if and only if there exists no non-trivial cobounding maps for $\varphi$.
\end{theorem}

\begin{proof}
    We prove the contrapositive implications. If there is a non-trivial cobounding map $\alpha\from X\to H\backslash G$, then the subset $Y_\alpha$ from Proposition~\ref{p:minimal-cobounding} is a proper, non-empty, closed invariant subset of $G\skewprod_{\varphi}  X$. Conversely if $G\skewprod_{\varphi}  X$ is not minimal then it has (by Zorn's lemma) a proper minimal closed invariant subset $Y$, which must then be of the form $Y = Y_\alpha$ for some cobounding map $\alpha\from X\to H\backslash G$, by \ref{i:inv-to-cobound} of Proposition~\ref{p:minimal-cobounding}. Since $Y_\alpha$ is proper, $\alpha$ must be non-trivial.
\end{proof}

Given two cobounding maps $\alpha$ and $\beta$, respectively mod $H$ and $K$, write $\alpha\leq\beta$ if $H\leq K$ and $\alpha^{-1}(Hg)\subseteq\beta^{-1}(Kg)$ for all $g\in G$. This gives a partial order on cobounding maps, which corresponds directly to the ordering of the corresponding closed invariant subsets under inclusion.

\begin{proposition}
    Let $X$ be a minimal shift space on $A$ and $\varphi\colon A^*\to G$ be a morphism onto a finite group $G$. For any two cobounding maps $\alpha$ and $\beta$, $Y_\alpha\subseteq Y_\beta$ if and only if $\alpha\leq\beta$. Therefore, the minimal closed invariant subsets of $G\skewprod_{\varphi}  X$ correspond to the minimal cobounding maps under $\leq$.
\end{proposition}

\begin{proof}
    Assume $\alpha\from X\to H\backslash G$ and $\beta\from X\to K\backslash G$. If $\alpha\leq\beta$, then $Hg\subseteq Kg$ and $\alpha^{-1}(Hg)\subseteq\beta^{-1}(Kg)$ for every $g\in G$, thus
    \begin{equation*}
        Y_\alpha = \bigcup_{Hg\in H\backslash G}Hg\times\alpha^{-1}(Hg) \subseteq \bigcup_{Kg\in K\backslash G}Kg\times\beta^{-1}(Kg) = Y_\beta.
    \end{equation*}

    Conversely, assume that $Y_\alpha\subseteq Y_\beta$ and fix $Hg\in H\backslash G$. Then $Hg\times\alpha^{-1}(Hg)\subseteq Kg'\times\beta^{-1}(Kg')$ for some $g'\in G$. In particular, $g\in Kg'$ so we may assume $g = g'$. We then deduce that $H\subseteq K$ and $\alpha^{-1}(Hg)\subseteq\beta^{-1}(Kg)$.
\end{proof}

The left action of $G$ on $G\skewprod _{\varphi} X$ corresponds to the left action on cobounding maps given by $(g\alpha)(x) = g(\alpha(x))$, where $g\alpha$ is viewed as a cobounding map mod $H^g = gHg^{-1}$. This cobounding map is such that $g\alpha(x) = H^gh \iff \alpha(x) = Hg^{-1}h$, hence $X_{g\alpha} = gX_\alpha$. This shows that the passage from minimal closed invariant subsets to minimal cobounding maps preserves the left action of $G$. In particular, $G$ acts transitively on the set of minimal cobounding maps.

\begin{corollary} \label{cor:number}
    Let $X$ be a minimal shift space on $A$ and $\varphi\colon A^*\to G$ be a morphism onto a finite group $G$. Let $H$  be a subgroup of $G$  such that   there exists  a minimal cobounding map  $\alpha\from X\to H\backslash G$.  Then, for any other minimal cobounding map $\beta\from X\to K\backslash G$, the group   $K$  is  a conjugate of $H$.
  Consequently, the number of minimal closed invariant subsets of the skew product $G\skewprod _{\varphi} X$ equals $[G:H]$.
\end{corollary}

\subsection{Cobouding maps and return words}
We have established  in Theorem~\ref{t:minimalskew} a characterization of minimality in terms of return words. Without surprise, we also find links between cobounding maps and return words. This makes the relationship between \ref{i:minimalskew-minimal} and \ref{i:minimalskew-return} in Theorem~\ref{t:minimalskew} more transparent (see Remark~\ref{r:alternate} below).

\begin{proposition}\label{p:cobounding-return}
    Let $X$ be a minimal shift space on $A$ and $\varphi\colon A^*\to G$ be a morphism onto a finite group $G$. Let $H$ be a subgroup of $G$.
    \begin{enumerate}
        \item If a cobounding map $\alpha\from X\to H\backslash G$ takes constant value $Hg$ on a cylinder $[u]_X$, then $\varphi(\RR_X(u))\subseteq g^{-1}Hg$.
            \label{i:return}
        \item If a word $u\in\cL(X)$ satisfies $\varphi(\RR_X(u))\subseteq H$, then there exists a cobounding map $\alpha\from X\to H\backslash G$ which takes constant value $H$ on $[u]_X$.
            \label{i:existence}
        \item A cobounding map $\alpha\from X\to H\backslash G$ is minimal if and only if $\varphi(\RR_X(u))$ generates $g^{-1}Hg$ whenever $u$ is such that $\alpha$ takes constant value $Hg$ on $[u]_X$.
            \label{i:minimal}
    \end{enumerate}
\end{proposition}

\begin{proof}
    \ref{i:return}. Consider $w\in\RR_X(u)^*$. Taking $x\in[wu]_X\subseteq [u]_X$, we find:
    \begin{equation*}
        \alpha(x) = \alpha(S^{|w|}x) = \alpha(x)\varphi(w) = Hg\varphi(w),
    \end{equation*}
    while the stabilizer of $Hg$ under the right action of $G$ on $H\backslash G$ is exactly $g^{-1}Hg$. 

    \ref{i:existence}. Let $u\in\cL(X)$ be such that $\varphi(\RR_X(u)) \leq H$. For $x\in X$, let 
    \begin{equation*}
        C_x = \{j\geq 0 \mid x_{[j,j+|u|)}=u\}.
    \end{equation*}
    Observe that under the assumption that $\varphi(\RR_X(u)) \subseteq H$, the value of $H\varphi(x_{[0,j)})^{-1}$ is identical for every $j\in C_x$. Hence, we may define a map $\alpha\from X\to H\backslash G$ by 
    \begin{equation*}
        \alpha(x) = H\varphi(x_{[0,j)})^{-1},\quad\text{where $j\in C_x$}.
    \end{equation*}

    Since $X$ is a minimal shift space, there exists a constant $m>0$ independent of $x$ such that $\min(C_x)<m$; hence $\alpha$ is continuous, since its value is determined by the first $m$ letters. Moreover, fixing $x\in X$ and $j\in C_x$ with $j\geq 1$, we find that
    \begin{equation*}
        \alpha(Sx) = H\varphi(x_{[1,j)})^{-1} = H\varphi(x_{[0,j)})^{-1}\varphi(x_0) = \alpha(x)\varphi(x_0).
    \end{equation*}

    \ref{i:minimal}. Assume first that $\alpha$ is minimal. Fix $h\in H$ and suppose that $\alpha$ takes constant value $Hg$ on some cylinder $[u]$. We need to show that $g^{-1}hg$ belongs to $\langle\varphi(\RR_X(u))\rangle$. Since $\{g\}\times[u]$ and $\{hg\}\times[u]$ are two non-empty clopen subsets of the minimal closed invariant subset $Y_\alpha$, we may find $k\in\Z$ with $(\{g\}\times[u])\cap T_{\varphi}^{-k}(\{hg\}\times[u])\neq\varnothing$. Choose $x$ in that intersection and let $w = x_{[0,k)}$ if $k\geq 0$ and $w = x_{[k,0)}$ otherwise. It follows that:
    \begin{equation*}
    (hg,x) = T_{\varphi}^k(g,x) = 
    \begin{cases}
        (g\varphi(w), S^kx) & \text{if } k \geq  0,\\
        (g\varphi(w)^{-1}, S^kx), & \text{otherwise}.
    \end{cases}
    \end{equation*}
    hence $\varphi(w) = g^{-1}h^{\pm 1}g$ (according to  whether $k\geq 0$ or $k< 0$). But $w$ is a concatenation of elements of $\RR_X(u)$, thus $g^{-1}hg$ belongs to the subgroup of $G$ generated by $\varphi(\RR_X(u))$.

    To prove the converse, we consider a cobounding map $\beta\from X\to K\backslash G$ such that $\beta\leq\alpha$. Let $u$ be a word such that both $\alpha$ and $\beta$ are constant on $[u]$, say with $\alpha([u]) = Hg$ and $\beta([u]) = Kh$. Note that $h\in Kh\subseteq Hg$, so we may assume that $h=g$. By part \ref{i:return} of the statement, $\varphi(\RR_X(u))$ generates a subgroup of $g^{-1}Kg$, while it also generates $g^{-1}Hg$ by our assumption. Thus $H=K$ and $\beta = \alpha$ on $[u]$. As we may partition $X$ into a union of such cylinders, we get $\alpha=\beta$, thus showing that $\alpha$ is minimal.
\end{proof}

\begin{remark}\label{r:alternate}
    From this result combined with Theorem~\ref{theo:minimal}, we deduce an alternate proof for the equivalence between \ref{i:minimalskew-minimal} and \ref{i:minimalskew-return} in Theorem~\ref{t:minimalskew}. Indeed, on the one hand, if the skew product is not minimal, then Theorem~\ref{theo:minimal} states that there exists a non-trivial cobounding map; hence it follows from \ref{i:return} of Proposition~\ref{p:cobounding-return} that for a sufficiently long word $u$, all return words in $\RR_X(u)$ are mapped inside some proper subgroup of $G$. On the other hand, if the image of some return set $\RR_X(u)$ fails to generate $G$, then by \ref{i:existence} of Proposition~\ref{p:cobounding-return}, there exists a non-trivial cobounding map, hence the skew product cannot be minimal by Theorem~\ref{theo:minimal}.
\end{remark}

Roughly speaking, the smaller the subgroup, the more restrictive the coboundary condition, hence the most stringent cobounding maps are the cobounding maps mod 1,
where the notation mod 1 refers to working modulo the trivial subgroup.
One important fact is that having such cobounding maps turns out to be sufficient for the measure $\nu\times\mu$ to be ergodic on each of the minimal closed invariant subsets of $G\skewprod _{\varphi} X$.
\begin{proposition} \label{prop:minimal-ergodic}
    Let $X$ be a minimal shift space on $A$ with an ergodic measure $\mu$ and $\varphi\colon A^*\to G$ a morphism onto a finite group $G$ with uniform probability measure $\nu$. If  there exists  a cobounding map $\alpha\from X\to G$ mod 1, then the product measure $\nu\times\mu$ is ergodic on $Y_\alpha$.
\end{proposition}

\begin{proof}
    The map $\gamma\from X\to Y_\alpha$, $\gamma(x) = (\alpha(x),x)$, is a homeomorphism which intertwines $S$ and $T_{\varphi}$ and satisfies $\mu(E)/|G| = (\nu\times\mu)(\gamma(E))$ for every measurable set $E\subseteq X$. Thus $(Y_\alpha,T_{\varphi},\nu\times\mu)$ is measure-theoretically isomorphic to $(X,S,\mu)$, and since the latter is ergodic, so is the former.
\end{proof}

We moreover observe that cobounding maps mod 1 are minimal by Proposition~\ref{p:cobounding-return} \ref{i:minimal}. In this special case, Proposition~\ref{p:cobounding-return} also yields the following.

\begin{corollary}\label{c:mod1}
    Let $X$ be a minimal shift space on $A$ with and $\varphi\colon A^*\to G$ a morphism onto a finite group $G$. The following conditions are equivalent:
    \begin{enumerate}
        \item $\varphi$ has a cobounding map mod 1 on $X$. 
            \label{i:mod1-cobounding}
        \item $\varphi(\RR_X(u)) = 1$ for every long enough $u\in\cL(X)$. 
            \label{i:mod1-return-all}
        \item $\varphi(\RR_X(u)) = 1$ for some word $u\in\cL(X)$. 
            \label{i:mod1-return-some}
    \end{enumerate}
\end{corollary}

\begin{proof}
    That \ref{i:mod1-cobounding} implies \ref{i:mod1-return-all} follows from Proposition~\ref{p:cobounding-return} \ref{i:return}; \ref{i:mod1-return-all} implies \ref{i:mod1-return-some} is trivial; \ref{i:mod1-return-some} implies \ref{i:mod1-cobounding} follows from Proposition~\ref{p:cobounding-return} \ref{i:existence}.
\end{proof}

\subsection{A formula for density in terms of cobounding maps} \label{ss:gencase}

The next theorem is our third main result. It gives a simple closed form for the density in terms of any minimal cobounding map under suitable assumptions. It generalizes Corollary~\ref{c:equidistribution} within the setting of minimal shift spaces.

\thirdmain*

Note that when $G\skewprod_{\varphi}  X$ is ergodic, the trivial cobounding map $\alpha\from X\to G\backslash G$ is minimal and we recover the formula from Corollary~\ref{c:equidistribution}:
\begin{equation*}
    \delta_\mu(L) = \frac{1}{|G|}\sum_{k\in K}\sum_{Gg\in G\backslash G}\mu(\alpha^{-1}(Gg))\mu(\alpha^{-1}(Ggk)) = \frac{|K|}{|G|}.
\end{equation*}

\begin{proof}
    Without loss of generality we may assume that $L = \varphi^{-1}(k)$ for some fixed $k\in G$. Fix a minimal cobounding map $\alpha\from X\to H\backslash G$ and let $\bar\mu$ be the measure on $G\skewprod_{\varphi}  X$ defined by 
    \begin{equation*}
        \bar\mu(B) = [G:H](\nu\times\mu)(B\cap Y_\alpha).
    \end{equation*}
    Note that $\bar\mu$ is ergodic by assumption. For $g\in G$ let $U_g = \{g\}\times X$. We claim that $\nu\times\mu$ projects to $\mu$. Indeed, for every measurable subset $B\subseteq X$,
    \begin{align*}
        \bar\mu(G\times B) &= \sum_{Hg\in H\backslash G}\bar\mu((Hg\times B)\cap Y_\alpha) \\
                           &= [G:H] \sum_{Hg\in H\backslash G}(\nu\times\mu)(Hg\times (B\cap\alpha^{-1}(Hg))) \\
                           &= \sum_{Hg\in H\backslash G}\mu(B\cap \alpha^{-1}(Hg))\\
                           &= \mu(B).
    \end{align*}
    Applying Theorem~\ref{t:first-main} with $\{g\}\times X = U_g$ we obtain
    \begin{align*}
        \delta_\mu(\varphi^{-1}(k)) &= \sum_{g\in G} \bar\mu(U_g)\bar\mu(U_{gk}) \\
                                    &= [G:H]^2\sum_{g\in G} (\nu\times\mu)(U_g\cap Y_\alpha)\,(\nu\times\mu)(U_{gk}\cap Y_\alpha) \\
                                    &= [G:H]^2\sum_{g\in G} \frac{1}{|G|^2}\mu(\alpha^{-1}(Hg))\mu(\alpha^{-1}(Hgk)) \\
                                    &= \frac{1}{|H|^2}\sum_{g\in G}\mu(\alpha^{-1}(Hg))\mu(\alpha^{-1}(Hgk)) \\
                                    \intertext{which after regrouping like terms} 
                                    &= \frac{1}{|H|}\sum_{Hg\in H\backslash G}\mu(\alpha^{-1}(Hg))\mu(\alpha^{-1}(Hgk)),
    \end{align*}
    concluding the proof. 
\end{proof}

Recall from Corollary~\ref{cor:minerg} that the ergodicity assumption from the above theorem is equivalent to $\nu\times\mu$ being ergodic on \emph{any} closed invariant subset (which is then necessarily minimal).

\subsection{Examples}\label{ss:exa}

We finish the section with two examples that illustrate various aspects of Theorem~\ref{t:third-main}.

\begin{example}\label{eg:thue-morse-2}
    Let $X$ be the Thue--Morse shift with its unique invariant measure $\mu$ (see Example~\ref{eg:thue-morse-1}) and let $\varphi\from A^*\to\Z/2\Z$, $\varphi(a) = 1$, $\varphi(b)=0$. 

    The morphism $\varphi$ has two cobounding maps mod 1 on $X$, hence by Corollary~\ref{c:mod1} $\Z/2\Z\skewprod_{\varphi}  X$ has two minimal closed invariant subsets on which the measure $\nu\times\mu$ is ergodic.
    The cobounding maps take constant values on the cylinders of length~7; one of them is depicted in Figure~\ref{f:coloring-thue-morse}.
    Note that the cobounding maps are not constant on cylinders of length~6: for instance, the map of Figure~\ref{f:coloring-thue-morse} satisfies $\alpha([baabbab])=1$ and $\alpha([baabbaa])=0$, but $[baabbaa]$ and $[baabbaa]$ are both contained in the cylinder $[baabba]$ of length 6.

    \begin{figure}[hbt]
        \centering
        \begin{tikzpicture}[node/.style={draw,rectangle},font=\small]
            \node[node,label=above:{\small$0$}] (abbabaa) at (0,5) {$abbabaa$};
            \node[node,label=right:{\small$0$}] (baabbaa) at (6,3.5) {$baabbaa$};
            \node[node,label=right:{\small$0$}] (aabbaab) at (6,2.5) {$aabbaab$};
            \node[node,label=below:{\small$0$}] (bbaabab) at (4,1) {$bbaabab$};
            \node[node,label=below:{\small$0$}] (baababb) at (2,1) {$baababb$};
            \node[node,label=below:{\small$0$}] (aababba) at (0,1) {$aababba$};
            \node[node,label=left: {\small$0$}] (babbaba) at (0,3.6) {$babbaba$};
            \node[node,label=left: {\small$0$}] (babbaab) at (-4,1.5) {$babbaab$};
            \node[node,label=left: {\small$0$}] (abbaabb) at (-4,2.5) {$abbaabb$};
            \node[node,label=above:{\small$1$}] (bbabaab) at (2,5) {$bbabaab$};
            \node[node,label=above:{\small$1$}] (babaabb) at (4,5) {$babaabb$};
            \node[node,label=right:{\small$1$}] (babaaba) at (2,3.6) {$babaaba$};
            \node[node,label=right:{\small$1$}] (abaabba) at (6,4.5) {$abaabba$};
            \node[node,label=right:{\small$1$}] (abbaaba) at (6,1.5) {$abbaaba$};
            \node[node,label=right:{\small$1$}] (abaabab) at (2,2.3) {$abaabab$};
            \node[node,label=left: {\small$1$}] (ababbab) at (0,2.3) {$ababbab$};
            \node[node,label=below:{\small$1$}] (ababbaa) at (-2,1) {$ababbaa$};
            \node[node,label=left: {\small$1$}] (bbaabba) at (-4,3.5) {$bbaabba$};
            \node[node,label=left: {\small$1$}] (baabbab) at (-4,4.5) {$baabbab$};
            \node[node,label=above:{\small$1$}] (aabbaba) at (-2,5) {$aabbaba$};

            \draw[->,above] (abbabaa) edge node {$a$} (bbabaab);
            \draw[->,above] (bbabaab) edge node {$b$} (babaabb);
            \draw[->,right] (bbabaab) edge node {$b$} (babaaba);
            \draw[->,above,bend left=15] (babaabb) edge node {$b$} (abaabba);
            \draw[->,right] (abaabba) edge node {$a$} (baabbaa);
            \draw[->,right] (baabbaa) edge node {$b$} (aabbaab);
            \draw[->,right] (aabbaab) edge node {$a$} (abbaaba);
            \draw[->,above,bend left=15] (abbaaba) edge node {$a$} (bbaabab);
            \draw[->,above] (bbaabab) edge node {$b$} (baababb);
            \draw[->,right] (babaaba) edge node {$b$} (abaabab);
            \draw[->,right] (abaabab) edge node {$a$} (baababb);
            \draw[->,above] (baababb) edge node {$b$} (aababba);
            \draw[->,right] (aababba) edge node {$a$} (ababbab);
            \draw[->,right] (ababbab) edge node {$a$} (babbaba);
            \draw[->,right] (babbaba) edge node {$b$} (abbabaa);
            \draw[->,above] (aababba) edge node {$a$} (ababbaa);
            \draw[->,above,bend left=15] (ababbaa) edge node {$a$} (babbaab);
            \draw[->,right] (babbaab) edge node {$b$} (abbaabb);
            \draw[->,right] (abbaabb) edge node {$a$} (bbaabba);
            \draw[->,right] (bbaabba) edge node {$b$} (baabbab);
            \draw[->,above,bend left=15] (baabbab) edge node {$b$} (aabbaba);
            \draw[->,above] (aabbaba) edge node {$a$} (abbabaa);
        \end{tikzpicture}
        \caption{One of the two cobounding maps $X\to\Z/2\Z$ on the Thue--Morse shift for the morphism $\varphi\from \{a,b\}^*\to\Z/2\Z$, $\varphi(a)=1$, $\varphi(b)=0$. The map is constant on cylinders of length 7}\label{f:coloring-thue-morse}
    \end{figure}
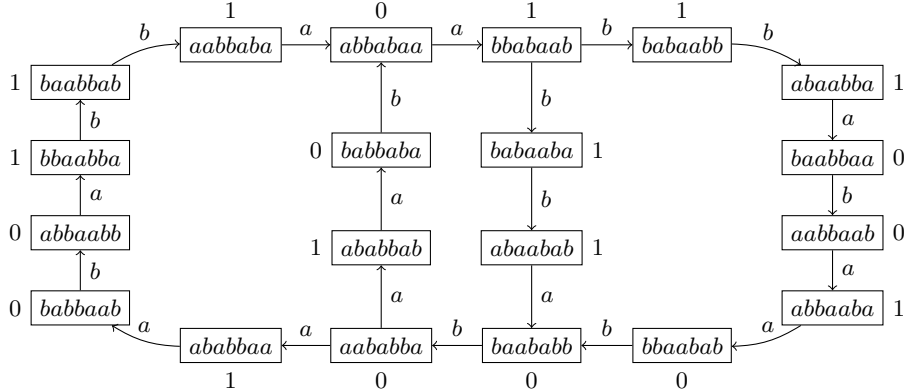

    Observe that the cobounding maps are \emph{fair}, in the sense that the preimages $\alpha^{-1}(g)$ have the same $\mu$-measure for all elements of the group. Therefore, even though the skew product $\Z/2\Z\skewprod_{\varphi}  X$ is \emph{not} ergodic, the shift space $X$ is still evenly distributed with respect to $\varphi$, in the sense that     \begin{equation*}
        \delta_\mu(\varphi^{-1}(1)) = \delta_\mu(\varphi^{-1}(0)) = 1/2. 
    \end{equation*}
\end{example}

\begin{example}\label{eg:unimodular}
 Consider the following substitution and morphism onto $S_3$ defined on the alphabet $A = \{a,b,c\}$: 
 \begin{equation*}
     \sigma\from a\mapsto aab,\  b\mapsto acb,\ c\mapsto ba, \qquad  \varphi(a)=\varphi(c) = (1\:2\:3), \ \varphi(b)=(1\:2).
 \end{equation*}
 Let $X$ be the shift space generated by the primitive substitution $\sigma$. We claim that the skew product $S_3\skewprod_{\varphi}  X$ has \emph{three} minimal closed invariant subsets.  These minimal subsets correspond to three cobounding maps
    \begin{equation*}
        \alpha_{12}\from X\to \langle (1\:2)\rangle\backslash S_3,\quad \alpha_{13}\from X\to \langle (1\:3)\rangle\backslash S_3,\quad \alpha_{23}\from X\to \langle (2\:3)\rangle\backslash S_3.
    \end{equation*}

    The cobounding map $\alpha_{12}$ is depicted in Figure~\ref{f:cobounding-unimod}, and the other two can be deduced from $\alpha_{12}$ using the natural left action of $S_3$ on cobounding maps. Note that in all cases the subgroup involved has index 3 in $S_3$, in accordance with Corollary~\ref{cor:number}. 

    \begin{figure}
        \centering
        \begin{tikzpicture}[node/.style={draw,rectangle},font=\small,xscale=2,yscale=1]

            \node[node,label=below:{\strut$H$}] (abac) at (0,0) {${abac}$};
            \node[node,label=below:{$H(1\:3)$}] (bacb) at (-1,0)   {${bacb}$};
            \node[node,label=below:{$H(2\:3)$}] (aaba) at (1,0)    {${aaba}$};
            \node[node,label=left:{\strut$H$}] (cbac) at (-1.5,1)   {${cbac}$};
            \node[node,label=right:{\strut$H$}] (abaa) at (1.5,1)    {${abaa}$};
            \node[node,label={$H(2\:3)$}] (acba) at (-1,2)   {${acba}$};
            \node[node,label={\strut$H$}] (cbaa) at (0,2) {${cbaa}$};
            \node[node,label={$H(1\:3)$}] (baab) at (1,2) {${baab}$};
            \node[node,label=left:{$H(1\:3)$}] (aacb) at (-1.5,3) {${aacb}$};
            \node[node,label=right:{\strut$H$}] (aabb) at (1.5,3) {${aabb}$};
            \node[node,label={$H(2\:3)$}] (baac) at (-1,4) {${baac}$};
            \node[node,label={$H(1\:3)$}] (bbaa) at (0,4)    {${bbaa}$};
            \node[node,label={\strut$H$}] (abba) at (1,4) {${abba}$};

            \draw[->,right] (bacb) edge node [font=\small] {${b}$} (acba) ;
            \draw[->,below] (abac) edge node [font=\small] {${a}$} (bacb) ;
            \draw[->,below] (aaba) edge node [font=\small] {${a}$} (abac) ;
            \draw[->,below right,bend right=25] (aaba) edge node [font=\small] {${a}$} (abaa) ;
            \draw[->,below left ,bend right=25] (cbac) edge node [font=\small] {${c}$} (bacb) ;
            \draw[->,above right,bend right=25] (abaa) edge node [font=\small] {${a}$} (baab) ;
            \draw[->,above] (acba) edge node [font=\small] {${a}$} (cbaa) ;
            \draw[->,above left ,bend right=25] (acba) edge node [font=\small] {${a}$} (cbac) ;
            \draw[->,above] (cbaa) edge node [font=\small] {${c}$} (baab) ;
            \draw[->,left ] (baab) edge node [font=\small] {${b}$} (aaba) ;
            \draw[->,below right,bend right=25] (baab) edge node [font=\small] {${b}$} (aabb) ;
            \draw[->,below left ,bend right=25] (aacb) edge node [font=\small] {${a}$} (acba) ;
            \draw[->,above right,bend right=25] (aabb) edge node [font=\small] {${a}$} (abba) ;
            \draw[->,above left ,bend right=25] (baac) edge node [font=\small] {${b}$} (aacb) ;
            \draw[->,above] (bbaa) edge node [font=\small] {${b}$} (baac) ;
            \draw[->,above] (abba) edge node [font=\small] {${a}$} (bbaa) ;
        \end{tikzpicture}
        \caption{Cobounding map mod $H = \langle(1\:2)\rangle$ on the shift of Example~\ref{eg:unimodular} for the morphism $\varphi\from \{a,b,c\}^*\to S_3$, $\varphi(a)=\varphi(c) = (1\:2\:3)$, $\varphi(b)=(1\:2)$. The map is constant on cylinders of length 4}
        \label{f:cobounding-unimod}
    \end{figure}
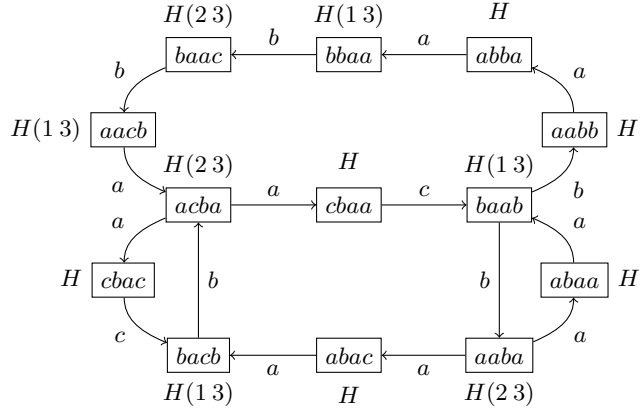

    We now briefly sketch a proof of the fact that the above cobounding maps are indeed minimal. In the present case, this is equivalent to showing that there are no cobounding maps mod 1. Recall that by Corollary~\ref{c:mod1}, there is a cobounding map mod 1 if and only if $\RR_X(u)$ has trivial image under $\varphi$ for all sufficiently long words $u$. 

    By the first main result from \cite{Berthe2023}, since $\sigma$ is a bifix encoding, there exists a constant $K>0$ such that for all $u\in \cL(\sigma)$ with $|u|\geq K$,
    \begin{equation*}
        \RR_X(\sigma(u)) = \sigma(\RR_X(u)).
    \end{equation*}

    Using the formula provided in \cite{Berthe2023} we find the upper bound $K\leq 6$, but direct computations show that we can take $K=2$; in fact, the one-letter word $u=c$ is the \emph{only} word which fails the above equality. Now take the sequence of words $u_n = \sigma^n(a)$; we claim that $(2\:3)\in\RR_X(u_n)$, for infinitely many $n$. Indeed, observe that the following equalities holds:  
    \begin{equation*}
        \varphi\circ\sigma = \varphi\circ\sigma^8,\quad \varphi\circ\sigma(a) = \varphi\circ\sigma(b) = \varphi\circ\sigma(c) = (2\:3). 
    \end{equation*}

    As $\RR_X(u_0) = \{{a, ba, bba, cba}\}$, it follows that $(2\:3) = \varphi(\sigma^{7k+1}(a))$ belongs to $\RR_X(u_{7k+1})$ for all $k\geq 0$. This shows that $\RR_X(u_n)$ have non-trivial images for infinitely many $n$, thus $\varphi$ has no cobounding map mod 1 on $X$. This confirms that the above cobounding maps are minimal.
\end{example}

\section{Ergodicity for primitive substitutions}
\label{s:morphic}

In this section we focus on the special case of shift spaces defined by primitive substitutions. Our main result is a sufficient condition for the minimal closed invariant subsets of skew products to be uniquely ergodic (Proposition~\ref{p:invertible-ergodic}). As a corollary, we deduce that substitutive dendric shifts have ergodic skew products with all finite groups (Theorem~\ref{t:dendric-ergodic}).   Note that the   family of dendric shifts, studied in \cite{BerstelDeFelicePerrinReutenauerRindone2012,BertheDeFeliceDolceLeroyPerrinReutenauerRindone2015,BertheDeFeliceDolceLeroyPerrinReutenauerRindone2015c,BDFDLPRR:15,BFFLPR:2015}, encompasses several classical families of shifts, such as Sturmian shifts,  codings of interval exchanges, and Arnoux--Rauzy shifts.

\subsection{Skew products based on primitive shifts}

Let us fix a primitive substitution $\sigma$ over a finite alphabet $A$ and let $X = X(\sigma)$ be the shift space defined by $\sigma$. The shift $X$ is a minimal shift space and we recall that it is uniquely ergodic by Michel's theorem. 

\begin{definition} 
    Let $\varphi\from A^*\to G$ be a morphism onto a finite group $G$. We say that the primitive substitution  $\sigma$ is \emph{invertible under $\varphi$} if:
    \begin{equation}\label{eq:invertibility}
        \exists  n\geq 1,\  \varphi\circ\sigma^n = \varphi. 
    \end{equation}
\end{definition}

\begin{example}
    Let $\sigma\from a\mapsto ab, b\mapsto a$ be the Fibonacci substitution and consider the morphism $\varphi\from A^*\to \Z/2\Z$, $\varphi(a)=1$, $\varphi(b)=0$. One checks that
    \begin{equation*}
        \sigma^3(a) = abaab,\ \sigma^3(b) = aba,
    \end{equation*}
    hence $\varphi\circ\sigma^3 = \varphi$, so $\sigma$ is invertible under $\varphi$.  This property   has already been used   in  the proof of  Proposition~\ref{p:fibo}.
    In fact, $\sigma$ has the much stronger property of \emph{being invertible under every homomorphism onto a finite group}, as we shall later see (Lemma~\ref{l:invertible1}).
\end{example}

\begin{proposition}\label{p:invertible-ergodic}
    Let $\sigma$ be a primitive substitution, $X$ be its shift space, and $\varphi\from X\to G$ be a morphism onto a finite group $G$. If $\sigma$ is invertible under $\varphi$, then the minimal closed invariant subsets of $G\skewprod_{\varphi}  X$ are uniquely ergodic.
\end{proposition}

\begin{proof}
    Up to replacing $\sigma$ by some power, we may assume without loss of generality (since this does not change the shift space) that $\varphi\circ\sigma=\varphi$ and that $\sigma$ has a fixed point $y\in X$. 

    Let $\Psi$ be the topological conjugacy from Lemma~\ref{l:isomskewprod}. The fact that $\varphi = \varphi\circ\sigma$ entails the existence of a substitution $\bar\sigma$ on $(G\times A)^*$ such that $\Psi(g,\sigma(x)) = \bar\sigma(\Psi(g,x))$, namely, when $\sigma(a) = b_0\ldots b_{n-1}$,
    \begin{equation*}
        \bar\sigma(g,a) = (g_0,b_0)(g_1,b_1)\ldots(g_{n-1},b_{n-1}),\quad \text{where}\ g_0=g,\ g_{i+1} = g_i\varphi(b_i).
    \end{equation*}
    Observe that, for every $g\in G$, the infinite word $z = \Psi(g,y)$ is a uniformly recurrent fixed point of $\bar\sigma$, since
    \begin{equation*}
        \bar\sigma(z) = \bar\sigma(\Psi(g,y)) = \Psi(g,\sigma(y)) = \Psi(g,y) = z.
    \end{equation*}
    
    Let $B$ be the subset of letters in $G\times A$ appearing in $z$; it follows that $\bar\sigma$ restricts to a substitution over $B$. Moreover, $z$ belongs to a minimal subset of $(G\times A)^\Z$, hence it must be uniformly recurrent. Since $\bar\sigma$ is a growing substitution fixing a uniformly recurrent word, it must be primitive and have for shift space the closed orbit of $z$, which is $\Psi(Y)$. In particular, $\Psi(Y)$ is uniquely ergodic by Michel's theorem, and so is~$Y$.
\end{proof}

Observe that Example~\ref{eg:unimodular} fails both the invertibility property \eqref{eq:invertibility} \emph{and} the property of Corollary~\ref{c:mod1}. At this time we do not know whether the product measure is ergodic on the minimal closed invariant subsets of the skew product. In contrast, here is an example which satisfies both \eqref{p:invertible-ergodic} and Corollary~\ref{c:mod1}.

\begin{example}\label{ex:fourletter-0}
    Consider the following substitution and morphism onto $\Z/2\Z$ defined on the alphabet $A = \{a,b,c,d\}$:
    \begin{equation*}
        \sigma\from a\mapsto baa,\ b\mapsto adc,\ c\mapsto cdc,\ d\mapsto ad,\qquad \varphi\from a,b\mapsto 0,\ c,d\mapsto 1.
    \end{equation*}
    Then $\varphi\circ\sigma=\varphi$, hence the minimal closed invariant subsets of $\Z/2\Z\skewprod _{\varphi} X$ are uniquely ergodic. The substitution $\bar\sigma$ satisfying $\Psi(g,\sigma(x)) = \bar\sigma(\Psi(g,x))$ is defined as follows on the alphabet $\Z/2\Z\times A$, written $a_0$, $a_1$, $b_0$, $b_1$, etc.\ for convenience,
    \begin{equation*}
        \begin{array}{llll}
            a_0\mapsto b_0 a_0 a_0,& b_0\mapsto a_0 d_0 c_1,& c_0\mapsto c_0 d_1 c_0,& d_0\mapsto a_0 d_0, \\
            a_1\mapsto b_1 a_1 a_1,& b_1\mapsto a_1 d_1 c_0,& c_1\mapsto c_1 d_0 c_1,& d_1\mapsto a_1 d_1. \\
        \end{array}
    \end{equation*}
    This substitution splits in two primitive substitutions defined respectively on the alphabets $B = \{a_0, b_0, c_1, d_0\}$ and $C = \{a_1, b_1, c_0, d_1\}$. Each corresponds to one of the two minimal closed invariant subsets of the skew product $\Z/2\Z\skewprod _{\varphi} X$.
\end{example}

In what follows, we say that a substitution $\sigma\from A^*\to A^*$ is \emph{invertible} if its extension to an endomorphism on the free group $F_A$ is an automorphism. We now state two  further properties  when the substitution is  assumed to be invertible.  
\begin{lemma}\label{l:invertible1}
    Let $\sigma$ be a primitive substitution. If $\sigma$ is invertible, then it is invertible under every morphism $\varphi\from A^*\to G$ onto a finite group $G$.
\end{lemma}

\begin{proof}
    Let $\aut(F_A)$ be the automorphism group of $F_A$ and $\hom(F_A, G)$ be the set of morphisms $F_A\to G$. For $\tau\in\aut(F_A)$, denote by $\tau_*$ the self-map of $\hom(F_A, G)$ defined by
    \begin{equation*}
        \tau_*(\varphi) = \varphi\circ\tau.
    \end{equation*}
    Since $\rho_*\circ\tau_* = (\tau\circ\rho)_*$, the map $\tau\mapsto\tau_*$ is a morphism from $\aut(F_A)$ to the symmetric group on $\hom(F_A,G)$. Moreover observe that $\hom(F_A,G)$ is a finite set, being in bijection with the set of maps $A\to G$. As a result, for every automorphism $\tau$ of $F_A$, $\tau_*$ is a permutation of $\hom(F_A,G)$ with finite order; in other words, there exists $n\geq 0$ such that $\tau_*^n=\mathrm{id}$, i.e. $\varphi\circ\tau^n = \varphi$ for every morphism $F_A\to G$. Applying this to the extension of $\sigma$ to an automorphism of $F_A$ yields the result.
\end{proof}

In what follows, we say that a substitution is \emph{aperiodic} if it generates a shift space that contains no finite orbit. We next establish the following lemma that will be used in the next section. Observe that the generation property expressed below implies the one  stated in Theorem~\ref{t:minimalskew}.

\begin{lemma}
    \label{l:invertible2}
    Let $\sigma$ be a primitive aperiodic substitution and $X$ the shift generated by $\sigma$. If $\RR_X(u)$ generates $F_A$ for every $u\in\cL(X)$, then $\sigma$ is invertible.
\end{lemma}

\begin{proof}
    First, fix a point $x\in X$ which is periodic under $\sigma$, meaning $\sigma^k(x)=x$ for some $k>0$~\cite[Proposition~1.4.8]{DurandPerrin2021}. Let $w$ be a word of the form $w = x_{[-n,n)}$ and let $u = x_{[0,n)}$ be its suffix of length $n$, for some $n\geq 0$. As Almeida and Costa observed in the proof of~\cite[Proposition~5.5]{Almeida2013}, it follows from Mossé's recognizability theorem~\cite[Theorem~3.1bis]{Mosse1992} that when $n$ is large enough, $u\RR_X(w)u^{-1}\subseteq\img(\sigma^k)$. Since by assumption $\RR_X(w)$ generates $F_A$, we conclude that $\sigma^k$ is a surjective endomorphism of $F_A$. Since finitely generated free groups have the Hopfian property~\cite[Proposition~3.5]{Lyndon2001}, it follows that $\sigma^k$ is invertible, hence so is $\sigma$.
\end{proof}

For the sake of completeness, we give an example showing that the converse of the above lemma fails, i.e. invertibility does not guarantee that all $\RR_X(u)$ generate $F_A$.

\begin{example} \label{ex:fourletter}
    Let $\sigma$ be the primitive substitution  from Example~\ref{ex:fourletter-0} defined on the four-letter alphabet $A=\{a,b,c,d\}$ by:
    \begin{equation*}
        \sigma\from a\mapsto baa,\ b\mapsto adc,\ c\mapsto cdc,\ d\mapsto ad.
    \end{equation*}

    This is an invertible substitution, with $\sigma^{-1}$ given by
    \[
    a\mapsto bc^{-1}d^{-1}b,  \ b\mapsto a b^{-1}dc b^{-2} dc b^{-1},\ c\mapsto d^{-1}b,\ d\mapsto b^{-1}dcb^{-1}d.
    \]
    Nonetheless, the following is not a generating set of $F_A$:
    \begin{equation*}
        \RR_X(a) = \{a, ba, dcba, dcdca, dcdcba\}.
    \end{equation*}

    In fact, this set of return words generate the rank 3 subgroup of $F_A$ with basis $\{a,b,dc\}$.
\end{example}

\subsection{Skew products based on dendric shifts}
\label{ss:dendric}

We  turn now  to the \emph{dendric} case. First, we recall the definition. Let $X$ be a shift space. For $w\in\cL(X)$, let $\lext(w)=\{a\in A\mid aw\in\cL(X)\}$ and $\rext(w)=\{a\in A\mid wa\in\cL(X)\}$. We denote by $\ext(w)$ the graph with vertices the disjoint union of $\rext(w)$ and $\lext(w)$ and edges the pairs $(a,b)\in A\times A$ such that $awb\in\cL(X)$; it is called the \emph{extension graph} of $w$. A shift space $X$ is \emph{dendric} if for every $w\in \cL(X)$, the extension graph $\ext(w)$ is a tree. For instance, every Sturmian shift is dendric~\cite{BertheDeFeliceDolceLeroyPerrinReutenauerRindone2015}. 

An important result concerning dendric shifts is the so-called \emph{Return Theorem} by Berthé et al. which we quote next.
\begin{theorem}[{\cite[Theorem~4.5]{BertheDeFeliceDolceLeroyPerrinReutenauerRindone2015}}]
    \label{t:return}
    Let $X$ be a dendric shift on an alphabet $A$. For every $w\in\cL(X)$, the set $\RR_X(w)$ is a basis of the free group on $A$. 
\end{theorem}

Therefore it follows from Theorem~\ref{t:minimalskew} that every skew product of a dendric shift and a finite group is minimal.

\begin{example}\label{exampleFibonacci2}
    The Fibonacci shift (see Section~\ref{ss:fibo})  is Sturmian  and therefore dendric. There, we have $\rext(a)=\lext(a)=\{a,b\}$ and the graph $\ext(a)$ is shown in Figure~\ref{figureFibonacci}. Moreover $\RR_X(a)=\{a,ab\}$, which is obviously a basis of the free group on $\{a,b\}$.
    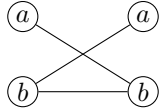
\begin{figure}[hbt]
    \centering
    \tikzstyle{loop above}=[in=50,out=130,loop]
    \tikzset{node/.style={circle, draw,minimum size=0.4cm,inner sep=0.4pt}}
	\tikzset{title/.style={minimum size=0.5cm,inner sep=0pt}}
        \begin{tikzpicture}[xscale=0.8]
          \node[node](al)at(0,0){$a$};\node[node](ar)at(2,0){$a$};
          \node[node](bl)at(0,-1){$b$};\node[node](br)at(2,-1){$b$};
          \draw(al)edge node{}(br);
          \draw(bl)edge node{}(ar);\draw(bl)edge node{}(br);
        \end{tikzpicture}
        \caption{The extension graph $\ext(a)$ in the Fibonacci shift}
        \label{figureFibonacci}
    \end{figure}
\end{example}

We also give an example of a shift space which is not dendric.
\begin{example}\label{eg:thue-morse-3}
    Let $X$ be the Thue--Morse shift from Example~\ref{eg:thue-morse-1}, generated by the two-letter substitution $\sigma\from a\mapsto ab, b\mapsto ba$. For the word $w = aba$, we find $\lext(w)=\rext(w)=\{a,b\}$ and the graph $\ext(w)$, depicted in Figure~\ref{figureMorse}, is not connected. 
    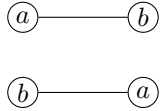
\begin{figure}[hbt]
        \centering
        \tikzset{node/.style={circle, draw,minimum size=0.4cm,inner sep=0.4pt}}
        \tikzset{title/.style={minimum size=0.5cm,inner sep=0pt}}
        \begin{tikzpicture}[xscale=0.8]
            \node[node](al) at (0,0)  {$a$};
            \node[node](ar) at (2,-1)  {$a$};
            \node[node](bl) at (0,-1) {$b$};
            \node[node](br) at (2,0) {$b$};
            \draw (al) edge node{} (br);
            \draw (bl) edge node{} (ar);
        \end{tikzpicture}
        \caption{The extension graph $\ext(w)$, $w=aba$, in the Thue--Morse shift}
        \label{figureMorse}
    \end{figure}
\end{example}

We recover as a consequence of the  following  general result the unique ergodicity of the skew product of Section~\ref{ss:fibo}.
\begin{theorem}\label{t:dendric-ergodic}
    Let $X$ be a dendric shift generated by a primitive substitution. Then the skew product $G\skewprod_{\varphi}  X$ is uniquely ergodic for every morphism $\varphi\from A^*\to G$ onto a finite group $G$.
\end{theorem}

The proof uses Lemmas~\ref{l:invertible1} and \ref{l:invertible2}. Observe that dendric shift spaces in particular fall under the scope of Lemma~\ref{l:invertible2}, thanks to the Return Theorem of Berthé et al., stated above as Theorem~\ref{t:return}. 

\begin{proof}
    Assume that $X$ is generated by the primitive morphism $\sigma$. Observe that $\sigma$ must be aperiodic, as dendric spaces cannot contain finite orbits. Moreover by the Return Theorem, $\langle\RR_X(w)\rangle=F_A$ for all $w\in\cL(X)$. Therefore we may apply Lemmas~\ref{l:invertible1} and \ref{l:invertible2} to conclude that $\sigma$ is invertible, and as a result invertible under every morphism $\varphi\from A^*\to G$ onto a finite group $G$. Applying Proposition~\ref{p:invertible-ergodic}, it follows that the skew product $G\skewprod_{\varphi}  X$ has uniquely ergodic minimal closed invariant subsets. As the skew product is also minimal by the Return Theorem (Theorem~\ref{t:return}) and Theorem~\ref{t:minimalskew}, this completes the proof.
\end{proof}

We thus obtain, as a direct application of Corollary~\ref{c:equidistribution}, the following result about the density of group languages in substitutive dendric shifts.
\begin{corollary}\label{c:dendric-ergodic}
    Let $X$ be a dendric shift generated by a primitive substitution and  let $\mu$ be its unique ergodic measure. For every morphism $\varphi\from A^*\to G$ onto a finite group $G$ and every language $L = \varphi^{-1}(K)$, $K\subseteq G$, the density $\delta_\mu(L)$ exists and is equal to $|K|/|G|$.
\end{corollary}

\section{Skew products based on Sturmian shifts} \label{s:examples}

We end with a discussion about earlier related works on skew products based on irrational rotations, mostly by Veech~\cite{Veech1969,Veech1975} and Jager and Liardet~\cite{JL:88}. Due to the nature of the examples involved, and to stay consistent with the relevant literature, it is convenient here to use alphabets consisting of natural numbers, such as $\{0,1\}$ and $\{1,2\}$.

Among the first classical examples of skew products, skew translations (i.e. skew products with base an irrational rotation on the unit circle), and their ergodic and spectral properties  have been widely investigated; see e.g. \cite{Veech1969,Veech1975,Stewart,Merrill:85,pre:Guenais2006,FerHub} and the classical references \cite{CFS:82,Petersen1983}. In particular,  they have been used to produce examples of interval exchanges that are not uniquely ergodic \cite{KN:76}. Such examples are based on skew products  of    irrational rotations associated with the group  $\Z/2\Z$, that are  minimal and not uniquely ergodic, with the skewing function being the indicator function of an interval \cite{Veech1969,Veech1975}. 

More precisely, let $\alpha$ be an irrational number in $[0,1]$. We consider the rotation  $R_{\alpha}\from x \mapsto  x + \alpha$ modulo $1$ defined on $\T=\R/\Z$. Let $I=[0, \beta)$ be a semi-open interval of $\T$.  Let ${1}_I\colon  {\mathbb T} \to \{0,1\}$ be the indicator function of $I$, i.e. $1_I(x)=1$ if and only if $x\in I$. Let $m \geq 2$ and let $G=\Z/m\Z$. Let $\varphi:\{0,1\}^*\to G$ be the   morphism defined by $0\mapsto 0$, $1 \mapsto 1$. We then consider the skew product $G\skewprod_I {\T}$ of $R_{\alpha}$ defined as 
\begin{equation*}
    (k, x) \mapsto ( k+  \varphi \circ 1_I (x),  x +\alpha)= ( k+  \varphi \circ 1_I (x), R_{\alpha} (x)).
\end{equation*}

Such skew products over rotations are closely related to symbolic skew products such as those considered in the present paper, and more precisely, to skew products over shifts obtained as  binary codings of rotations. Let $I^{\comp}$ stand for the complement of $I$ in $\T$.  We use the notation $I_0=I$, and $I_1=I^{\comp}$. Let $x$ be the infinite word in $\{0,1\}^\Z$ obtained by coding the orbit of $0$ under $R_{\alpha}$ with respect to the partition ${\mathcal I}=\{I, I^\comp\}$, i.e. for any $n \in \Z$, $x_n=1$ if and only if  $R_{\alpha}^n(0) \in I$, or in other words, $x_n = 1_I( R_{\alpha}^n (0))$ for all $n \in {\mathbb Z}$. Let $(X,S)$ be the shift space generated by $x$.  Then this shift is minimal and uniquely ergodic since $\alpha$ is irrational and $I$ is semi-open.  Indeed, we associate in a bijective way words  $w=w_1 \ldots w_n$  in $\{0,1\}^n$ 
with   sets $I_{w} := I_{w_1} \cap  R_{\alpha} ^{-1} I_{w_2} \cap  \dots   \cap R_{\alpha} ^{-n+1} I_{w_n}$  as follows: 
$w$ occurs at index $k$ in $x$ if and only if $k \alpha \in I_w$.  It is important to stress the fact that  $I$ is assumed to be semi-open. As an illustration of the relevance of  this hypothesis, 
 consider the case   $I=[0,\alpha]$,  with $\alpha <1/2$;  then, $00$   occurs   only once in the  orbit of $0$ and minimality fails. The fact that $I$ is semi-open  guarantees   that if $I_w$ is not empty, 
then 
there are \emph{infinitely many} $k$ such that  $k \alpha \in I_w$, 
by  density of the sequence  $(k \alpha)_k$, hence the uniform recurrence of   $x$, and  thus the minimality of $X$.   When the length of $I$ equals $\alpha$ or $1-\alpha$, such binary codings of rotations are Sturmian. 

Not all intervals $I$ lead to ergodic skew products. In fact, by \cite{Veech1969},   if   $\alpha$ is  not badly approximable, then  there exists an interval $I$ such that the skew product $G\skewprod_I\T$ of $R_{\alpha}$ is  not   ergodic, a  result  which  inspired   the  elegant   characterization of ergodicity  from  \cite{pre:Guenais2006} stated in terms of Ostrowski's numeration. Moreover, as proved in  \cite{Veech1975},   skewing a badly approximable rotation  over a finite number of intervals  with rational endpoints still provides a uniquely ergodic skew product. However,  \cite{Chaika:14} provides  a  $\Z/2\Z$ skew product of a  badly approximable rotation that  is   minimal and not uniquely ergodic, by skewing   over  the indicator function of a union of  two intervals. 

In most examples considered in \cite{Veech1969,Veech1975}, intervals have lengths that do not belong to $\Z \alpha+ \Z$. We consider here the  complementary case of Sturmian shifts. We can then apply Theorem \ref{t:dendric-ergodic} and 
Corollary~\ref{c:dendric-ergodic} when they are furthermore assumed to be generated by a substitution, 
such as exemplified below. Note also that substitutive Sturmian shifts have been characterized in \cite{Yasutomi,Crisp1993} (in particular, the parameter $\alpha$ of the underlying rotation is quadratic, and thus  badly approximable).

\begin{example}\label{ex:Fibocounting}
    Let $X$ be  the Fibonacci  shift on  the alphabet $\{0,1\}$.  We   consider the  skew product  $\Z/m\Z\skewprod_{\varphi}  X$, where $\varphi$ is the morphism $\{0,1\}^*\to\Z/m\Z$ given by $0\mapsto 0$, $1 \mapsto 1$.  In particular, one has $\varphi^{(n)}(x)= |x_0 \ldots x_{n-1}|_1$ modulo $m$ for $n \geq 0$, as explained in Example \ref{ex:counting0}. By Theorem \ref{t:dendric-ergodic}, $\Z/m\Z\skewprod _{\varphi} X$ is uniquely ergodic, which yields equidistribution results on the congruence of the number of visits of $R_{\alpha}$ to the interval $[0 ,\alpha)$, where $\alpha=\frac{\sqrt 5 -1}{2}$. In other words, for every $x \in X$, $r\in\Z/m\Z$ and $a\in \{0,1\}$, one has:
    \begin{equation*}
        \frac{1}{N} \card \{ n: 0 \leq n \leq N-1,\ |x_0 \ldots x_{n-1}|_a  \equiv r  \bmod m\} \to \frac{1}{m},
    \end{equation*}
    or in other words,
    \begin{equation*}
        \frac{1}{N}   \card \{ n: 0 \leq n \leq N-1,\   \card\{i:  0 \leq i  <n,\  i \alpha \in [0,\alpha)\} \equiv r \bmod m \} \to \frac{1}{m}.
    \end{equation*}
\end{example}

Finally, we consider in the   next example,  inspired by  the work of Jager and Liardet~\cite{JL:88},   equidistribution properties  for convergents  in continued fraction expansions.
\begin{example}\label{ex:Fibocontfrac}

    Let $X$ be the Fibonacci shift on the alphabet $ \{1,2\}$.
     We continue Example \ref{ex:convergents}  with the skew product  $G(2) \skewprod_{\varphi}  X$ with the non-Abelian skewing group $G(2)= \GL(2, {\mathbb Z}/2{\mathbb Z})$
 in relation to distribution properties modulo $2$  for convergents of continued fraction expansions, inspired by the work~\cite{JL:88}, which handles the  case of a random real  number.

Let $x=(x_n)_{n\in\Z} \in X$. Consider the  real number in the unit interval $[0,1]$ that admits $(x_n)_{n\geq 1}$  as its sequence of partial quotients and let $(p_n(x)/q_n(x))_{n\in\N}$ stand for the  associated sequence of  rational approximations.     By Theorem~\ref{t:dendric-ergodic}, the skew product $G(2) \skewprod_{\varphi} X$ is uniquely ergodic. We can thus deduce, as detailed in Example~\ref{ex:convergents}, the following distribution results in the group $G(2)$, as a  counterpart of \cite[Theorem~3.11]{JL:88} which holds for a.e.\ real number in $[0,1]$  (see also \cite{Szusz,Moeckel,Borda2025} for related works).  For every $k=1,2$ and for every $x \in X$
    \begin{gather*}
        \lim_{N\to\infty}\frac{1}{N}  \card\{ n: 1\leq  n  \leq N,\ q_n(x) \equiv 0 \bmod 2\}= \frac{1}{3},\\
        \lim_{N\to\infty}\frac{1}{N}  \card\{ n: 1\leq  n  \leq N,\ q_n(x) \equiv 1 \bmod 2\}= \frac{2}{3}.
    \end{gather*}
  
 In fact, the distribution in the group $G(2)$ of the sequence of continued fraction convergents $(p_n/q_n)_{n\in\N}$ whose sequence of partial quotients  is given by elements of the  Fibonacci shift $X$ behaves like that of a random irrational number. Note that we recover the well-known fact that certain residue classes are  attained more frequently than others
 (as already noted at the end of Example~\ref{ex:convergents}, the entries of    equidistributed  matrices in $G(2)$ are not equidistributed modulo $2$).  
 
  This statement can be considered as a  modulo $m$ counterpart of    L\'evy's theorem  stating   that  $\lim_{n\to\infty} \frac{\log q_n}{n}= \frac{\pi^2}{ 12 \log 2}$ a.e. 
    It is therefore  natural to ask    how the convergents   behave without taking them modulo $m$, in terms of convergence of  $ \frac{\log q_n}{n}$.
    Consider   the   cocyle map $\psi\colon\{0,1\}^* \rightarrow  \GL(2, {\mathbb R})$  defined similarly  as $\varphi$,  but now with the matrix $\psi(k)$ being     considered  in  $ \GL(2, 
    {\mathbb R})$ (without reduction modulo $2$). By \cite[Theorem 3]{Furman},  there exists  some constant $\Lambda_X>0$ such that $\lim_{n\to\infty} \frac1n \log \Vert \psi^{(n)} (x)\Vert=\Lambda_X$  for all $x \in X$, which yields  the  existence of  $\lim_{n\to\infty} \frac1n \log q_n (x)=\Lambda_X$  for all $x \in X$. In other words, the cocycle $\psi$ is uniform, with the terminology of \cite{Furman}. We use here the fact that  $X$ is minimal, uniquely ergodic and the cocycle map $\psi$ is such that  the entries of $\psi^{(2)}$ are positive. 
  
   Lastly, observe that similar results can be obtained for   higher-dimensional continued fractions   via skew products   defined   with  primitive dendric  shifts on  larger alphabets, such as codings of  interval exchanges; consider for instance  the Jacobi--Perron algorithm whose equidistribution properties modulo $m$ are studied in  \cite{BNN:2006} for random numbers.
\end{example}

\section*{Acknowledgements}

    We would first like to thank the anonymous referees for carefully reading through our paper and for making very helpful comments and suggestions.
    We warmly thank Samuel Petite for leading us to a mistake in an early version of this paper and France Gheeraert for suggesting a simplification in the proof of Theorem~\ref{t:minimalskew}. 
    We also thank Olivier Carton and Vincent Delecroix for insightful conversations about the notion of pointwise densities.

    The first author is supported by Agence Nationale de la Recherche through the project SymDynAr (ANR-23-CE40-0024) and 2024 ERC Synergy Project DynAMiCs (101167561). The second author was partially supported by the CTU Global Postdoc Fellowship program. The third author is supported by the KIAS Individual Grant (MG094701) at Korea Institute for Advanced Study. The first and fourth authors were supported by the Agence Nationale de la Recherche through the project IZES (ANR-22-CE40-0011).

\bibliographystyle{abbrvnat}
\bibliography{probasProfinis}

\end{document}